\documentclass[12pt]{amsart}

\usepackage{amssymb, amsmath, amsthm,dsfont,amscd}
\usepackage{tikz}
\usepackage{url}
\usepackage{fullpage}

\usepackage{setspace}
\newtheorem{theo}{Theorem}
\newtheorem{theorem}{Theorem}[section]
\newtheorem{lemma}[theorem]{Lemma}
\newtheorem{proposition}[theorem]{Proposition}
\newtheorem{corollary}[theorem]{Corollary}
\newtheorem{coroll}[theo]{Corollary}
\theoremstyle{definition}
\newtheorem{definition}[theorem]{Definition}
\newtheorem{remark}[theorem]{Remark}

\newtheorem*{theorem*}{Theorem}

\numberwithin{equation}{section}
\numberwithin{figure}{section}

\newcommand{\ho}{\mathcal{H}}

\newcommand{\A}{\alpha}
\newcommand{\M}{\mathcal{M}}
\newcommand{\TT}{\mathcal{T}}

\newcommand{\B}{\beta}

\newcommand{\N}{\mathds{N}}
\newcommand{\Z}{\mathds{Z}}

\newcommand{\I}{\mathcal{I}}

\newcommand{\R}{\mathds{R}}
\newcommand{\C}{\mathcal{C}}
\newcommand{\CC}{\gamma}
\newcommand{\CCC}{\Gamma}

\newcommand{\Si}{\sigma}
\newcommand{\E}{\epsilon}

\newcommand{\T}{\tau}

\newcommand{\alk}[2]{\mathrm{Lk}_{\uparrow}(#1,#2)}
\newcommand{\dlk}[2]{\mathrm{Lk}_{\downarrow}(#1,#2)}

\newcommand{\hsmH}{\mathcal{H}}

\begin{document}

\title{Geometry of Houghton's groups}


\author[SR.~Lee]{Sang Rae Lee}
\address{Dept.\ of Mathematics\\
      Technion - Israel Institute of Technology\\
	Haifa, 32000
Israel}
\email{srlee@tx.technion.ac.il}

\date{\today}

\begin{abstract}

Hougthon's groups $\{\ho_n\}$ is a family of groups where each $\ho_n$ consists of `translations at infinity' on $n$ rays of discrete points emanating from the origin on the plane. Brown shows $\ho_n$ has type $F\!P_{n-1}$ but not $F\!P_{n}$ by constructing infinite dimensional cell complex on which $\ho_n$ acts with certain conditions. We modify his idea to construct $n$-dimensional CAT(0) cubical complex $X_n$ on which $\ho_n$ acts with the same conditions as before. Brown also shows $\ho_n$ is finitely presented provided $n \geq 3$. Johnson provides a finite presentation for $\ho_3$. We extend his result to provide finite presentations of $\ho_n$ for $n>3$. We also establish exponential isoperimetric inequalities of $\ho_n$ for $n\geq3$. 
\end{abstract}

\maketitle

\section{Intoduction}

\indent One of major paradigms in geometric group theory is the idea that one can understand algebraic properties of groups by studying their actions on geometric spaces. There is a basic geometric object that one can associate to any finitely generated group: its Cayley graph with the word metric.

Every pair $(G, S)$ of a group and its finite generating set $S$ has an associated $word\; metric$. The distance $d(g,h)$ between elements $g, h \in G$ is defined to be the length of the shortest word in $S \cup S^{-1}$ which is equal to $g^{-1}h$ in $G$. The $Cayley\; graph$ $\Gamma(G,S)$ of a pair $(G,S)$ has $G$ as its vertex set, and a vertex $g_1$ is connected to another vertex $g_2$ by an edge labeled by $s$ if $g_2 = g_1 s$ for some $s \in S$. A word metric on a pair $(G,S)$ extends to a metric on $\Gamma(G,S)$ provided one declares each edge in $\Gamma(G,S)$ has length $1$. The action of $G$ on itself by left multiplication extends naturally to an isometric action on the Cayley graph (with fixed generating set $S$).

We can use a Cayley graph $\CCC(G,S)$ to define geometric properties of a group $G$. One such property is the number of ends $e(G)$ of $G$. Roughly speaking, the number of ends $e(G)$ of a finitely generated group $G$ measures the number of ``connected components of $G$ at infinity". For a finitely generated group $G$, $e(G)$ is determined by the number of $ends$ $e(\CCC)$, of a Cayley graph $\CCC$ of $G$, which can be defined as follows. To find $e(\Gamma)$, remove a compact set $K$ from $\Gamma$, and count the number of unbounded components of $\Gamma - K$. The number $e(\CCC)$ is defined to be the supremum of this number over all compact sets. Finally it can be shown that $e(\CCC)$ is independent of the choice of a finite generating set $S$ used in the construction of $\CCC$ (\cite{Stallings2}). So it makes sense to define $e(G) = e(\CCC)$.


In \cite{Stallings}, Stallings showed that, for a finitely generated group $G$, the geometric condition of having more than one end is equivalent to the algebraic condition that $G$ $splits$ (that is, $G$ can be written as a free product with amalgamation or HNN extension) over a finite subgroup. On the other hand, the theory of Bass-Serre relates the algebraic condition that $G$ splits to certain types of actions of $G$ on a tree. In view of Bass-Serre theory, Stallings' theorem can be restated as follows.

\begin{theorem}\label{Stallings} A finitely generated group $G$  satisfies $ e(G) > 1$ if and only if $G$ acts on a tree (without inversion) with finite edge-stabilizers.
\end{theorem}

There is a generalization of Theorem \ref{Stallings}, which replaces the tree by a possibly higher dimensional space and uses a more general notion of ends than $e(G)$. In \cite{Houghton_2}, C. H. Houghton introduced the concept of the number of ends $e(G,H)$ of a finitely generated group with respect to a subgroup $H\leq G$. Being a subgroup of $G$, $H$ acts on a Cayley graph $\Gamma$ of $G$ by the left multiplication. Now define $e(G,H)$ to be the number of ends of the quotient graph $\CCC/H$. Again, one can show that this is independent of the particular finite generating set $S$ used in the construction of $\CCC$ (\cite{Scott}).  A group $G$ is called \textit{multi-ended} if $e(G,H)>1$ for some subgroup $H$.

Roughly speaking, a piecewise Euclidean (PE) cubical complex is build from a collection of a regular Euclidean cubes by glueing their faces via isometries. A $cubing$ is a $1$-connected PE cubical complex satisfying some additional non-positive curvature conditions. We will make this notion precise in Section \ref{section_CC}.

One dimensional cubes are just unit length line segments, and so a $1$-dimensional cubical complex is simply a graph. The condition of being simply connected means that the graph is a tree. Therefore one can think of cubings as generalizations of trees.

In \cite{Sageev}, Sageev proved the following remarkable generalization of Stallings' result.

\begin{theorem}\label{Sageev} A finitely generated group $G$ is multi-ended if and only if $G$ acts `essentially' on a cubing.
\end{theorem}

This theorem includes the possibility that the cubing is infinite dimensional. An essential action means that the action has an unbounded orbit provided the cubing is finite dimensional.

Non-positively curved cubical complexes play a central role in low-dimensional topology and geometric group theory. A striking example of
their importance is given by Agol's recent proof \cite{Agol} of the Virtual Haken Conjecture and the Virtual Fibration Conjecture in $3$--manifold
topology. Agol's proof relies on results of Haglund and Wise which concern fundamental groups of a $speical$ class of non-positively curved cubical complexes.
In \cite{Haglund-Wise}, Haglund and Wise showed if the fundamental group of a special cubical complex is word-hyperbolic then every quasiconvex subgroup is separable.



This thesis explores geometric properties of a particular class of groups, termed Houghton's groups, introduced in \cite{Houghton}. Roughly speaking, Houghton's group $\ho_n$ ($n \in \N$) is the group of permutations of $n$ rays of discrete points which are eventual translations (each permutation acts as a translation along each ray outside a finite set). See Section \ref{definition_houghton_group} for details.

There are $n$ canonical copies of $\ho_{n-1}$ inside $\ho_n$, and $\ho_n$ is multi-ended with respect to each of them. The $i^{th}$ subgroup, $1\leq i\leq n$, is obtained by restricting to permutations which fix $i^{th}$ ray pointwise. One of the main results of this thesis is to produce an action of $\ho_n$ on a $n$-dimensional cubing $X_n$. Note that depending on subgroups which are taken into account, there are various cubings on which $\ho_n$ acts. One feature of our cubing is that $X_n$ encodes all of those subgroups $\ho_{n-1}$ at once. See Section \ref{sec_Sageev_construction} for a detailed description.

\begin{theo} \label{thm_action_on_X_n}
For each integer $n\!\geq\! 1$, there exists a $n$-dimensional cubing $X_n$ and a Morse function $h:X_n\rightarrow \mathds{R}_{\geq0}$ such that $\mathcal{H}_n$ acts on $X_n$ properly (but not cocompactly) by height-preserving semi-simple isometries. Furthermore, for each $r \in \mathds{R}_{\geq0}$ the action of $\mathcal{H}_n$ restricted to the level set $h^{\!-1}\!(r)$ is cocompact.
\end{theo}

An additional feature of $X_n$ is that it comes equipped with a height function (Morse function) to the non-negative reals. The group $\ho_n$ acts as a height-preserving fashion,  and the quotient of any level set by $\ho_n$ is cocompact. As an application we recover Brown's results for finiteness properties of $\ho_n$.

\begin{coroll}\label{coroll_fin_prop_H_n}
For $n\!\geq\!2$, $\mathcal{H}_n$ is of type $F\!P_{\!{n-1}}$ but not $F\!P_{\!n}$, it is finitely presented for $n \geq 3$.
\end{coroll}

Knowing that $\ho_n$ is finitely presented for $n \geq 3$ prompts a natural question. What are explicit presentations for the $\ho_n$? In \cite{johnson}, Johnson answered this for $\ho_3$. Another main result of this thesis provides explicit presentations for all $\ho_n$ ($ n \geq 3$).

\begin{theo}\label{thm_fin_presen_Hough}
For $n\geq 3$, $\mathcal{H}_n$ is generated by $ g_1, \cdots, g_{n-1}, \alpha$ with relators
\begin{equation*}
    \alpha^2\!=\!1, \,(\alpha \alpha^{g_1}\!)^3\! =\!1,\, [\alpha , \alpha^{g_1^2}]\!=\!1, \alpha\!=\![g_i,g_j] ,\;\,\alpha^{g_i^{-1}}\! =\! \alpha^{g_i^{-1}}\;\text{for}\; 1\leq i\neq j \leq n-1.
\end{equation*}
\end{theo}

Finally we determine bounds for the Dehn functions of $\ho_n$. We give a formal definition of Dehn functions in Section $4$. Intuitively, given a finite presentation $P = \langle A | R \rangle$ for a group $G$, and given a word $w$ representing the identity in $G$ with $|w| \leq x$, the Dehn function of a presentation $P$, $\delta_{P}(x)$, measures the least upper bound on the number of relations in term of $x$, which one must apply to check $w=1$. Although the function $\delta_{P}(x)$ depends on the presentation, the growth type of this function is independent of choice of a finite presentation for $G$. The $Dehn\; function$ of a finitely presented group $G$ is defined to be the growth type of $\delta_{P}(x)$. See Section \ref{section_Dehn_function} for details. An isoperimetric function for a group is an upper bound of the Dehn function. In Section $4$ we establish isoperimetric inequalities for $\ho_n$ for $n\geq3$.

\begin{theo}\label{thm:exp_bd_n}
For any $n\geq 3$ the Dehn function $\delta_{\mathcal{H}_n}(x)$ of $\mathcal{H}_n$ satisfies
$$
\delta_{\mathcal{H}_n}(x) \preccurlyeq e^x.
$$
\end{theo}


\section{Houghton's Groups $\ho_n$}

\subsection{Definition of $\ho_n$}\label{definition_houghton_group}
Fix an integer $n \geq 1$. Let $\N$ be the positive integers. For each $k$, $1\leq k\leq n$, let
$$
R_k = \{m e^{i\theta}: m \in \N, \,\theta= {\pi/2 + 2\pi(k-1)}/n\}\subset \mathds{C}.
$$
and let $Y_n = \bigcup_{k=1}^n R_k$ be the disjoint union of $n$ copies of $\N$, each arranged along a ray emanating from the origin in the plane. We shall use the notation $\{1, \cdots, n\} \times \N$ for $Y_n$, letting $(k,p)$ denote the point of $R_k$ with distance $p$ from the origin. A bijection $g:Y_n \to Y_n$ is called an \emph{eventual translation} if it acts as a translation on each $R_k$ outside a finite set. More precisely $g$ is an eventual translation if the following holds:
\begin{itemize}
    \item [$\star$] There is an $n$-tuple $(m_1 ,\cdots, m_n) \in \mathds{Z}^n$ and a finite set $K_g \subset Y_n$ such that $ (k,p)\cdot g=(k, p+m_k)$ for all $(k,p) \in Y_n - K_g$.
\end{itemize}

\begin{definition}For an integer $n\geq1$, Houghton's group $\ho_n$ is defined to be the group of all eventual translations of $Y_n$.
\end{definition}

As indicated in the definition $\star$, Houghton's group $\ho_n$ acts on $Y_n$ on the right; thus $gh$ denotes $g$ followed by $h$ for $g,h \in \ho_n$. For notational convenience we denote $g^{-1}$ by $\overline{g}$.

Let $g_i$ be the translation on the ray of $R_1 \cup R_{i+1}$, from $R_1$ to $R_{i+1}$ by $1$ for $1 \leq i \leq n-1$. More precisely, $g_i$ is defined by
\begin{equation}\label{eq_generator_g_i}
    (j,p)g_i =
    \begin{cases}
        (1,p-1) & \text{if } j=1 \;\text{and } p \geq 2,\\
        (i+1,1) & \text{if } (j,p)=(1, 1)\\
        (i+1,p+1) & \text{if }j=i+1,\\
        (j,p)& \text{otherwise}.
    \end{cases}
\end{equation}

Figure \ref{examples} illustrates some examples of elements of $\ho_n$, where points which do not involve arrows are meant to be fixed, and points of each finite set $K$ are indicated by circles. Finite sets $K_{g_i}$ and $K_{g_j}$ are singleton sets. After simple computation, one can check that the commutator of $g_i$ and $g_j$ ($ i \neq j$) is the transposition exchanging $(1,1)$ and $(1,2)$. The last element $g$ is rather generic and $K_g$ consists of eight points (circles).

\begin{figure}[ht]
\includegraphics[width=1.0\textwidth]{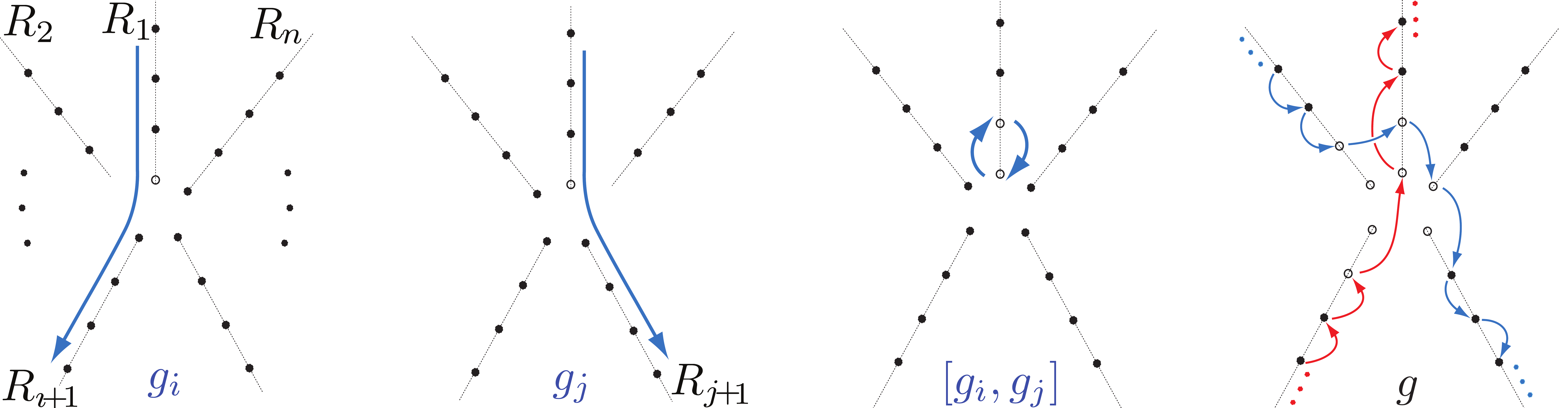}
\caption{Some elements of $\ho_n$} \label{examples}
\end{figure}

\subsection{Abelianization of $\ho_n$} \label{sec_abel}

Assigning an $n$-tuple $(m_1 ,\cdots, m_n) \in \mathds{Z}^n$ to each $g \in \ho_n$ defines a homomorphism $\varphi :\hsmH_n \to \mathds{Z}^n$. Let $\Sigma_{n,\infty}$ denote the infinite symmetric group consisting of all permutations of $Y_n$ with finite support.

\begin{lemma}
For $n\in \N$, $\Sigma_{n,\infty} \leq \ho_n$.
\end{lemma}

\begin{proof}Every element $f \in \Sigma_{n,\infty}$ has a finite support. This means there exists a finite set $K \subset Y_n$ such that $f$ acts on each $R_i \subset Y_n$ as the identity outside $K$. So $f \in \ho_n$.
\end{proof}

\begin{lemma}\label{lm_Si=ker}
For $n\geq 3$, $ker\varphi = \Sigma_{n,\infty} = [\ho_n , \ho_n]$.
\end{lemma}

\begin{proof}Fix $n \geq 3$. We show the equalities by verifying a chain of inclusions. Each element $g \in ker\varphi$ acts on $Y_n$ as the identity outside a finite set $K \subset Y_n$. This means that $g$ is a permutation on $K$, and so $g \in \Sigma_{n,\infty}$. The claim in the proof of Lemma \ref{lm_fin_gen_H_n} implies that every element with finite support is a product of conjugations of $[g_i, g_j]$ ($i \neq j$). Suppose $g \in [\ho_n , \ho_n]$. Lemma \ref{lm_fin_gen_H_n} says $\ho_n$ is generated by $g_1, \cdots, g_{n-1}$ if $n\geq 3$. Take an expression for $g$ in letters $g_1, \cdots, g_{n-1}$. Since $g \in [\ho_n , \ho_n]$, the sum of powers of $g_i$'s in this expression is $0$ for each $ i= 1, \cdots, n-1$. So $g \in ker\varphi$.
\end{proof}

In Section \ref{sec_presentation}, we shall see that $\ho_n$ is generated by $g_1, \cdots, g_{n-1}$ defined in equation (\ref{eq_generator_g_i}). Note that $\varphi(g_i) \in \mathds{Z}^n$ has only two nonzero values $-1$, $1$, and
$$
\varphi(g_i)=(-1, 0, \cdots,0, 1 , 0,\cdots, 0)
$$where $1$ occurs in $(i+1)^{th}$ component. The image of $\ho_n$ is generated by those elements and $\varphi(\ho_n)=\{(m_1 ,\cdots, m_n)|\sum m_i =0\}$ is free abelian of rank $n-1$. 
\begin{corollary}\label{coro_abelianization}For $n\geq 3$, the ablelianization of $\ho_n$ is given by the following short exact sequence.
\begin{equation}\label{eq_abelianization_1}
1 \to [\ho_n,\ho_n] \to \hsmH_n \xrightarrow{\varphi} \mathds{Z}^{n-1} \to 1.
\end{equation}
\end{corollary}

The above result was originally shown by C. H. Houghton in \cite{Houghton}. Note that the group $\ho_n$ has the property that the rank of the abelianization is one less than the number of rays of the space $Y_n$ (namely, the number of ends of $Y_n$ with respect to the action) on which $\ho_n$ acts transitively. This was the main reason that Houghton introduced and studied a family of groups $\{\ho_n\}_{n \in \N}$ in the same paper.

\begin{remark}For $n\geq 3$, by replacing the commutator subgroup by $\Sigma_{n, \infty}$ in the short exact sequence (\ref{eq_abelianization_1}), we have the following short exact sequence.
\begin{equation}\label{eq_abelianization_Si}
1 \to \Sigma_{n,\infty} \to \hsmH_n \xrightarrow{\varphi} \mathds{Z}^{n-1} \to 1,
\end{equation} The embedding $\Sigma_{n,\infty} \hookrightarrow \ho_n$ is crucial for our study in two directions. Firstly, it gives rise to finite presentation for $\ho_n$ for $n\geq 3$ (Section \ref{sec_presentation}). Secondly, this inclusion is used to construct exponential isoperimetric inequalities for $\ho_n$ $n\geq 3$ (Section \ref{section_Exp_isoperi_ineq_ho}).
\end{remark}

We suppress $n$ in $\Sigma_{n, \infty}$ when the underlying set $Y_n$ is clear from the context.

\subsection{Finite Presentations for $\ho_n$ ($n\geq3$)} \label{sec_presentation}
Note that every element of $\ho_1$ acts on $Y_1 = \N$ by the identity outside a finite set. If the element acted as a non-trivial translation outside of a finite set, then it would fail to be a bijection of $Y_1$. So $\ho_1 \subset \Sigma_{1,\infty}$. The converse inclusion is obvious, therefore $\ho_1$ is the infinite symmetric group $\Sigma_{1, \infty}$ of all permutations of $\N$ with finite support. By considering elements with successively larger supports, one can argue that $\Sigma_{1,\infty}$ is not finitely generated. It is known \cite{Brown} that $\ho_2$ is finitely generated but not finitely presented, and that $\ho_n$ is finitely presented for $n \geq 3$. Brown showed more than just finite presentedenss, and in Section \ref{sec_finite_prop}, we will give a precise statement and a new proof of Brown's result about the finiteness properties of $\ho_n$ (Corollary \ref{coroll_fin_prop_H_n}). In this section we will provide explicit finite presentations for $\ho_n$ ($n \geq 3$).

\medskip
\noindent
{\bf Finite generating sets for $\ho_n$ ($n\geq2$).}\\ We begin by proving that $\ho_2$ is generated by two elements. Many of the techniques in this argument extend to the $\ho_n$ for $n \geq 3$.
Consider the two elements $g_1, \B  \in \ho_2$ where $g_1$ is defined by equation (\ref{eq_generator_g_i}) and $\B$ is the transposition exchanging $(1,1)$ and $(2,1)$ as depicted in Figure \ref{ho_2}.

\begin{figure}[ht]
\begin{center}
\includegraphics[width=1.0\textwidth]{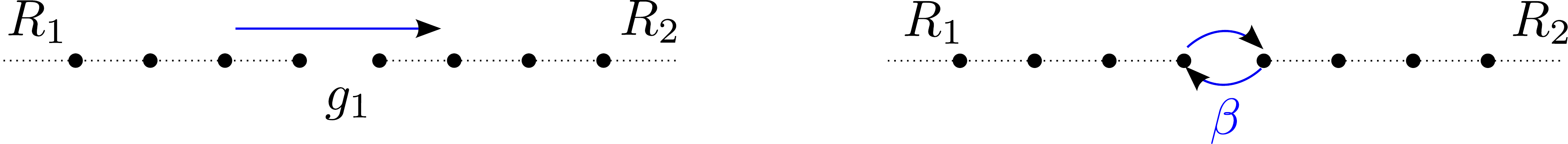}
\caption{$\ho_2$ is generated by $g_1$ and $\B$} \label{ho_2}
\end{center}
\end{figure}

Suppose $g  \in\ho_2$. Since $\varphi (g_1)$ generates the image $\varphi(\ho_2) \cong \mathds{Z}$ there exists $k \in \mathds{Z}$ such that the given $g$ and $g_1^k$ agree on $Y_2 -K$ for some finite set $K \subset Y_2$. So we have
$$
\overline{g}_1^{k} g= f\;\; \text {or}\;\; g= g_1^k f
$$
where $f \in Perm(K)$.
One can take suitable conjugates $\overline{g}_1^{k'} \B g_1^{k'}$ to express transpositions exchanging two consecutive points of $Y_2$. For example, $\overline{g}_1^{3} \B g_1^{3}$ is the transposition swapping $(2,3)$ and $(2,4)$. Since any permutation of a finite set $K$ can be written as a product of those transpositions, $f$ can be written as a product of conjugations of $\B$. So $g= g_1^k f \in \ho_2$.

\begin{remark}\label{rmk_conjugation}
The conjugation $\overline{h}gh$ is denoted by $g^h$. Note that $(i,p)g=(j,q)$ is equivalent to $(i,p)h g^{h}=(j,q)h$. The following observation is quite elementary but useful: If $\B$ is a transposition exchanging two points $(i,p)$ and $(j,q)$ then $\B^{h}$ is the transposition exchanging $(i,p)h$ and $(j,q)h$.
\end{remark}

We have seen that a translation and a transposition are sufficient to generate all of $\ho_2$. Moreover, we have seen from the third example of Figure \ref{examples} that $[g_i, g_j]$ is a transposition on $Y_n$ for $1 \leq i \neq j \leq n-1$. So the following lemma should not be a surprise.

\begin{lemma}\label{lm_fin_gen_H_n}
$\hsmH_n$ is generated by $g_1, \cdots, g_{n-1}$ for $n \geq 3$ .
\end{lemma}
\begin{proof}Let $g \in \hsmH_n$ and $\varphi(g)= (m_1 ,\cdots, m_n)$. Note that $m_1 = -(m_2 +\cdots +m_n)$. Consider the element $g' \in \hsmH_n$ given by $$
g'= g_1^{m_2}g_2^{m_3} \cdots g_{n-1}^{m_n}.
$$
By construction, $\varphi(g') = \varphi(g)$, and so two element $g'$ and $g$ agree on $Y_n - K$ for some finite set $K$. This means there exists $f \in Perm(K)$ such that $g=g'f$. So it suffices to show that every finite permutation can be written as a product of the $g_i^{\pm1}$, $i=1, \cdots,n-1$.

\textbf{Claim:} Any transposition of $Y_n$ is a conjugation of $[g_1, g_2]= g_1 g_2 \overline{g}_1 \overline{g}_2$ by a product of $g_i$'s. First we express transpositions exchanging points of $R_1$ in this manner. We already saw that $\A=[g_1, g_2]$ is the transposition exchanging $(1,1)$ and $(1,2)$.

Consider the following elements $h_1, h_2\in \hsmH_n$:
$$
h_1=\A^{g_1 \overline{g}_2^{(p-2)}\overline{g}_1} \; \text{and}\; h_2=\A^{g_1 \overline{g}_2^{(q-2)}\overline{g}_1}
$$
where $p, q \geq 2$ are integers. By the observation in Remark \ref{rmk_conjugation}, $h_1$ is the transposition exchanging $(1,1)$ and $(1,p)$. By the same reason, one sees $h_2$ is the transposition swapping $(1,1)$ and $(1,q)$. Therefore the permutation exchanging $(1,p)$ and $(1,q)$ can be expresses by the conjugation $h_1^{h_2}$. This allows one to express transpositions of $R_{i+1}$ for all $i$ by conjugations $h_1^{h_2 g_i^{m}}$ for some $m \in \mathds{N}$.

\begin{figure}[ht]
\begin{center}
\includegraphics[width=0.8\textwidth]{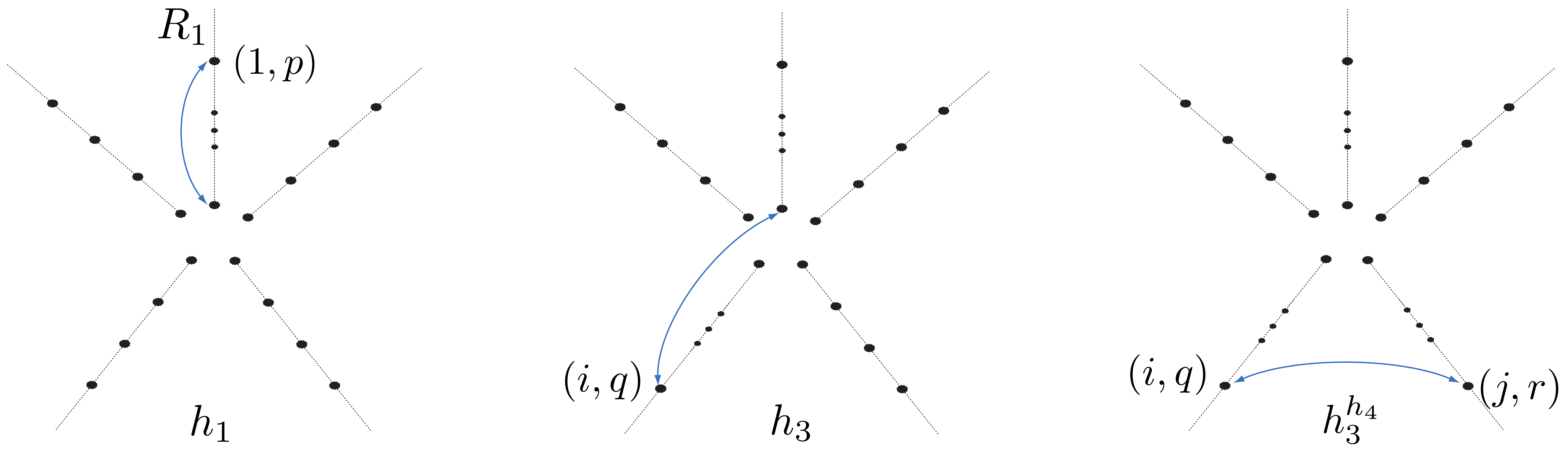}
\caption{Transpositions are conjugations of $[g_i, g_j]$} \label{trans}
\end{center}
\end{figure}

Next we show that transpositions exchanging points of different rays can be expressed in a similar fashion. Set $2\leq i \neq j \leq n$ and $q,r \in \mathds{N}$. The permutations $h_3, h_4$ given by
$$
h_3= \A^{g_j \,g_i^{q}\,\overline{g}_j}\; \text{and}\; h_4=\A^{g_i\, g_j^{r}\,\overline{g}_i}
$$
are the transpositions swapping $(1,1)$ and $(i,q)$ and swapping $(1,1)$ and $(j,r)$ respectively. Thus the conjugation $h_1^{h_3}$ represents the transposition exchanging $(1,p)$ and $(i,q)$. Similarly the conjugation $h_3^{h_4}$ represents the transposition swapping $(i,q)$ and $(j,r)$. Figure \ref{trans} illustrates various transpositions of $Y_n$ given by conjugations of $\A$. Our claim follows since $i, j, p,q$ and $r$ are arbitrary. So we are done because any finite permutation $f \in Perm(K)$ above can be written as a product of transpositions.
\end{proof}

\medskip
\noindent
{\bf Finite presentation for $\ho_3$.}\\ D. L. Johnson (\cite{johnson}) provided a finite presentation for $\ho_3$.
\begin{theorem}\label{thm_johnson}
The Houghton's group $\hsmH_3$ is isomorphic to $H_3$ presented by
\begin{equation}\label{eq_abstract_H_3}
  H_3 =\langle g_1, g_{2},\A\,| \,r_1, \cdots, r_5 \rangle,
\end{equation}
\begin{tabular}{ l l l }
 $r_1$&:&  $\A^2=1$,\\
 $r_2$&:&  $(\A \,\A^{\overline{g}_1})^3 =1$ ,\\
$r_3$&:&  $[\A , \A^{\overline{g}_1^2}]=1$,\\
$r_4$&:&  $[g_1, g_2] =\A$,\\
$r_5$&:& $\A^{\overline{g}_1} = \A^{\overline{g}_2}$.
\end{tabular}
\end{theorem}
Note that there is a map $\psi$ from an abstract group $H_3$ to $\ho_3$ such that
$\psi(g_i)$ is the translation $g_i$ defined in equation (\ref{eq_generator_g_i}) ($i=1,2$) and $\psi(\A)$ is the transposition $\A = [g_1, g_2]$. Observe that relations $r_1, r_2$ and $r_3$ are reminiscent relations of Coxeter systems for finite symmetric groups. It is easy to check that $\ho_3$ satisfies those five relations. Johnson showed that the group $H_3$ fits into the short exact sequence of (\ref{eq_abelianization_Si}) with the normal subgroup generated by $\A$, and that $\psi: H_3 \to \ho_3$ induces isomorphisms between the normal subgroups and the quotients. So the following diagram commute. The Five Lemma completes the proof.

\[\begin{CD}\label{CD_five_lemma}
       1@>>> \langle \langle\,\A\,\rangle \rangle @>>>  H_3 @>>>  H_3 / { \langle\langle\,\A\,\rangle \rangle}@>>>  1 \\
      @VVV @VVV @VVV @VVV @VVV\\
      1 @>>>  [\ho_3 , \ho_3] @>>> \ho_3 @>>> \mathds{Z}^{2} @>>>  1.
\end{CD}\]

\medskip
\noindent
{\bf Finite presentations for $\ho_n$ ($n\geq3$).}\\We extend the previous argument to find finite presentations for $\ho_n$ for all $n\geq3$.
The first step is to establish appropriate presentations for symmetric groups on finite balls of $Y_n$ by applying Tietze transformation to the Coxeter systems for finite symmetric groups.

Fix positive integers $n,r \in  \N$. Let $B_{n,r}$ denote the ball of $Y_n$ centered at the origin with radius $r$, i.e., $B_{n,r} =\{(i,p) \in Y_n :1\leq i \leq n, 1\leq p\leq r\}$. The ball $B_{n,r}$ contains $nr$ points which can be identified to the points of $\{1,2,\cdots, r, \cdots, 2r, \cdots,  nr\} \subset \N$ via $\chi:B_{n,r} \to \N$ defined by
$$
\chi(i,p)= (i-1)r +p.
$$
Recall the following Coxeter system for the symmetric group $S_{nr}$ on $\{1,2, \cdots, nr\}$
\begin{equation}\label{coxeter_sys}
   S_{nr} = \langle \sigma_1, \cdots, \sigma_{nr-1}\,|\, \sigma_k^2,\,(\sigma_k \,\sigma_{k+1})^3, \,[\sigma_k ,\sigma_{k'}] \;\text{for}\; |k-k'|\geq 2 \rangle
\end{equation}

\begin{remark}\label{rmk_coxeter_rel}
One can interpret the above three types of relators respectively as follows;
\begin{itemize}
    \item generators are involutions,
    \item two generators with overlapping support satisfy braid relations,
    \item two generators commute if they have disjoint supports.
\end{itemize}

\noindent Note that every collection of transpositions of $S_{nr}$ satisfies the three types of relations. We shall see in Theorem \ref{thm_iso_presen_symm} that a `reasonable' set of transpositions satisfying the above relations provides a presentation for the symmetric group $S_{nr}$.
\end{remark}

The arrangement of $Y_n$ in the plane is different than the arrangement of $nr$ points $\{1,2, \cdots, nr\}$ in $\mathds{N}$. In particular, the adjacency is different. By a pair of two \emph{adjacent} points in $Y_n$, we mean either a pair of two consecutive points in a ray of $Y_n$ or a pair of $(1,1)$ and $(j+1,1)$ for $1\!\leq \!j\!\leq \!n-1$. We aim to produce a new presentation for the symmetric group $\Sigma_{n,r}$ on $B_{n,r}$ with a generating set which consists of the transpositions exchanging all pairs of adjacent points in $Y_n$. More precisely we need the following $nr\!-\!1$ transpositions
\begin{itemize}
    \item[] $\A^i_p$, exchanging $(i,\,p)$ and $(i,\,p\!+\!1)$ for $1\!\leq \!i\!\leq \!n$ and $1\!\leq \!p \!\leq \!r\!-\!1$
    \item[]$\A^{j+1}_0$, exchanging $(1,1)$ and $(j+1,1)$ for $1\!\leq \!j\!\leq \!n-1$.
\end{itemize}

For an element $\Si \in S_{nr}$ consider the element of $Perm(B_{n,r})$ given by
$$
(\Si)^{\chi^{-1}} = \chi(\Si)\chi^{-1}
$$
which stands for the composition
$ B_{n,r} \xrightarrow {\chi} \{1, \cdots, nr\} \xrightarrow{\Si}
\{1, \cdots, nr\} \xrightarrow{\chi^{-1}} B_{n,r}$. Let us denote the induced element by $\chi^\ast(\Si)$. New generators of the first type are defined by $\chi^\ast$ as follows.

\begin{equation}\label{eq_transform_gen_symm}
\A_p^i:=\chi^\ast(\Si_{(i-1)r + p})
\end{equation} for $1\!\leq \!i\!\leq \!n$ and $1\!\leq \!p \!\leq \!r\!-\!1$. Note $\chi^{\ast}$ defines a bijection
$$
 \{\sigma_1, \,\cdots,\widehat{\sigma_r},\cdots,\widehat{\sigma_{2r}},\cdots,\widehat{\sigma_{(n-1)r}},\cdots, \sigma_{nr-1} \}\leftrightarrow \{ \A_p^i \;| \;1\!\leq \!i\!\leq \!n, \, 1\!\leq \!p \!\leq \!r\!-\!1 \},
$$
($\;\widehat{}$
 denotes deletion).

Replace $\Si_{(i-1)r + p}$ by $\A_p^i$ in the Coxeter system (\ref{coxeter_sys}), $1\!\leq \!i\!\leq \!n$ and $1\!\leq \!p \!\leq \!r\!-\!1$. Then (\ref{coxeter_sys}) becomes

\begin{equation}\label{eq_coxeter_sys_new}
S_{nr} = \langle \A_p^i, \,\Si_{jr}\,|\, R \,\rangle\;\,1\!\leq \!i\!\leq \!n\;, 1\!\leq \!p \!\leq \!r\!-\!1,\;\text{and}\; 1\!\leq \!j\!\leq \!n-1,
\end{equation}where $R$ consisting of

\begin{itemize}
    \item[]$R_1$: $(\A_p^i)^2, \,(\Si_{jr})^2$, $\forall \;i, p, j$
    \item[]$R_2$: $[\A_p^i,\A_q^i]$, \,$\forall \; i,\;|p\!-q|\geq 2$
     \item[]$R_3$: $[\A_p^i,\A_p^{i'}]$, \,$\forall \; p,q,\;i \neq i'$
     \item[]$R_4$: $[\A_p^i,\Si_{jr}]$, \,$\forall \; j,\;(i-1)r +p  \neq jr \pm 1$
     \item[] $R_5$: $[\Si_{jr},\Si_{j'r}]$, \, $j \neq j'$
    \item[]$R_6$: $(\A_p^i\,\A_{p+1}^i)^3$, \;$\forall \; i,\;1\leq p \leq r-2$
     \item[] $R_7$: $(\A_p^i\,\Si_{jr})^3 $, \,$\forall \; j,\;(i-1)r +p  = jr \pm 1$.
\end{itemize}

Note that the difference in presentations (\ref{coxeter_sys}) and (\ref{eq_coxeter_sys_new}) is just notation and the set $R$ of relations exhibits the same idea of Remark \ref{rmk_coxeter_rel}. We also want to replace $\Si_{jr}$ by $\A^{j+1}_0$ for $j = 1, \cdots, n-1$ as follows. For each $j$, $\A^{j+1}_0$ corresponds, under $\chi^\ast$, to the transposition $(1\,\, jr\! +\!1)$ of $S_{nr}$. Observe that the transposition $(1\,\, jr\! +\!1)$ of $S_{nr}$ can be expressed as $(\Si_{jr})^{\overline{w_{jr-1}}}$ where $w_k = \sigma_1 \Si_2 \cdots \Si_{k}$.

\begin{figure}[ht]
\begin{center}
\includegraphics[width=1.0\textwidth]{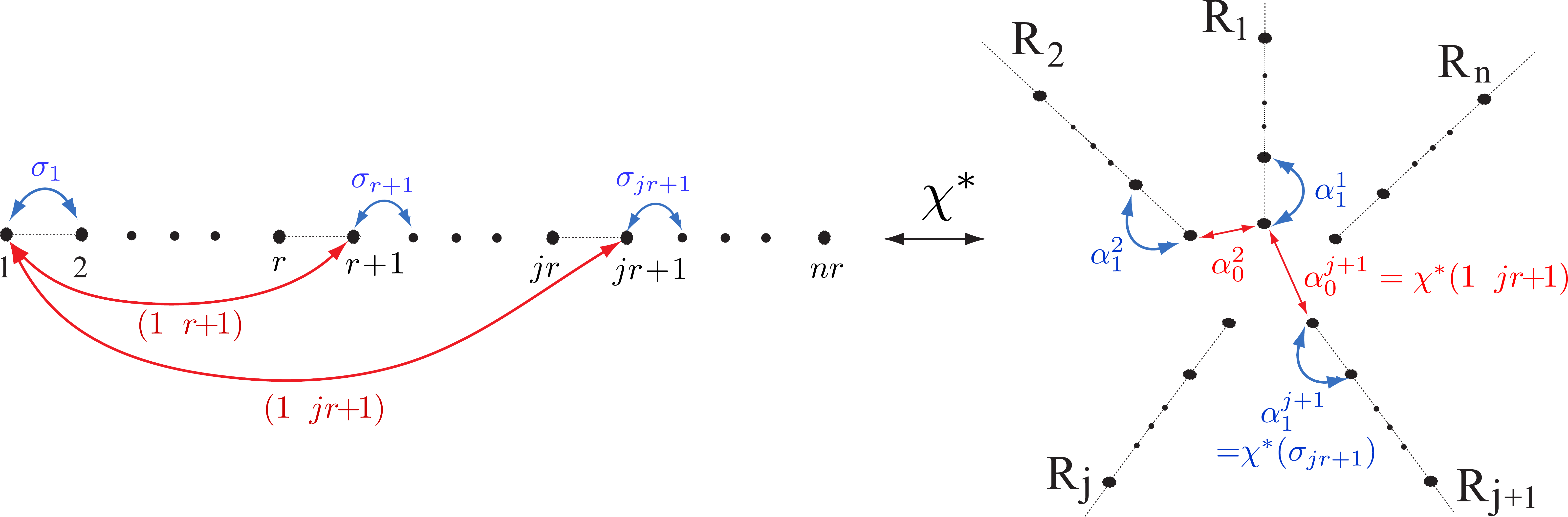}
\caption{Correspondence between two generating sets via $\chi^\ast$ } \label{fig_tietze}
\end{center}
\end{figure}

As in (\ref{eq_transform_gen_symm}), define new generators of the second type by

\begin{equation}\label{eq_transform_gen_symm_2}
\A^{j+1}_0:= \chi^{\ast}((\Si_{jr})^{\overline{w_{jr-1}}})\; \text{for}\; j = 1, \cdots, n-1.
\end{equation}
Figure \ref{fig_tietze} illustrates how one transforms the generating set of $S_{nr}$ to a new generating set.

Consider the group $\Sigma_{n,r}$ presented by
\begin{equation}\label{eq_presen_fin_symm_new}
\Sigma_{n,r} = \langle \A_p^i, \,\A^{j+1}_0\,|\, R \,\rangle\;\,1\!\leq \!i\!\leq \!n\;, 1\!\leq \!p \!\leq \!r\!-\!1,\;\text{and}\; 1\!\leq \!j\!\leq \!n-1,
\end{equation}where $R'$ consists of
\begin{itemize}
    \item[] $R'_1$: $(\A_p^i)^2, \,(\A^j_0)^2$, $\forall \;i, p, j$
    \item[] $R'_2$: $[\A_p^i,\A_q^i]$, \,$\forall \; i,\;|p\!-q|\geq 2$
    \item[]  $R'_3$: $[\A_p^i,\A_p^{i'}]$, \,$\forall \; p,q,\;i \neq i'$
    \item[] $R'_4$: $[\A_p^i,\A^{j+1}_0]$, \,$\forall \; j,\;\A_p^i \neq \A_1^1,\,\A_{1}^{j+1}$
    \item[] $R'_5$: $(\A_p^i\,\A_{p+1}^i)^3$, \;$\forall \; i,\;1\leq p \leq r-2$
     \item[] $R'_6$: $(\A_1^1\,\A^{j+1}_0)^3 $, $(\A_1^{j+1}\,\A^j_0)^3 $.
\end{itemize}

\noindent Note that $R$ and $R'$ share the same idea described in Remark \ref{rmk_coxeter_rel}. Expecting $R'$ to replace the original relators of $R$ is reasonable.

\begin{theorem}\label{thm_iso_presen_symm}
With the above definition, $\Sigma_{n,r} \cong S_{nr}$ for all integers $n, \, r \geq 1$.
\end{theorem}
\begin{proof}

First we show $ \langle \!\langle R \rangle \!\rangle  \leq \langle \!\langle R' \rangle \!\rangle$. It is clear $ \langle \!\langle R_1 \rangle \!\rangle  \leq \langle \!\langle R'_1 \rangle \!\rangle$. Basic idea is to use appropriate conjugation relations between $\Si_{jr}$ and $\A^{j+1}_0$ given by (\ref{eq_transform_gen_symm_2}).

In order to proceed we need the following identities for $j =1,2, \cdots, n-1$

\begin{equation}\label{eq_rel_1}
    [\Si_{jr}, \A^i_p] = [\A^{j+1}_0\!, \A^i_p]^{w_{jr-1}}\;\text{if}\;(i-1)r+p \geq jr+2,
\end{equation}

\begin{equation}\label{eq_rel_2}
[\Si_{jr}, \A^i_p] = [\A^{j+1}_0\!, \A^i_{p+1}]^{w_{jr-1}}\;\text{if}\;(i-1)r+p \leq jr-2,
\end{equation}

\begin{equation}\label{eq_rel_3}
[\Si_{jr}, \Si_{j'r}] = [\A^{j+1}_0\!, \A^{j'+1}_{1}]^{w_{jr-1}}\;\text{if}\;j' <j+1,
\end{equation}

\begin{equation}\label{eq_rel_4}
(\A^{j+1}_1 \Si_{jr})^3 = ((\A^{j+1}_1 \A_0^{j+1})^3)^{w_{jr-1}}
\end{equation}

\begin{equation}\label{eq_rel_5}
(\A^{j}_{r-1} \Si_{jr})^3 = ((\A^1_1 \A_0^{j+1})^3)^{w_{jr-1} w_{jr-2}}
\end{equation}

\begin{equation}\label{eq_rel_6}
(\A_1^{1})^{w_{jr-1} w_{jr-2}}= \A^{j}_{r-1}.
\end{equation}

First we apply an induction argument on $j$ to verify (\ref{eq_rel_1}), (\ref{eq_rel_2}) and (\ref{eq_rel_4}). The induction assumption for (\ref{eq_rel_1}) together with relators of $R'_3$ yields $\A_p^i = (\A^i_p)^{w_{jr-1}}$ if $(i-1)r+p \geq jr+2$. So (\ref{eq_rel_1}) follows from (\ref{eq_transform_gen_symm_2}). Similarly one can verify (\ref{eq_rel_2}) by using (\ref{eq_rel_1}) together with $R'_2$ and $R'_5$. The identity (\ref{eq_rel_4}) follows from (\ref{eq_rel_1}), $R'_2$ and $R'_3$.

For (\ref{eq_rel_3}), (\ref{eq_rel_5}) and (\ref{eq_rel_6}) we apply simultaneous induction on $j$ together with (\ref{eq_rel_1}), (\ref{eq_rel_2}) and (\ref{eq_rel_4}). The base case of (\ref{eq_rel_3}) holds trivially. The relation (\ref{eq_rel_6}) is clear when $j=1$, which establishes the base case of (\ref{eq_rel_5}). Suppose (\ref{eq_rel_3}), (\ref{eq_rel_5}) and (\ref{eq_rel_6}) hold for $j=k$. Observe that
\begin{align*}
&w_{(k+1)r-1} w_{(k+1)r-2}\\=& (\sigma_1 \cdots  \Si_{kr} \Si_{kr\!+\!1} \cdots \Si_{(k\!+\!1)r\!-\!1})(\sigma_1 \Si_2 \cdots  \Si_{kr} \Si_{kr\!+\!1} \cdots \Si_{(k\!+\!1)r\!-\!2}) \\
 =& (\sigma_1 \cdots  \Si_{kr})( \sigma_1 \Si_2 \cdots \Si_{kr\!-\!1} )(\Si_{kr\!+\!1} \cdots \Si_{(k\!+\!1)r\!-\!1})(\Si_{kr} \Si_{kr\!+\!1} \cdots \Si_{(k\!+\!1)r\!-\!2})  &(\ref{eq_rel_3}), \; R'_2,\; R'_3\\
=& (\sigma_1 \cdots  \Si_{kr\!-\!1})( \sigma_1 \Si_2 \cdots \Si_{kr\!-\!2}\Si_{kr}\Si_{kr\!-\!1} )(\Si_{kr\!+\!1} \cdots \Si_{(k\!+\!1)r\!-\!1})(\Si_{kr} \Si_{kr\!+\!1} \cdots \Si_{(k\!+\!1)r\!-\!2})  &(\ref{eq_rel_2}), \;(\ref{eq_rel_3})
\end{align*}
Thus we have
\begin{align*}
(\A_1^1)^{w_{(k\!+\!1)r\!-\!1} w_{(k\!+\!1)r\!-\!2}}&=(\A_{r-1}^k)^{(\Si_{kr}\Si_{kr\!-\!1} \Si_{kr\!+\!1} \cdots \Si_{(k\!+\!1)r\!-\!1})(\Si_{kr} \Si_{kr\!+\!1} \cdots \Si_{(k\!+\!1)r\!-\!2})} &\; (\ref{eq_rel_5})\\
 =& (\Si_{kr})^{(\Si_{kr\!+\!1} \cdots \Si_{(k\!+\!1)r\!-\!1})(\Si_{kr} \Si_{kr\!+\!1} \cdots \Si_{(k\!+\!1)r\!-\!2})} &\; (\ref{eq_rel_4})\\
=& (\Si_{kr})^{(\Si_{kr\!+\!1} \Si_{kr} \Si_{kr\!+\!2}\cdots \Si_{(k\!+\!1)r\!-\!1})( \Si_{kr\!+\!1} \cdots \Si_{(k\!+\!1)r\!-\!2})} &(\ref{eq_rel_4})\\
=&(\A^{k\!+\!1}_1)^{(\Si_{kr\!+\!2}\cdots \Si_{(k\!+\!1)r\!-\!1})( \Si_{kr\!+\!1} \cdots \Si_{(k\!+\!1)r\!-\!2})} &\; (\ref{eq_rel_4})\\
=& \cdots =\A^{k\!+\!1}_{r\!-\!1} & R'_2, R'_5
\end{align*}

So we verify the statement (\ref{eq_rel_5}) with $j= k+1$ which imply (\ref{eq_rel_4}) for $j= k+1$.
For (\ref{eq_rel_3}) we need to check $(\A^{j'+1}_1)^{w_{(k+1)r-1}} = \Si_{j'r}$ if $j' < k+2$;
\begin{align*}
(\A^{j'+1}_1)^{w_{(k+1)r-1}}&= (\A^{j'+1}_1)^{\sigma_1 \Si_2 \cdots \Si_{j'\!r} \Si_{j'\!r+1} \cdots \Si_{(k+1)r-1}} &\text{definition of}\; \Si_{(k+1)r-1}\\
& = (\A^{j'+1}_1)^{\Si_{j'\!r} \Si_{j'\!r+1} \cdots \Si_{(k+1)r-1}} &(\ref{eq_rel_1}), \; R'_3\\
& = (\Si_{(j'+1)r})^{\Si_{j'\!r+2} \cdots \Si_{(k+1)r-1}} &(\ref{eq_rel_1}), \;(\ref{eq_rel_4})\\
&= \cdots = \Si_{(j'+1)r} &(\ref{eq_rel_1}).
\end{align*}

Now we see identities (\ref{eq_rel_1})-(\ref{eq_rel_5}) imply $\langle \!\langle R \rangle \!\rangle \leq  \langle \!\langle R' \rangle \!\rangle$. On the other hand, Remark \ref{rmk_coxeter_rel} guarantees $ \langle \!\langle R' \rangle \!\rangle  \leq \langle \!\langle R \rangle \!\rangle$. So we have $ \langle \!\langle R' \rangle \!\rangle  = \langle \!\langle R \rangle \!\rangle$

The next step is to apply Tietze transformations to the presentation
(\ref{eq_coxeter_sys_new}) which is isomorphic to Coxeter presentation (\ref{coxeter_sys}) (up to notation for generators). First adjoin extra letters $\A^{j+1}_0$'s to get the following presentation
\begin{equation}\label{eq_new_presen_fin_symm}
 \langle \A_p^i,  \,\Si_{jr},\,\A^{j+1}_0\,|\, R \cup R_\B \,\rangle
\end{equation}
where $R_\B$ consists of words of the form
$$
\overline{(\A^{j+1}_0)}(\Si_{jr})^{\overline{w_{jr-1}}}
$$
corresponding to (\ref{eq_transform_gen_symm_2}). Then replace $R \cup R_\B$ by $R' \cup R_\B$ in (\ref{eq_new_presen_fin_symm}). This is legitimate since two normal closures $\langle \!\langle R \rangle \!\rangle$ and $\langle \!\langle R' \rangle \!\rangle$ are the same.
Finally we remove $\Si_{jr}$'s from the generating set and $R_\B$ from the relators simultaneously to get the presentation $\Sigma_{n,r}$. In all, $\Sigma_{n,r}$ and $S_{nr}$ are isomorphic.
\end{proof}

Let $n$ be a positive integer. From Theorem \ref{thm_iso_presen_symm} we have a sequence of symmetric groups $\Sigma_{n,r} = \langle \,A_{n,r} \, | \,  Q_{n,r}\,\rangle$ on $B_{n,r} \subset Y_n$, where $A_{n,r}$ and $ Q_{n,r}$ denote the generators and relators in (\ref{eq_presen_fin_symm_new}) respectively. The natural inclusion $B_{n,r} \hookrightarrow B_{n, r+1}$ induces an inclusion $i_r: \Sigma_{n,r} \hookrightarrow \Sigma_{n,r+1}$ such that
$$
A_{n,r} \hookrightarrow A_{n,r+1} \; \text{and}\; Q_{n,r} \hookrightarrow Q_{n,r+1}
$$ for $r \in \N$. The direct limit $\underrightarrow{lim}_r\, \Sigma_{n,r}$ with respect to inclusions $i_r$ is nothing but the infinite symmetric group $\Sigma_{n,\infty}$ on $Y_n$ consisting of permutations with finite support.

\begin{lemma}\label{lm_A_n_iso_Sigma}
For each $n \in \N$, $\Sigma_{n, \infty} \cong \langle \,A_n \, | \,  Q_n\,\rangle$, where $A_n = \cup_r A_{n,r}$ and $Q_n = \cup_r Q_{n,r}$.
\end{lemma}

\begin{proof}For each $k \in \N$ we have a commuting diagram

\[\begin{CD}
    \langle \,\cup_{r=1}^k A_{n,r} \, | \, \cup_{r=1}^k Q_{n,r}\,\rangle @>>>  \langle \,A_n \, | \,  Q_n\,\rangle \\
       @VVV @VVV\\
       \Sigma_{n,k} @>>> \Sigma_{n, \infty}.
\end{CD}\]
The isomorphism on the right follows from the fact that the left map is an isomorphism and that the two horizontal maps are inclusions.
\end{proof}

The previous Lemma implies that one can get a presentation for $\Sigma_{n,\infty}$ by allowing the subscript $p$ to be any positive integer in (\ref{eq_presen_fin_symm_new}). We keep the same notation $R'_1$-$R'_6$ to indicate relators of the same type in $Q_n$.

Recall elements $g_1, \cdots, g_{n-1}$ of $\ho_n$ defined in (\ref{eq_generator_g_i}), and it is easy to see that
\begin{equation*}
[g_i, g_j]=\A^1_1
\end{equation*} for $i \neq j$. Now consider the action of $g_1, \cdots, g_{n-1}$ on $A_n$.

\begin{equation*}
(\A_p^1)^{g_k^{-m}}\! =\! \A^1_{p+m},\;
\;(\A_1^1)^{g_k}\!\! =\!\A^k_0,\;\,\text{for} \;\,p,m \in \N, 1\leq k\leq n-1\;
\end{equation*}
\begin{equation*}
(\A^k_0)^{g_k}\! \!=\! \A_{1}^{k+1},\;\text{for} \;\, 1\leq k\leq n-1\;
\end{equation*}
\begin{equation*}
(\A^k_0)^{g_{k'}^{-1}}\! =\! (\A^k_0)^{\A_1^1},\;
(\A^k_0)^{g_{k'}}\!\! =\! (\A^k_0)^{\A^{k'}_0},\;\text{for} \;\, 1\leq k \neq k'\leq n-1\;
\end{equation*}
\begin{equation*}
(\A_p^i)^{g_{i-1}^{m}}\!\! =\!(\A^i_{p+m}),\;
(\A^i_p)^{g_j^{-m}}\!\! =\!(\A^i_p)^{g_j^{m}}\!=\!(\A_{p}^i), \;\,\text{for} \;\,p,m \in \N,\;\,2\leq i\leq n\;\, \text{and}\;\, j \neq i-1\
\end{equation*}
Let $Q'_n$ be the relators corresponding to the above action. For $n\geq 3$, consider the group $H_n$ defined by
\begin{equation}\label{presen_for_abs_H_n}
H_n = \langle \, g_1, \cdots, g_{n-1},A_n \, | \,  Q_n, Q'_n, (\A^1_1)^{-1} [g_i, g_j]\,\rangle,  \; 1 \leq i < j \leq n-1
\end{equation}

\begin{lemma}\label{lm_quot_H_n_by_A_n}
We have a short exact sequence for each $n$;
\begin{equation}\label{eq_abelianization}
1 \to \langle\,\mathcal{A}_n\,\rangle \to H_n \rightarrow \mathds{Z}^{n-1} \to 1.
\end{equation}
\end{lemma}
\begin{proof}
Note that the relations $Q'_n$ ensure that $\langle A_n\rangle$ is a normal subgroup of $H_n$. Clearly the set of relations $Q_n$ becomes trivial in the quotient $H_n/\langle A_n \rangle$. Likewise, the relations in $Q'_n$ all become trivial in the quotient $H_n/\langle A_n \rangle$. The extra relator $(\A^1_1)^{-1} [g_i, g_j]$ implies that $g_i$ and $g_j$ commute in ${H_n/\langle\,\mathcal{A}_n\,\rangle\!}$ ($i \neq j$), so the quotient is the free abelian group generated by $g_1, \cdots, g_{n-1}$.
\end{proof}

\begin{lemma}\label{lm_iso_abs_H_n_Hough}
For $n\geq 3$, $H_n \cong \hsmH_n$.
\end{lemma}
\begin{proof}As in Johnson's argument for $n=3$, there exists a homomorphism $\psi :H_n \to \ho_n$ defined by
$$
\psi(g_i) = g_i \; \text{and} \; \psi(\A)= \A.
$$
So we have a commutative diagram with exact rows.
\[\begin{CD}
       1@>>>  \langle\,\mathcal{A}_n\,\rangle @>>>  H_n @>>>  H_n / { \langle\,\mathcal{A}_n\,\rangle }@>>>  1 \\
      @VVV @VVV @VVV @VVV @VVV \\
      1 @>>>  [\ho_n , \ho_n] @>>> \ho_n @>>> \mathds{Z}^{n-1} @>>>  1
\end{CD}\]
The existence of the top row and the isomorphism $H_n/\langle A_n\rangle \simeq \Z^{n-1}$ are the result of Lemma \ref{lm_quot_H_n_by_A_n}. The bottom row is given by Corollary \ref{coro_abelianization}, and the isomorphism $\langle A_n \rangle \simeq [H_n, H_n]$ follows from Lemma \ref{lm_Si=ker} and Lemma \ref{lm_A_n_iso_Sigma}. As the outer four vertical maps are isomorphisms, so is the inner one by the Five lemma.
\end{proof}

\begin{theorem}\label{thm_fin_presen_abs_H_n}
For $n \geq 3$, $H_n$ has a finite presentation
    \begin{equation}\label{eq_fin_presen_abs_H_n}
        H_n \cong \langle \,g_1, \cdots, g_{n-1}, \A\,|\,\,P_n \,\rangle
    \end{equation} where $P_n$ consists of\\

\begin{tabular}{ l l l }
 $r_1'$&:& $\A^2\!=\!1$,\\
 $r_2 '$&:& $(\A \,\A^{g_1})^3 =1$,\\
$r_3 '$&:& $\A \sim \A^{\overline{g}_1^2}$,\\
$r_4 '$&:& $[g_i, g_j] \!=\!\A$, for $1\!\leq \!i\!< \!j \leq \!n\!-\!1$,\\
$r_5 '$&:& $\A^{\overline{g}_i} = \A^{\overline{g}_j}$, for $1\!\leq \!i\!< \!j \leq \!n\!-\!1$
\end{tabular}\\

($\sim$ denotes commutation).

\end{theorem}
\begin{proof}$P_{n}$ already contains $(\A^1_1)^{-1} [g_i, g_j]$ for $1 \leq i < j \leq n-1$. We show $\langle\!\langle P_{n} \rangle\! \rangle$ contains $\langle\!\langle \mathcal{Q}'_{n} \cup \mathcal{Q}_{n} \rangle\!\rangle$ for all $n\geq3$ by induction on $n$. The base case is established by Theorem \ref{thm_johnson}. Assume $H_n$ has the presentation as in \ref{eq_fin_presen_abs_H_n} for $n\geq 3$. Obviously $g_1, \cdots, g_n$ and $\A$ generate $\ho_{n+1}$. Consider the natural inclusion of $\iota: Y_{n} \rightarrow Y_{n+1}$ such that $\iota(R_i)$ is the $i^{th}$ ray of $Y_{n+1}$. This induces an embedding $\iota:H_n \to H_{n+1}$, and $\mathcal{Q'}_{n+1}$ contains relations corresponding the action of $g_{n}$ on $\A^{n}_0$ and $\A_p^{n+1}$'s:
\begin{equation}\label{n}
(\A^{n}_0)^{{g}_{n}}\!=\!\A_1^{n+1},\;(\A^{n}_0)^{\overline{g}_{n}}\!=\!\A_1^{1},\;
(\A_{p}^{n+1})^{g_n}\!=\!\A_{p+1}^{n+1}
\end{equation}
for $p \in \mathds{N}$. This implies $\langle\!\langle \iota(\mathcal{Q}_n)\rangle\!\rangle = \langle\!\langle \mathcal{Q}_{n+1} \rangle\!\rangle$ in $H_{n+1}$. So the normal closure of $ P_{n+1}$ contains $\langle\!\langle \mathcal{Q}_{n+1} \rangle\!\rangle$ by the induction assumption.

We must examine the action of $g_1, \cdots, g_{n}$ on $\mathcal{A}_n$'s:
\begin{equation}\label{base1}
(\A^1_p)^{\overline{g}_n}=(\A^1_p)^{\overline{g}_i}, \;k \in \N,\,1\leq i \leq n-1,
\end{equation}
\begin{equation}\label{base2}
(\A^i_p)^{g_n}= \A^i_p, \;p \in \N,\,2\leq i \leq n-1,
\end{equation}
\begin{equation}\label{base3}
(\A^i_0)^{g_n}=(\A^i_0)^{\A^{n}_0},\;(\A^i_0)^{\overline{g}_n}=(\A^i_0)^{\A_1^1},\;1\leq i \leq n-1,
\end{equation}
\begin{equation}\label{base4}
(\A^{n}_0)^{\overline{g}_i}=(\A^{n}_0)^{\A_1^1}, \;(\A^{n}_0)^{g_i}=(\A^{n+1}_0)^{\A^{i}_0},
\;1\leq i \leq n-1,
\end{equation}
\begin{equation}\label{base5}
(\A^{n+1}_p)^{g_i}={\A^{n+1}_p}, \;k \in \N,\,1\leq i \leq n-1,
\end{equation} together with (\ref{n}) and $\iota(Q'_n)$. From Lemma \ref{lm_key_identity} we see that $P_{n+1}$ implies
\begin{equation}\label{stone4}
\A \sim (\overline{g}_{n})^{\,\overline{g}_i^{2}},\; \A \sim (\overline{g}_i)^{\,\overline{g}_{n}^{2}}\!,\;1\leq i \leq n-1,
\end{equation}
\begin{equation}\label{stone5}
\A^{\,\overline{g}_i^{k}}=\A^{\,\overline{g}_{n}^{k}}\!, \; \A \sim \A^{\overline{g}_i^{(k+1)}}, \;k \in \N,
\;1\leq i \leq n-1
\end{equation}
We want to establish identities (\ref{base1})-(\ref{base5}) from (\ref{stone4}),  (\ref{stone5}) and $P_{n+1}$. The identity (\ref{base1}) is a part of (\ref{stone5}). Observe that, by using (\ref{n}), one can reduce (\ref{base2}) and (\ref{base5}) respectively, as follows
\begin{equation}\label{Base2}
\A^{g_i^{k+1}} \!\sim g_n, \; \A^{g_n^{k+1}}\!\sim g_i,\; 2\leq i \leq n-1, \,k \in \N,
\end{equation}

One can verify (\ref{Base2}) by inductions based on (\ref{stone4}) as follows. The identity (\ref{stone4}) provides the base cases. Assume $\A^{g_i^{k+1}} \!\sim g_n$ for $k \geq 1$. Then $r_1$ and $r_2$ of $P_{n+1}$ imply
$$
\A^{g_i^{k+2}}\!\sim \overline{g}_i g_n g_i = \overline{g}_i \A g_i  g_n = \A^{g_i} g_n
$$On the other hand, $\A^{g_i^{k+2}}\!\sim \A^{g_i}$ by (\ref{stone5}). Thus $\A^{g_i^{k+2}}\!\sim g_n$. An analogous argument verifies the second identity of (\ref{Base2}).

Note that (\ref{base3}) and (\ref{base4}) are equivalent to
$$
\A^{g_i}\! \sim \overline{g}_n \A, \; \A^{g_n}\! \sim  \A g_i,\; 1\leq i \leq n-1,
$$
which are immediate consequence of (\ref{stone4}).

In all, $\langle\!\langle P_{n+1} \rangle\! \rangle$ also contains $\langle\!\langle \mathcal{Q}'_{n+1} \rangle\!\rangle$. So $P_{n+1}$ is enough for relators in $H_{n+1}$, and hence $H_n$ has the presentation (\ref{eq_fin_presen_abs_H_n}) for $n \geq 3$.
\end{proof}

\begin{lemma}\label{lm_key_identity}
$P_n$ implies the following identities
\begin{equation}\label{stone1}
\A \sim (\overline{g}_{n-1})^{\,\overline{g}_i^{2}},\; \A \sim (\overline{g}_i)^{\,\overline{g}_{n-1}^{2}}\!,\;1\leq i \leq n-2,
\end{equation}
\begin{equation}\label{stone2}
\A^{\,\overline{g}_i^{k}}=\A^{\,\overline{g}_{n-1}^{k}}\!, \; \A \sim \A^{\overline{g}_i^{(k+1)}}, \;k \in \N,
\;1\leq i \leq n-2
\end{equation}
\end{lemma}
\begin{proof}
From $r_1$, $r_2$ and $r_5$ we have
$$
\A \A^{\overline{g}_i} \A =  \A^{\overline{g}_i} \A \A^{\overline{g}_i} \Rightarrow \A^{\overline{g}_i \A}=\A^{g_i \A \overline{g}_i }\Rightarrow \A^{g_i \A}= \A^{\overline{g}_i \A g_i} = \A^{g_{n-1}^{-1} \A g_i}=\A^{g_i \overline{g}_{n-1}}.
$$
Thus,
$$
\A^{g_i}  \sim \overline{g}_{n-1} \overline{\A} = g_i \overline{g}_{n-1} \overline{g}_i \Rightarrow  \A \sim (\overline{g}_{n-1})^{\overline{g}_i^{2}}.
$$
One obtains $\A \sim (\overline{g}_i)^{\,\overline{g}_{n-1}^{2}}$ by using an analogous argument. Next, to establish (\ref{stone2}), we use simultaneous induction on $k$ together with (\ref{stone1}). Observe that $r_4$ and $r_5$  provide the base case $k=1$ and that $\A^{\,\overline{g}_i^{k}}=\A^{\,\overline{g}_{n-1}^{k}}$ holds trivially when $k=0$. Now assume
\begin{equation}\label{stone3}
\A^{\,\overline{g}_i^{(k-2)}}=\A^{\,\overline{g}_{n-1}^{(k-2)}},\; \A^{\,\overline{g}_i^{(k-1)}}=\A^{\,\overline{g}_{n-1}^{(k-1)}},\; \A \sim \A^{\overline{g}_i^{k}}
\end{equation}for $k \geq 2$. From $r_1$ and (\ref{stone2}),
$$
\A^{\,\overline{g}_i^{k}} =\A^{\,\overline{g}_i^{k}\A}= \A^{\,\overline{g}_{n-1}^{(k-1)} \overline{g}_i \A} = \A^{\,\overline{g}_{n-1}^{(k-2)} \overline{g}_i \overline{g}_{n-1}} =\A^{\,\overline{g}_i^{(k-1)} \overline{g}_i \overline{g}_{n-1}} =\A^{\,\overline{g}_{n-1}^{k}}.
$$ For the second assertion, note that $r_4$ and the first identity of (\ref{stone1}) imply that $\A^{g_i^2}$ commute with both $\A$ and $g_{n-1}^{-1}$. So we have
$$
\A ^{g_i^2} \sim \A^{\overline{g}_{n-1}^{(k-1)}}= \A^{\overline{g}_i^{(k-1)}},
$$or equivalently,
$$
\A \sim \A^{\overline{g}_i^{(k+1)}}.
$$

\end{proof}

As a consequence of Lemma \ref{lm_iso_abs_H_n_Hough} and Theorem \ref{thm_fin_presen_abs_H_n} we have

\medskip
{\bf Theorem \ref{thm_fin_presen_Hough}.}\;
{\sl For $n \geq 3$, $\ho_n$ has a finite presentation
    \begin{equation}\label{fin_ho_n}
        \ho_n \cong \langle \,g_1, \cdots, g_{n-1}, \A\,|\,\,P_n \,\rangle
    \end{equation} where $P_n$ is the same as in the presentation (\ref{eq_fin_presen_abs_H_n}) of Theorem \ref{thm_fin_presen_abs_H_n}.}
\medskip

\subsection{Further properties of $\ho_n$}

\begin{definition}[\textbf{Amenable and a-T-menable groups}]
A (discrete) group $G$ is called $amenable$ if it has a $F{\o}lner$ sequence, i.e.,
 there exists a sequence $\{F_i\}_{i\in \N}$ of finite subsets of $G$ such that
$$
| g \cdot F_i \,\triangle \,F_i| \over |F_i|
$$tends to $0$ for each $g \in G$.

A group $G$ is called \emph{a-T-menable} if there is a proper continuous affine action of $G$ on a Hilbert space.
\end{definition}

Recall the following known fact about amenable groups (\cite{Runde}).
\begin{theorem}\label{facts}The class of amenable group contains abelian groups, and locally finite groups. An extension of an amenable group by another amenable group is again amenable.
\end{theorem}

The above facts together with the short exact sequence (\ref{eq_abelianization_Si}) implies amenability of $\ho_n$.
\begin{equation*}
1 \to \Sigma_{n,\infty} \to \hsmH_n \xrightarrow{\varphi} \mathds{Z}^{n-1} \to 1,
\end{equation*}

\begin{theorem}$\hsmH_n$ is amenable for all $n$.
\end{theorem}
The class of amenable groups contains the class of a-T-menable groups.
\begin{corollary}
$\hsmH_n$ is a-T-amenable for all $n$. As a consequence, Houghton's groups satisfy the Baum-Connes conjecture.
\end{corollary}



\section{A cubing  for $\ho_n$}\label{cubical}
\subsection{Definition of the cubing $X_n$}\label{construction}

In \cite{Brown}, Ken Brown constructed an infinite dimensional cell complex on which $\ho_n$ acts by taking the geometric realization of a poset, which is defined using factorizations in a monoid $\M_n$ containing $\ho_n$. In this section we modify his idea to construct an $n$-dimensional cubical complex $X_n$ on which $\ho_n$ acts. It turns out that $X_n$ is a cubing for all $n \in \N$. We discuss this in detail in Section \ref{section_CC}.

Two monoids $\M_n$ and $\TT_n$ are important for the construction. Fix a positive integer $n \in \N$. Let $\M_n$ be the monoid of injective maps $Y_n \to Y_n$ which behave as eventual translations, i.e., each $\A \in \M_n$ satisfies
\begin{itemize}
    \item [$\star$] There is an $n$-tuple $(m_1 ,\cdots, m_n) \in \mathds{Z}^n$ and a finite set $K \subset Y_n$ such that $ \A(k,p)=(k, p\!+\!m_k)$ for all $(k,p) \in Y_n - K$.
\end{itemize}
The group homomorphism $\varphi :\ho_n \to \Z ^n$ naturally extends to a monoid homomorphism $\varphi :\M_n \to \Z^n$ and
$\varphi(\A) =  (m_1 ,\cdots, m_n)$. A monoid homomorphism $h:\M_n \to \Z_{\geq 0}$ is defined by $h(\A)= \sum m_i$. This map $h$ is called the \emph{ height function} on $\M_n$. For convenience let $S(\A)$ denote the discrete set $Y_n -(Y_n)\A$. One can readily check that $h(\A)= |S(\A)|$ and that $\ho_n = h^{-1}(0)$.

Consider elements $t_1, \cdots, t_n $ of $\M_n$ where $t_i$ is the \emph{translation} by $1$ on the $i^{th}$ ray, i.e.,
\begin{equation*}
(j,p) t_i=
    \begin{cases}
        (j,p+1) & \text{if } j=i \\
        (j,p) & \text{if } j \neq i.
    \end{cases}
\end{equation*}for all $p \in \mathds{N}$.
Let $\TT_n \subset \M_n$ be the commutative submonoid generated by $t_1, \cdots, t_n $.

Figure \ref{example of M_n} illustrates some examples of $\M_n$, as before, where points which do not involve arrows are meant to be fixed, points of each finite set $K$ are indicated by circles. The left most one is a generator $t_i$ of $\TT_n$ which behaves as the translation on $R_i$ by $1$ and fixes $Y_n - {R_i}$ pointwise. The next example shows the commutativity of $\TT_n$; $t_i t_j = t_j t_i$. In the third figure, $\A = [g_1, g_2]$ as before. The last element $g$ is rather generic one with $h(g)=1$.
\begin{figure}[ht]
\begin{center}
\includegraphics[width=1.0\textwidth]{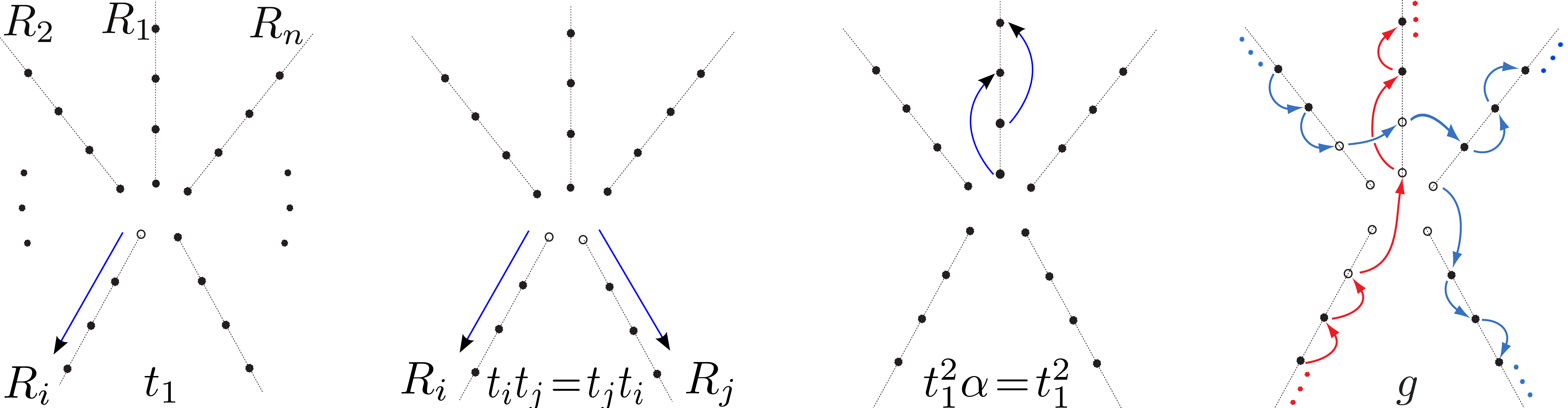}
\caption{Some elements of $\M_n$} \label{example of M_n}
\end{center}
\end{figure}

One crucial feature of $\TT_n$ is that it is commutative. $\TT_n$ acts on $\M_n$ by left multiplication (pre-composition). This action induces a relation $\geq$ on $\M_n$, i.e.,
\begin{equation}\label{porder}
\A_1 \geq \A_2 \;\; \textrm{if} \;\; \A_1 =t\A_2  \;\; \textrm{for some} \;\;t\in \TT_n.
\end{equation}
for all $\alpha_1, \alpha_2, \alpha_3 \in \mathcal{M}_n$

\begin{proposition}\label{prop_porder}
The relation $\geq $ on $\M_n$ is a partial order, i.e., it satisfies for all $\A_1, \A_2, \A_3 \in \M_n$,
\begin{itemize}
    \item $\A_1 \geq \A_1$,
    \item if $\A_1 \geq \A_2$ and $\A_2 \geq \A_1$ then $\A_1 =\A_2$,
    \item if $\A_1 \geq \A_2$ and $\A_2 \geq \A_3$ then $\A_1 \geq \A_3$.
\end{itemize}
\end{proposition}

\begin{proof}Reflexivity can by verified by using $t=1\in \TT_n$ in equation \ref{porder}. The condition $\A_1 \geq \A_2$ and $\A_2 \geq \A_1$ implies that $h(\A_1)=h(\A_2)$ and so $t=1$ is the only candidate for $t$ in those inequalities. So we have antisymmetry. One can check transitivity by taking an obvious composition of $t$'s in the two inequalities.
\end{proof}

Being a submonoid, $\ho_n$ acts on $\M_n$. In this paper, we focus on the action of $\ho_n$ on $\M_n$ by right multiplication (post-composition).

\begin{remark}
The two actions of $\ho_n$ and $\TT_n$ on $\M_n$ commute.
\end{remark}

A chain $\A_0\leq \A_1 \leq \cdots \leq \A_k$ corresponds to a $k$-simplex in the the geometric realization $ \|\M_n\|$. The length of chains is not bounded and hence the dimension of $\| \M_n\|$ is infinite. To show the finiteness properties of $\ho_n$ (Corollary \ref{coroll_fin_prop_H_n}), Brown studied the action of $\ho_n$ on $\| \M_n\|$ equipped with a filtration satisfying the condition of Theorem \ref{KBT}.

Instead, for the construction of desired complex $X_n$, we want to start with Cayley graph of $\M_n$ with respect to the generating set $\{t_1, \cdots, t_n\}$ for $\TT_n$. Recall the \emph{Cayley graph} $\C_{\mathcal{A}}(\Gamma)$ of a group $\Gamma$ with respect to a generating set $\mathcal{A}$ is the metric graph whose vertices are in $1-1$ correspondence with the elements of $\Gamma$ and which has an edge (labelled by $a$) of length $1$ joining $\gamma$ to $\gamma a$ for each $\gamma\in \Gamma$ and $a \in \mathcal{A}$.

\begin{definition} [\textbf{Cayley graph $\C_n$}] Associated to the monoid $\M_n$, a directed graph $\C_n$ is defined by
\begin{itemize}
    \item The vertex set $\mathcal{C}^{(0)}_n$ is in $1-1$ correspondence with the elements of $\M_n$. By abusing notation, let $\A$ denote the vertex of $\C_n$ corresponding the same element of $\M_n$.
    \item A vertex $\A$ is joined to another vertex $\B$ by an edge $e$ (of length $1$) labelled by $t_i$ if $\B = t_i\A$ for some $t_i$; we assume all edges are directed, i.e., directed from $\A$ to $\B$.
\end{itemize}
\end{definition}

The monoid homomorphism $h : \M_n \to \Z_{\geq 0}$ can be extended linearly to a map $h :\C_n \to \R_{\geq0}$. Note that $h(\B) =h(\A)+1$ if $\B =t_i\A$ for some $t_i$. In this situation, we say that the edge $e_{t_i}$ is directed ``upward" with respect to the height function $h$.

\begin{remark}\label{rmk:unique_asc_edge}
For a vertex $\B$ and a fixed $i$ ($1\leq i \leq n$), there is unique edge labelled by $t_i$ whose initial vertex is $\B$; namely, the edge which ends in the uniquely determined vertex $t_i\beta$.
\end{remark}

\medskip
\noindent
{\bf Least upper bound in $\M_n$.}

\begin{lemma}\label{lm:lub}
For $\A_1, \A_2 \in \M_n$, there exists a unique element $\B$ such that
\begin{itemize}
    \item $\B\geq \A_1$ and $\B \geq \A_2$, and
    \item if $\B' \geq \A_1$ and $\B' \geq \A_2$, then $\B' \geq \B$.
\end{itemize}
\end{lemma}

\begin{proof}Suppose $\varphi(\A_1)= (m_1, \cdots, m_n)$ and $\varphi(\A_2)= (m'_1, \cdots, m'_n)$. Consider the restrictions $\A_1{_{|R_i}}$ and $\A_2{_{|R_i}}$ for $i=1, \cdots, n$. One can find smallest non-negative integers $k_i, k'_i$ so that
\begin{equation}\label{finite_0}
t_i^{k_i}\A_1{_{|R_i}}= t_i^{k'_i}\A_2{_{|R_i}}
\end{equation}
for each $i$ as follows. If $m_i > m'_i$ then there exists smallest $p_0 \in \N \cup \{0\}$ such that
\begin{equation*}\label{finite}
(i, p)\A_1= (i, p)t_i^{m_i -m'_i}\A_2
\end{equation*}
for all $p > p_0$. The existence of $p_0$ is obvious. On $R_i$, $\A_1$ and $t_i^{m_i -m'_i}\A_2$ agree outside a finite set $L$ (this set $L$ can be smaller than the finite set in the definition of $\M_n$). The integer $k_0$ is determined by the point in $L \cap R_i$ with the largest distance from the origin (if the intersection is trivial then $p_0 =0$). Set $k_i = p_0$ and $k'_i = m_i - m'_i +p_0$. The integers $k_i$ and $k'_i$ are desired powers for $t_i$ satisfying (\ref{finite_0}). One can find appropriate $k_i$ and $k'_i$ in a similar way in case $m_i < m'_i$. If $m_i = m'_i$ then set $k_i = k'_i= {p_0}$. Apply this process to get $k_i$'s and $k'_i$'s for all $i=1, \cdots, n$.

Now set $\B = t_1^{k_1} \cdots t_n^{k_n} \A_1 = t_1^{k'_1} \cdots t_n^{k'_n}\A_2 $. The first condition clearly holds. The definition of $\geq$ together with the minimality of the $k_i$ and $k'_i$ ensure that if $\beta' \geq \alpha$ and
$\beta' \geq \alpha'$, then $\beta' \geq \beta$, and hence condition two holds. Finally, the uniqueness of $\beta$ follows from the second condition together with the antisymmetry of $\geq$. (Proposition \ref{prop_porder}).
\end{proof}

\begin{definition}[\textbf{Least upper bound}]
An element $\B$ is called an \emph{upper bound} of $\A_1$ and $\A_2$ if it satisfies the first condition of Lemma \ref{lm:lub}. If $\beta$ is the unique element satisfying both conditions of Lemma \ref{lm:lub}, then it is called the \emph{least upper bound} (simply $lub$) of $\A_1$ and $\A_2$ and is denoted by $\A_1 \wedge \A_2$.
\end{definition}

\begin{figure}[ht]
\begin{center}
\includegraphics[width=1.0\textwidth]{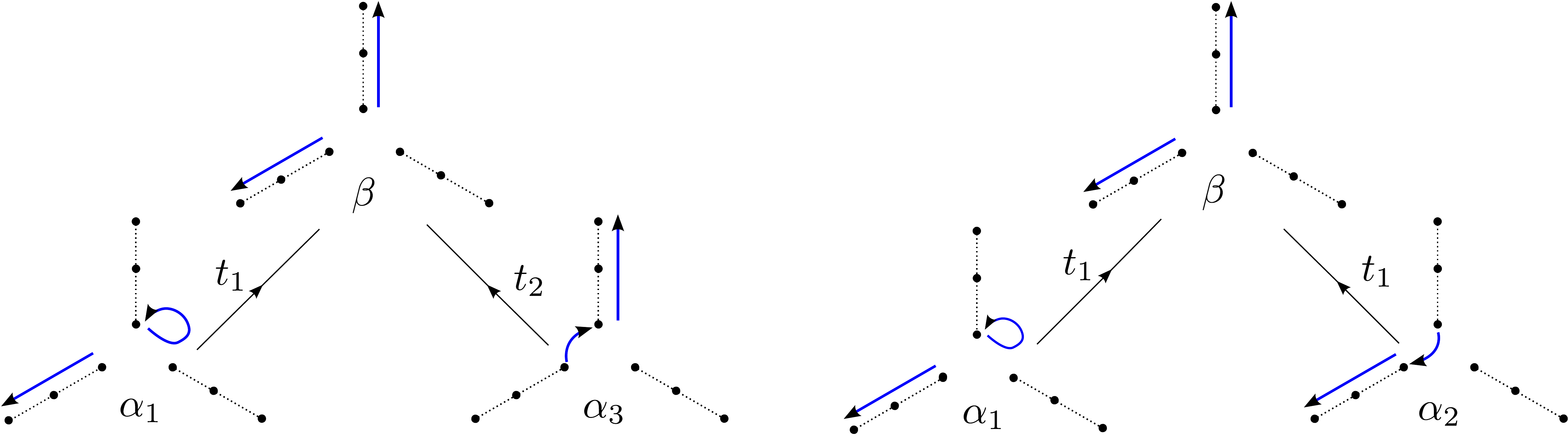}
\caption{$\B$ is an upper bound of $\A_1, \A_2$ and $\A_3$}\label{fig_exmaple_lub}
\end{center}
\end{figure}

Figure \ref{fig_exmaple_lub} shows some elements of $\M_n$ with their upper bound; $\B$ is an upper bound of $\A_1, \A_2$ and $\A_3$. Note that, in the right figure, edges directed into the vertex $\B$ are not necessarily labelled by distinct generators of $\TT_n$. The reader should contrast this with the situation of edges directed away from a vertex as described in Remark \ref{rmk:unique_asc_edge}.


The following proposition implies that the notion of $lub$ for a finite collection of $\M_n$ is well-defined.
\begin{proposition}\label{prop:lub_multi}
Suppose $\{\A_1 , \cdots, \A_\ell\} \subset \M_n$. There exists $lub$ of those finite elements, i.e. there exists a unique element $\B$ satisfying
\begin{itemize}
    \item $\B\geq \A_j$, $j=1, \cdots, \ell$,
    \item if $\B' \geq \A_j$ for all $j$ then $\B' \geq \B$.
\end{itemize}
\end{proposition}

\begin{proof}
We extend the idea in the proof of Lemma \ref{lm:lub} to find $\B$ with desired properties. Fix $i$ ($1\leq i \leq n$). Suppose the translation lengths of $\A_j$ `at infinity' is $m_{i,j}$ (i.e., $\A_j$ has image $m_{i,j}$ under the composition $ \M_n \xrightarrow {\varphi} \Z^n  \xrightarrow{\pi_i}
\Z$ where $\pi_i$ is the projection to the $i^{th}$ component). We want to find smallest non-negative integers $k_{i,1}, \cdots, k_{i,\ell}$ so that
\begin{equation*}\label{finite_1}
t_i^{k_{i,j}}\A_j{_{|R_i}}= t_i^{k_{i,j'}}\A_{j'}{_{|R_i}}
\end{equation*}for all $j,j' \in \{1, \cdots, \ell\}$. Let $M_i = max\{m_{i,1}, \cdots, m_{i,\ell}\}$. Observe that, for $j=1, \cdots, \ell$, elements $t_i^{M_i -m_{i,j}} \A_j$ act as a translation on $R_i$ by $M_i$ outside a finite set $K$. They agree on $R_i$ outside a finite set $L$ (as in the proof Lemma \ref{lm:lub}, $L$ can be smaller than $K$). So there exists smallest $p_0 \in \N\cup\{0\}$ such that
\begin{equation*}\label{finite_3}
(i, p)\A_j=(i, p)\A_{j'}
\end{equation*}for all $p > p_0$ and $j,j' \in \{1, \cdots, \ell\}$. Now set $k_{i,j} = M_i -m_j + p_0$ for $j=1, \cdots , \ell$. Apply this process to find $k_{i,j}$ for all $i=1, \cdots, n$. As before, we define
\begin{equation}\label{finite_3}
\B =t_1^{k_{1,j}} \,t_2^{k_{2,j}} \cdots t_n^{k_{n,j}}\,\A_j
\end{equation}
for some $j \in \{1, \cdots, \ell\}$. Note it is well-defined because $t_1^{k_{1,j}}  \cdots t_n^{k_{n,j}}\,\A_j= t_1^{k_{1,j'}} \cdots t_n^{k_{n,j'}}\,\A_{j'}$ for any $j' \in \{1, \cdots, \ell\}$. The rest of the argument is analogous to the proof of Lemma \ref{lm:lub}. By the definition (\ref{finite_3}) of $\B$ it satisfies the first condition. The second condition is also satisfied by the minimality of powers $k_{i,j}$. Finally, the uniqueness follows from the antisymmetry of $\geq$.
\end{proof}

\begin{remark}\label{rmk:lub_inclusion}
The operator $lub$ preserves the inclusion between finite sets of $\M_n$, i.e.,
$$
lub(C_1) \leq lub(C_2)
$$
if $C_1 \subset C_2 \subset \M_n$.
\end{remark}

\medskip
\noindent
{\bf Cubical structure of $\C_n$ and the definition of $X_n$.}\\
Commutativity of $\TT_n$ plays an important role in the construction of $X_n$; any permutation in the product $t_1 t_2 \cdots t_k$, $k \leq n$, represents the same element in $\TT_n$. Yet each variation in the expression $t_1 t_2 \cdots t_k$ determines a path from $\A_1$ to $ t_1 t_2 \cdots t_k \A_1$. These $k!$ many paths form the $1$-skeleton of a $k$-cube. Figure \ref{fig_boundary_3_cube} illustrates $3!$ paths from to which form $1$-skeleton of a $3$-cube having $\B$ and $t_i t_j t_k \B$ as its bottom and top vertex respectively.

\begin{figure}[ht]
\begin{center}
\includegraphics[width=0.3\textwidth]{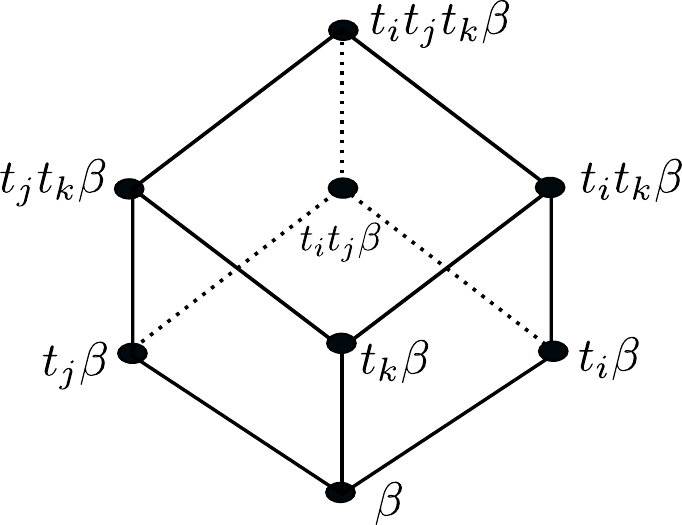}
\caption{Three distinct generators $t_i, t_j, t_k$ determine a $3$-cube in $X_n$.} \label{fig_boundary_3_cube}
\end{center}
\end{figure}

\begin{definition}[\textbf{Cubical complex $X_n$}]\label{dfn:X_n}
For each $n\in \N$ $X_n$ is defined inductively as follows
\begin{itemize}
    \item[(1)] $X_n^{(1)} := \mathcal{C}_n$,
    \item [(2)] for each $k \geq 2$, $X_n^{(k)}$ is obtained from $X_n^{(k-1)}$ by attaching a $k$-cube along every copy of the boundary of a $k$-cube in  $X_n^{(k-1)}$.
\end{itemize}
\end{definition}

The height function $h:\C_n \to \R_{\geq 0}$ extends linearly to a Morse function $h:X_n \to  \R_{\geq 0}$. In our study, a Morse function $h$ is roughly a height function which restricts to each $k$-cube to give the standard height function. (See Definition \ref{dfn:Morse_function} for a Morse function) Let $\Box^k$ denote the standard $k$-cube, i.e., $\Box^k = \{(x_1, \cdots, x_k) \in \R^k\,|\, 0 \leq x_i \leq 1\}$. The standard height function $\overline{h} : \Box^k \to \R$ is defined by $\overline{h}(x_1,\cdots x_k) = \sum x_i$. Our Morse function $h$ can be described as follows. Let $\Si^k_j \subset X_n$ be a $k$-cube and let $\varphi_j : \Box^k \to {\Si^k} \subset X_n$ denote the attaching map (isometric embedding) for $\Si^k_j$ used in the construction of $X_n$. The image of $\Si^k_j = \varphi_j(\Box^k)$ under $h$ is the translation of $\overline{h}(\Box^k)$ by $h\varphi_j (0)$. In other words, the following diagram commutes.

\begin{figure}[ht]
\begin{center}
\includegraphics[width=0.35\textwidth]{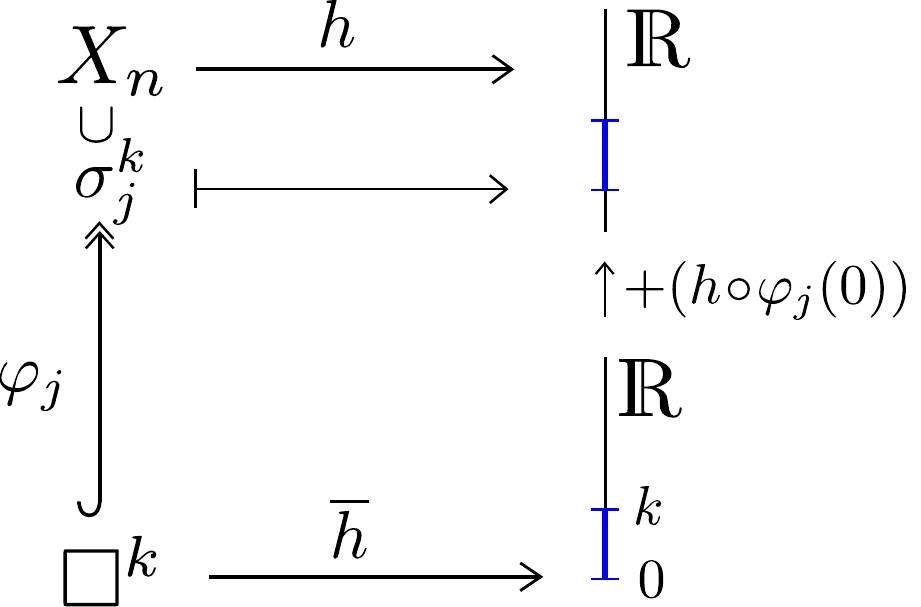}
\caption{A Morse function $h$ restricted to a cube is the standard height function up to a translation.}\label{fig_diagram_morse}
\end{center}
\end{figure}

Two actions of $\ho_n$ and $\TT_n$ on $\M_n$ extend to actions on $X_n$. Note that cell(cube) structure of $X_n$ is completely determined by $\TT_n$ whose action commutes with the action of $\ho_n$. So the action of $\ho_n$ on $X_n$ is cellular, i.e., it preserves cell structure. If a cube $\sigma$ is spanned by a collection of vertices $\{\A_1, \cdots, \A_m\}$, then the cube $\sigma \cdot g$ is the same dimensional cube spanned by $\{\A_1 g, \cdots, \A_m g\}$.

Intuitively the action of $\ho_n$ can be considered as a `horizontal' action. Each $g\in \ho_n$ preserves the height. By contrast, the action of $\TT_n$ is `vertical'. Each non trivial element of $\TT_n$ increases the height under the action. We discuss these actions on $X_n$ further in section \ref{sec_action_of_H_n_on_X_n}.

\begin{lemma}\label{lm:C_n_conn}
The graph $\C_n$ is simplicial and connected for all $n \in \N$.
\end{lemma}

\begin{proof}Connectedness is an immediate consequence of the existence of an upper bound for any pair of vertices in $\M_n$ (Lemma \ref{lm:lub}). Suppose $\A_1$ and $\A_2$ are vertices of $\C_n$. Let $\B$ be a vertex with
$$
\B = \T_1 \A_1 \;\text{and}\; \B = \T_2 \A_2
$$ for some $\T_1, \T_2 \in \TT_n$. Observe that $\T_1$ and $\T_2$ define paths $p_1$ and $p_2$ where $p_i$ joins $\A_i$ to $\B$, $i=1,2$. So $\A_1$ and $\A_2$ are connected by the concatenation of the two paths.

If an edge connects two vertices $\B_1$ and $\B_2$ then their heights differ by $1$. So no edge in $\C_n$ starts and ends at the same vertex. Finally we want to show there is no bigon in $C_n$. Suppose two distinct edges share two vertices $\B_1$ and $\B_2$. Again,those two vertices can not have the same height and so we may assume  $\B_1 \geq \B_2$. Our claim is that those two edges are labelled by one generater of $\TT_n$. If two edges were labelled by two distinct generators $t_i$ and $t_j$, then we would have $\B_1 = t_i \B_2$ and $\B_1 = t_j \B_2$. This contradicts the fact that $\B_1$ is injective because
$$
(i, 1)\B_1 = (i, 1)t_i \B_2 = (i, 2)\B_2 \;\text{but} \; (i, 1)\B_1 = (i, 1)t_j \B_2 = (i, 1)\B_2.
$$ Therefore the two edges are labelled by a generater of $\TT_n$. This means
\begin{equation}\label{finite_4}
\B_1 = t_i \B_2
\end{equation} for some $t_i$. By Remark \ref{rmk:unique_asc_edge}, there is only one edge connecting $\B_1$ and $\B_2$ satisfying (\ref{finite_4}). So no two distinct edges share two vertices in $\C_n$.
\end{proof}

\begin{lemma}\label{lm_X_n_1_1-conn}
The graph $X_1$ is simply connected.
\end{lemma}

\begin{proof}Suppose $\ell$ is a loop
in $X_1$. Consider the image of $\ell$ under a Morse function $h$. If $h(\ell) = [a,b]$ and $b-a <1$ then $\ell$ lies entirely on a ball of a vertex with radius $<1$. This ball is contractible and so $\ell$ is null homotopic. Assume $b-a \geq 1$. We may further assume that $a$ and $b$ are integers. Let $\ell\cap h^{-1}(a)=\{\A_1, \cdots, \A_k$. By Remark \ref{rmk:unique_asc_edge}, there exists an unique edge $e_j$ emanating from each $\A_j$ ($1\leq j \leq k$). Observe that the loop $\ell$ contains those edges and that $\ell$ is homotopy equivalent to $\ell - \cup_j e_j$ (Indeed, what we remove is not a whole edge but an edge minus the top vertex). Under this homotopy equivalence one can homotope $\ell$ to another loop $\ell'$ such that $h(\ell') = [a+1, b]$. Essentially, in the passage from $\ell$ to $\ell'$, we removed back-trackings involving vertices $\A_1, \cdots, \A_k$. Apply the same process to establish homotopy equivalence between $\ell$ and a vertex with height $b$. It can be shown that $h^{-1}(b)$ is a singleton set even if it is not clear a priori. After applying the above homotopy equivalences up to $b-a$ steps one obtains homotopy equivalence between the loop $\ell$ and $\ell\cap h^{-1}(b)$. If the preimage $h^{-1}(b)$ consisted of multiple number of vertices then $\ell$ was homotopic to a discrete set $h^{-1}(b)$. We have shown that $\ell$ is null homotopic.
\end{proof}

\medskip
\noindent
{\bf Lifting across a square.}
Fix $n \geq 2$. Suppose $\B$ is a vertex of $X_n$. For a pair of generators $t_i$ and $t_j$ ($i\neq j$), there exists a unique square $\Si$ containing four vertices $\B, t_i \B, t_j \B$ and $t_i t_j \B$. Consider the path $p$ with length $2$, joining two vertices $t_i\B$ and $t_j\B$, consisting of two lower edges of the square $\Si$. We say a path $p$ has a $turn$ at $\B$. This path $p$ is homotopic to a path $\tilde{p}$ rel two ends vertices which consists of two upper edges of $\Si$. Figure \ref{fig_lifting_square} illustrates this idea. We say $\tilde{p}$ is a $lifting$ of $p$ (across a square $\Si$). The notion of lifting is useful in proving Lemma \ref{lm_X_n_1-conn} as well as amenability of $\C_n$ in Section \ref{sec_median}.

\begin{figure}[ht]
\begin{center}
\includegraphics[width=0.30\textwidth]{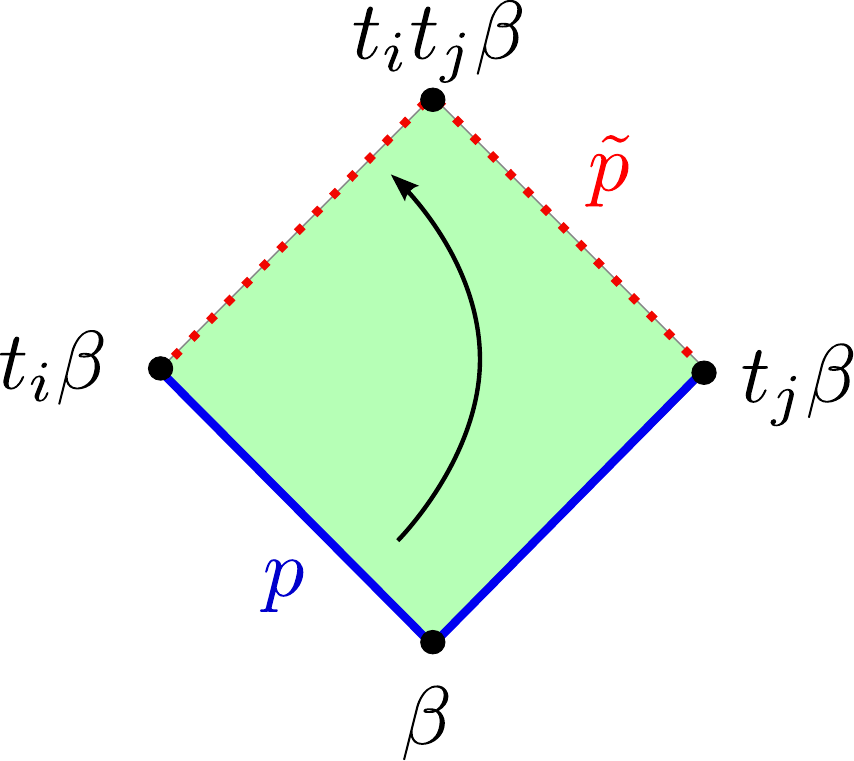}
\caption{A path with a turn at $\B$ lifts to a path $\tilde{p}$.}\label{fig_lifting_square}
\end{center}
\end{figure}

\begin{lemma}\label{lm_X_n_1-conn}
For all $n \geq 2$, $X_n$ is simply connected.
\end{lemma}
\begin{proof}
Suppose $\ell:\mathds{S}^1 \to X_n$ is a loop in $X_n$ for $n \geq 2$. We may assume that $\ell$ lies in $1$-skeleton of $X_n$ and that $h(\ell)=[a,b]$ for some non negative integers $a<b$. We may further assume that $\ell$ does not contain back-trackings.
As in the proof of Lemma \ref{lm_X_n_1_1-conn}, one can eliminate the case $b-a <1$ by using rather trivial homotopy equivalence for a ball of a vertex with radius $<1$. Let $\alpha_1, \ldots , \alpha_k$ be the vertices of $\ell$. By Proposition \ref{prop:lub_multi}, there exists a unique $lub$ $\beta = lub\{\alpha_1, \cdots , \alpha_k\}$. We want to establish a homotopy equivalence between $\ell$ and the trivial loop $\B$ by applying liftings across squares. We start with vertices in $\ell\cap h^{-1}(a)=\{\B_1, \cdots, \B_{k'}\}$. The loop $\ell$ may intersect multiple number of squares even for a single vertex $\B_j \in h^{-1}(a)$. Let $\ell^{-1}(\B_j) = \{v_{j,1},\cdots v_{j,r_j}\} \subset\mathds{S}^1  $ for $j = 1, \cdots, k'$. From the fact that $\ell$ does not contain back-tracking, we have the following crucial observation for each $j$.  \\
(1) Wach preimage $v_{j,m}$ belongs to an interval $I_{j,m}$ such that $\ell (I_{j,m})$ is a path having a turn at $\B_j$, \\
(2) the two end vertices of $\ell (I_{j,m})$ are given by $t_i \B_j$ and $t_{i'} \B_j$ for distinct generators $t_i$ and $t_{i'}$ of $\TT_n$.

So each interval $I_{j,m}$ determines a square $\Si_{j,m}$ whose bottom vertex is $\B_j$. Now take liftings of $\ell (I_{j,m})$ for all $j$ and $m$ ($ 1 \leq j \leq k'$, $1 \leq m \leq r$). The resulting loop $\ell'$ is homotopy equivalent to $\ell$. Another important observation is that the top vertex $w_{j,m}$ of a square $\Si_{j,m}$ in which lifting occured satisfies
\begin{equation}\label{eq_ineq_1}
\B \geq w_{j,m}.
\end{equation} Indeed, three vertices on $\ell (I_{j,m})$ were already vertices of the loop $\ell$. The top vertex $w_{j,m}$ is the $lub$ of those three vertices. By Remark \ref{rmk:lub_inclusion}, each vertex $w_{j,m}$ satisfies $\B \geq w_{j,m}$. This means that in the passage from $\ell$ to $\ell'$ vertices of $\ell\cap h^{-1}(a)$ were replaced by some vertices satisfying equation (\ref{eq_ineq_1}). So $\B$ is still an upper bound of the vertices of $\ell'$. Observe that $h(\ell')= [a+1,b]$.

Apply the same process consecutively to vertices with smallest height in each step. The property given by equation (\ref{eq_ineq_1}) ensures that this procedure stops after finitely many ($h(\B)-a$) steps. In all, the given loop $\ell$ converges to $\B$.
\end{proof}


\subsection{Properties of the monoid $\M_n$}
In this section we study some algebraic properties which are useful for us to study geometry of $X_n$ in the sequel.

\begin{lemma}\label{genM_n}
$\M_n = \TT_n \ho_n =\{t g \mid t\in \TT_n, \,g\in \ho_n \}$
\end{lemma}

\begin{proof}Suppose $\B \in \M_n$ and $\varphi(\B) =(m_1 ,\cdots, m_n)$. There exist $t \in \TT_n$ and $g \in \ho_n$ such that $h(t)= h(\A)$ and $\varphi(\B)= \varphi(t\,g)$. Thus $\B$ and $tg$ agree on $Y_n - K$ for some finite set $K$.  We want to find $f\in \ho_n$ such that $\B=tgf$. Existence of such $f$ comes from the fact that being injections, $\B$ and $tg$ have right inverses $f_1$ and $f_2$ respectively (i.e., left inverses in the composition of functions). Consider those right inverses $f_1$ and $f_2$
$$
f_1: (K)\B \to K\;\text{and}\; f_2 : (K)tg \to K.
$$
By definition, $\B f_1 = tg f_2$ on $K$. Observe that one can turn $f_2 f_1^{-1}$ into an element $f \in \ho_n$ with $supp(f)\subset (K) tg \cup (K)\B$. For example, one can extend $f_2 f_1^{-1}: (K) tg \to (K) \B$ to $f:(K) tg \cup (K)\B \to (K) tg \cup (K)\B $ by a bijection $f' : (K)\B - (K)tg \to (K)tg - (K)\B$. More precisely, $f$ is defined by
\begin{equation*}
f =
    \begin{cases}
       f_2 f_1^{-1} & \text{on  } (K) tg \\
       f' & \text{on  } (K)\B - (K)tg
    \end{cases}
\end{equation*}where $f'$ is a bijection between congruent finite sets $(K)\B - (K)tg $ and $ (K)tg - (K)\B$.
\end{proof}

\begin{definition}[\textbf{Greatest lower bound}]\label{def_glb}
An element $\B$ is called a \emph{lower bound} of $\A_1$ and $\A_2$ if it satisfies the first condition below, and is called the \emph{greatest lower bound} (simply $glb$) if it satisfies the second condition as well.
\end{definition}
\begin{itemize}
    \item $\B \leq \A_1$ and $\B \leq \A_2$,
    \item If $\B' \leq \A_1$ and $\B' \leq \A_2$ then $\B' \leq \B$.
\end{itemize}
The $glb$ of $\A_1$ and $\A_2$ is denoted by $\A_1 \vee \A_2$.

Note that not every pair of vertices admits a lower bound (and in particular greatest lower bound). For example, the right figure in Figure \ref{fig_glb_square} describes $\alpha_1$ and $\alpha_3$ which do not have a lower bound. If there was a common lower bound $\gamma'$ then we would be forced to have $(2,1)\CC' = (1,1) \CC'$, contradicting the injectivity of $\gamma'$. The geometric interpretation of this is that there is no square containing $\A_1$ and $\A_3$. However, if one replaces $\A_1$ by $\A_2$ then there exists a common lower bound (indeed $\A_2 \vee \A_3$) for $\A_2$ and $\A_3$ as illustrated in the left figure. Note that $(1,1)\A_2 \neq (2, 1)\A_3$ but $(1,1)\A_1 =(1,1)= (2, 1)\A_3$. Lemma \ref{coro_square} states that such equation determines the existence of $glb$ of those pairs. In general, existence of a common lower bound guarantees the existence of $glb$.

\begin{figure}[ht]
\begin{center}
\includegraphics[width=1.0\textwidth]{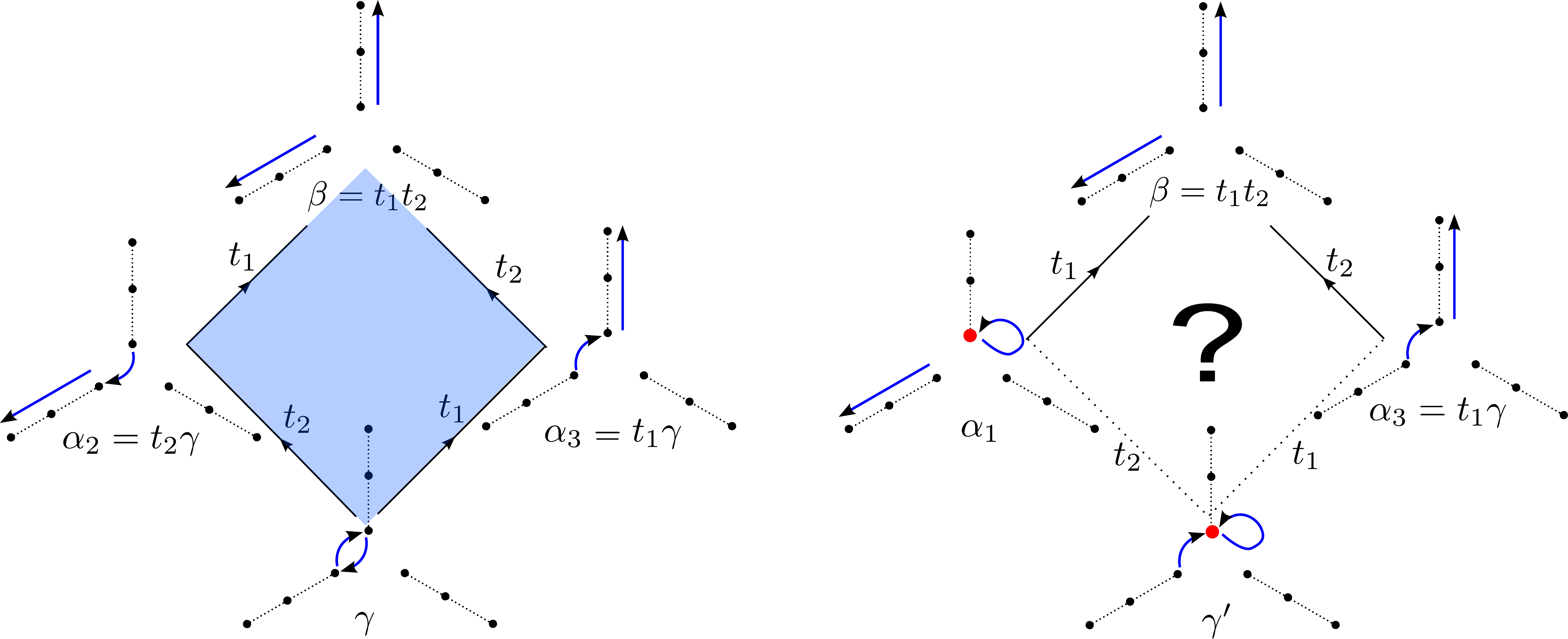}
\caption{Examples $\alpha_1, \alpha_2, \alpha_3$ where $\alpha_2 \vee \alpha_3 = \gamma$ exists, but $\alpha_1 \vee \alpha_3$ does not exist.} \label{fig_glb_square}
\end{center}
\end{figure}



\begin{lemma}\label{lm_glb}
For any $\A_1, \A_2 \in \M_n$, $\A_1 \vee \A_2$ exists and unique if there is a lower bound of $\A_1$ and $\A_2$.
\end{lemma}
\begin{proof}Suppose $\gamma $ is a lower bound of $\A_1$ and $\A_2$. Then there exist $\T =t_1^{k_1}, \cdots, t_n^{k_n}$ and $\T'=t_1^{k'_1}, \cdots, t_n^{k'_n}$ satisfying
$$\A_1=\T \gamma \;\, \text{and}\;\, \A_2 = \T' \gamma.
$$

\textbf{Claim:} $\B:=t_1^{m_1}\cdots t_n^{m_n}\gamma$ is the \emph{glb}, where $m_i=min\{k_i, k'_i\}$. \\
The commutativity of $\TT_n$ implies that
$$
\A_1= t_1^{l_1}, \cdots, t_n^{l_n} \B\;\, \text{and}\;\, \A_2= t_1^{l'_1}, \cdots, t_n^{l'_n}\B
$$
where $l_i = max\{k_i - k'_i, 0\}$ and $l'_i = max\{k'_i - k_i, 0\}$. So $\B$ satisfies the first condition. Observe that $l_i l'_i=0$ for each $i$. This means that either
$$
\A_1{_{|R_i}} = \B_{|R_i} \;\, \text{or}\;\, \A_2{_{|R_i}} = \B_{|R_i}
$$
for each $i$. Suppose $\B'$ is another element with $\A_1 = \T_1\B'$ and $\A_2 = \T_2\B'$ for some $\T_1, \T_2 \in \TT_n$. Then we have
$$
(\T_1\B')_{|R_i} = \B_{|R_i} \;\,\text{or}\;\,(\T_2\B')_{|R_i}= \B_{|R_i}
$$
for each $i$. So $\B' \leq \B$. Uniqueness of the glb follows from the second condition in Definition \ref{def_glb} together with antisymmetry of $\geq$ (Proposition \ref{prop_porder}).
\end{proof}

\begin{corollary}\label{coro_glb_T_n}
If $\A_1, \A_2 \in \TT_n$ then $\A_1 \vee \A_2$ always exists.
\end{corollary}
\begin{proof}
If two elements $\A_1$ and $\A_2$ belong to $\TT_n$, then they have a common lower bound $1_{\M_n}$, namely the identity map on $Y_n$. By Lemma \ref{lm_glb}, $\A_1 \vee \A_2$ exists. \end{proof}

\begin{lemma}\label{lm_lub,glb}
For a pair of elements $\A_1 , \A_2 \in \M_n$, the $lub$ and $glb$ of the pair can be characterized by height function as follows.
 \begin{itemize}
    \item [(1)]$\B$ is the $lub$ of $\A_1, \A_2$ if $\B$ is an upper bound and $\B$ has the smallest height among the upper bounds;
    \item [(2)]$\B$ is the $glb$ of $\A_1, \A_2$ if $\B$ is a lower bound and $\B$ has the largest height among the lower bounds.
 \end{itemize}
\end{lemma}
\begin{proof}Suppose $\B$ is an element satisfying the condition ($1$) above. We want to show $\B = \A_1 \wedge \A_2$. By the definition of $\A_1 \wedge \A_2$ it is obvious that $h(\B) = h(\A_1 \wedge \A_2)$. By Lemma \ref{lm_glb}, $\B \vee (\A_1 \wedge \A_2)$ exists since $\B $ and $ \A_1 \wedge \A_2$ have common lower bounds $\A_1$ and $\A_2$. Note that $\B \vee (\A_1 \wedge \A_2)$ is an upper bound of $\A_1$ and $\A_2$. If $\B \neq \A_1 \wedge \A_2$ then $h(\B \vee (\A_1 \wedge \A_2)) < h(\B)$, contradicting to our choice of $\B$. So $\B = \A_1 \wedge \A_2$.
Analogous argument can be applied to the case of $glb$. One can draw a contradiction by assuming $\B \neq (\A_1 \vee\A_2 )$ for an element $\B$ satisfying the condition of ($2$) above.
\end{proof}

The following specifies the relationship between $lub$ and $glb$ under assumption on existence of $glb$.
\begin{proposition}\label{prop_big_square}
Suppose, for $\A_1, \A_2 \in \M_n$, $\T \A_1 = \A_1 \wedge \A_2 = \T' \A_2$ where $\T = t_1^{k_1} \cdots t_n^{k_n}$ and $\T' = t_1^{k'_1} \cdots t_n^{k'_n}$. Then $\alpha_1 \vee \alpha_2$ exists if and only if the following conditions are satisfied for all $i = 1, 2, \cdots, n$
\begin{itemize}
    \item there is no common generater $t_i$ in $\T$ and $\T'$, i.e., $k_i k'_i= 0$,
    \item $\displaystyle{(\bigcup_{p \leq k_i}(i,p)\A_1 ) \cap  (\bigcup_{ p \leq k'_i}(i,p)\A_2) = \emptyset}$.
\end{itemize}Moreover, we have
\begin{equation}\label{eq_glb}
\T' (\A_1 \vee \A_2)= \A_1 \,\;\text{and}\;\, \T (\A_1 \vee \A_2)= \A_2.
\end{equation}
\end{proposition}

\begin{proof}
First we construct $\CC=\A_1 \vee \A_2$ by using the conditions. Define $\CC$ by \begin{equation*}
(i,p)\CC =
    \begin{cases}
     (i,p)\A_1   & \text{if  } k'_i=0 \\
      (i,p)\A_2& \text{if  } k_i=0
    \end{cases}
\end{equation*}for each $i=1, \cdots, n$. This is well defined since $k_i k'_i= 0$, and if $k_1 = k'_i=0$ then we have $(i,p)\A_1 = (i,p)\T \A_1 =(i,p)\T' \A_2=(i,p)\A_2$ for all $p \in N$. One can verify $\CC$ is the desired element as follows.

\textbf{Claim 1:} $\T' \CC = \A_1$ and $\T \CC = \A_1$.\\
Fix $i$. Suppose $k'_i = 0$. For all $p \in N$, we have $(i,p) \T' \CC = (i,p) t_i^{k'_i} \CC =(i,p)  \CC = (i,p) \A_1$. Suppose $k'_i \neq 0$. This means $k_i = 0$ and so we have
\begin{align*}
(i,p) \T' \CC &=  (i,p) t_i^{k'_i} \CC =(i,p+k'_i)  \CC = (i,p+k'_i) \A_2 =  (i,p) t_i^{k'_i} \A_2\\
& = (i,p) \T' \A_2= (i,p) \T \A_1 =(i,p) t_i^{k_i} \A_1 = (i,p)  \A_1,
\end{align*} for all $p \in N$. Since the above identities hold for every $i$, we have verified $\T' \CC = \A_1$. A similar argument shows that $\T \CC = \A_1$.

\textbf{Claim 2:} $\CC \in \M_n$. We need to verify injectivity. The image $(Y_n) \CC$ coincides with $(Y_n) \A_1$ up to a finite set. We already know that $(i,p)  \CC = (i,p) \A_1$ if $k'_i= 0$ ($p \in \N$). If $k'_i \neq 0$ then $(i,p+k'_i)  \CC = (i,p)  \A_1$ ($p \in \N$). If suffices to show, for $i$ with $k'_i \neq 0$, that $(i,p)\CC \notin (Y_n) \A_1$ for all $p \leq k'_i$. This is clear from the second condition of the Lemma because $(i,p) \CC\in \bigcup_{ p \leq k'_i}(i,p)\A_2$.

\begin{figure}[ht]
\begin{center}
\includegraphics[width=.3\textwidth]{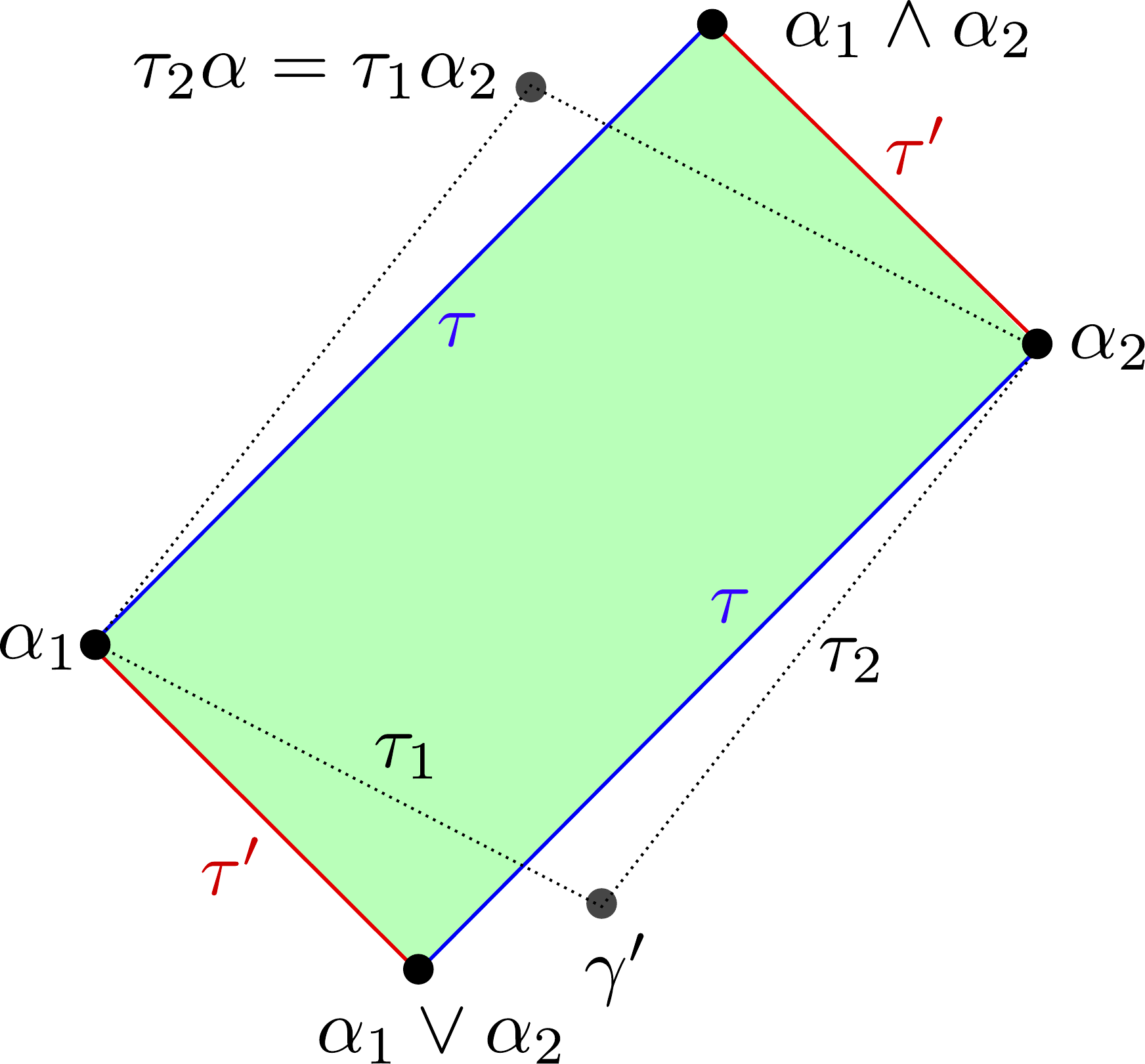}
\caption{`Big rectangle': if $\A_1 \vee \A_2$ exists $\A_1 \wedge \A_2$ determines $\A_1 \vee \A_2$, vice versa.} \label{fig:big_rectangle}
\end{center}
\end{figure}

\textbf{Claim 3:} $\CC = \A_1 \vee \A_2$. The lower bound $\CC$ has the largest height among lower bounds of $\A_1$ and $\A_2$. If there was another lower bound $\CC'$ of $\A_1$ and $\A_2$ with $h(\CC') > h(\CC)$ then there would be $\T_1, \T_2 \in \TT_n$ such that
$$
\T_1 \CC' = \A_1,\; \T_2 \CC' = \A_2\; \text{and}\; |\T_1|< |\T'|.
$$ This implies $\A_1$ and $\A_2$ have an upper bound $\T_2\A_1 = \T_1\A_2$ such that $h(\T_1\A_2)=(\T_1)+h(\A_2)$ is strictly smaller than $h(\A_1 \wedge \A_2) = h(\T') +h(\A_2)$ (see Figure \ref{fig:big_rectangle}). This contradicts the fact that $\A_1 \wedge \A_2$ has the smallest height among upper bounds of $\A_1$ and $\A_2$ (Lemma \ref{lm_lub,glb}). So the lower bound $\CC$ attains maximum height among lower bounds of $\A_1 $ and $ \A_2$. By Lemma \ref{lm_lub,glb}, $\CC = \A_1 \vee \A_2$.

The converse statement is rather easy to check. Suppose that $\A_1 \vee \A_2$ exists and that
\begin{equation}\label{eq_glb_1}
\T_1 (\A_1 \vee \A_2)= \A_1 \,\;\text{and}\;\, \T_2 (\A_1 \vee \A_2)= \A_2.
\end{equation} for some $\T_1, \T_2\in \TT_n$. As in the proof of Lemma \ref{lm_glb}, we see that there is no common generator in $\T_1$ and $\T_2$. So the first condition is satisfied. Injectivity of $\CC$ together with equation (\ref{eq_glb_1}) guarantees the second condition.

Finally, the argument in Claim $3$ above says that $\A_1 \vee \A_2$ determines $\A_1 \wedge \A_2$. If $\A_1 \vee \A_2$ satisfies equation (\ref{eq_glb_1}) then we have
$$
\T_1 \T = \T_2 \T'.
$$ Now the condition that there is no common generator in pairs $\T$ and $\T'$, and $\T_1$ and $\T_2$ enables one to conclude that $\T_1 = \T' $ and $\T_2 = \T$. Hence equation (\ref{eq_glb}) follows. \end{proof}

\begin{definition}[\textbf{Maximal elements}]
For $\B \in \M_n$, $\A$ is called a \emph{maximal element} of $\B$ if it satisfies
\begin{equation}\label{eq_identity_down}
\B = t_i\A
\end{equation}
for some generator $t_i $ of $\TT_n$.
\end{definition}

The following Lemma says the number of maximal elements of $\B$ varies depending on the height of $\B$.

\begin{lemma}\label{lm_down_for_each_i}
Suppose $h(\B)=  h$. There exists $n \times h$ many maximal elements of $\B$.
\end{lemma}

\begin{proof}
Fix $i$. The identity (\ref{eq_identity_down}) determines $\A$ completely except $(i,1)\A$. It also implies $(Y_n -(i,1))\A = (Y_n)\B$. So there are precisely $|S(\B)|=|Y_n - (Y_n )\B| =h$ many choices for $(i,1)\A$. It turns out any of these choice defines an injective map $\A$. Consider elements $\A_1, \cdots, \A_h$ defined by
\begin{equation*}
(j,p)\A_k =
    \begin{cases}
        (j,p)\B& \text{if } j\neq i, \;p \in \mathds{N} \\
        (i,p+1)\B & \text{if } j =i,\;p \in \mathds{N},
    \end{cases}
\end{equation*}and $(i,1)\A_k \in S(\B)$. They are all desired $h$ many injective maps.

\end{proof}

\begin{remark}\label{rmk_maxima_1-1_corres} The argument of Lemma \ref{lm_down_for_each_i} implies that, each maximal element $\A_k$ of $\B$ is labelled by $t_i$ as well as a point of $S(\B)$. Figures \ref{fig_exmaple_lub} and \ref{fig_glb_square} illustrate this idea for $\B = t_1 t_2$. Note that $S(\B) = \{(1,1),(1,2)\}$. In those figures maximal elements $\A_1$, $\A_2$ and $\A_3$ are labelled by
$$
\A_1 \leftrightarrow (t_1,\,(1,1)), \;\A_2 \leftrightarrow (t_1,\,(1,2)) , \;\A_3 \leftrightarrow (t_2,\,(1,1)).
$$
In general, for $\B \in \M_n$, the set of maximal elements of $\B$ is in $1$-to-$1$ correspondence with $\{1, \cdots, n\} \times S(\B)$.
\end{remark}

The corollaries below follows from Proposition \ref{prop_big_square}. They provide criterions when a collection of maximal elements form squares and $k$-cubes respectively.
\begin{corollary}\label{coro_square}
Suppose $t_i\A_1 = t_j\A_2$ ($i \neq j$). There exists $\A_1 \vee \A_2$ if and only if the two first coordinates are distinct and the second coordinates are distinct (using the ordered pair notation introduced in Remark \ref{rmk_maxima_1-1_corres}). 
\end{corollary}

\begin{proof}
This is a special case of Proposition \ref{prop_big_square}, with the two conditions of that proposition rephrased in terms of the coordinates introduced in Remark \ref{rmk_maxima_1-1_corres}.
\end{proof}

\begin{corollary}\label{coro_cube}
Suppose $\{\A_1, \A_2, \cdots, \A_k\}$ is a collection of maximal elements of $\B$. There exists $glb(\A_1, \cdots, \A_k)$ if and only if the first coordinates are all distinct and the second coordinates are all distinct (using the ordered pair notation introduced in Remark \ref{rmk_maxima_1-1_corres}).
\end{corollary}

\begin{proof}We construct $glb$ inductively with the base case $k=2$, which is done in Corollary \ref{coro_square}. Assume the statement holds for $k$ maximal vertices of $\B$. Say the glb of those vertices is $\CC$. Note that the information of $k$ vertices encoded in the coordinates completely determine $\CC$. Now apply Proposition \ref{prop_big_square} to $\CC $ and $ \A_{k+1}$ to check the existence of $\CC \vee \A_{k+1}$. Note that $\CC \wedge \A_{k+1} = \B$. The condition on distinct first coordinates implies there is no common generater $t_i$ in the two paths. The condition on the other coordinates implies the second condition of Proposition \ref{prop_big_square}. So $\CC \vee \A_{k+1}$ exists.

Conversely, the existence of glb directly implies there exists a cube whose top vertex is $\B$. So the condition on coordinates is satisfied.
\end{proof}

\subsection{$X_n$ is CAT(0)}\label{section_CC}

In this section we show $X_n$ is a cubing for each $n\in N$, i.e., $X_n$ is a $1$-connected non positively curved cubical complex. There are several ways to think about this result.

By Lemma \ref{lm_X_n_1-conn}, $X_n$ is simply connected. In Subsection \ref{section_gromov}, we prove $X_n$ is non positively curved by using Gromov's $link\;condition$.

Another way to see $X_n$ is a cubing is to use the fact (\cite{Chepoi}, \cite{Roller}) that there is $1$-to-$1$ correspondence
$$
\text{the class of cubings} \;\leftrightarrow\; \text{the class of median graphs}
$$
One associates a median graph to a cubing by considering the $1$--skeleton of the cubing. In the reverse direction, one thinks of a median graph as the $1$--skeleton of a cubing, and defines the cubing inductively on skeleta, as in Definition \ref{dfn:X_n}. In Subsection \ref{sec_median} we prove $\C_n$ is a median graph for all $n \in \N$.

In \cite{Sageev}, Sageev provided explicit construction of a cubing associated to a multi-ended group. In Subsection \ref{sec_Sageev_construction}, we discuss an interpretation of $X_n$ in terms of Sageev's construction.

\subsection{Gromov link condition}\label{section_gromov}
\begin{definition}[\textbf{cubical complex}]
Intuitively, a cubical complex is a regular CW-complex, except that it is built out of Euclidean cubes $I^k$ instead of balls. More precisely, a \emph{cubical complex} $X$ is a CW-complex where for each $k$-cell $\sigma^k_j \subset X$ its attaching map $\varphi_j : \partial I^k \to X^{n-1}$ satisfies the following conditions:
\begin{itemize}
    \item[(1)] the restriction of $\varphi_j$ to each face of $I^k$ is a linear homeomorphism onto a cube of one lower dimension,
    \item[(2)] $\varphi_j$ is a homeomorphism onto its image.
\end{itemize}We give $X$ the standard CW-topology.
\end{definition}

The non-positive curvature condition we will use is a local condition captured by conditions on the link of a vertex.

\begin{definition}[\textbf{Link of a vertex}]
The \emph{link of a vertex $v$ in $I^k$}, denoted by $Lk(v,I^k)$, is defined to be intersection of the cube $I^k$ and the unit sphere $\mathds{S}^{k-1}$ centered at $v$ with respect to $L^1$ metric. Note that $Lk(v,I^k)$ is the standard simplex of dimension $k-1$. For a vertex $\A \in X$ and a cell $\Si^k_j$ containing $\A$, the \emph{link of $\A$ in $\sigma^k_j $} is defined to be the image $\varphi_j (Lk(v,I^k))$, where $\varphi_j (v)=\A$.
The \emph{link of $\A$ in $X$}, denoted by $Lk(\A,X)$, is defined to be the union of all links of $\A$ in cells containing $\A$.
\end{definition}

\begin{definition}[\textbf{Ascending/Descending links of a vertex}]
For a vertex $\A$ of a cubical complex $X$ equipped with a Morse function $h$, the descending link $Lk_\downarrow(\A, X)$ is defined by
$$
Lk_\downarrow(\A, X) = \bigcup \{ \varphi_j(Lk(v,\sigma^k_j): h \varphi_j\;\text{attains maximum at}\; v, \; \varphi_j(v)=\A\}.
$$
Likewise, the ascending link $Lk_\uparrow(\A, X)$ is defined by
$$
Lk_\uparrow(\A, X) = \bigcup \{ \varphi_j(Lk(v,\sigma^k_j): h \varphi_j\;\text{attains minimum at}\; v, \; \varphi_j(v)=\A\}.
$$
\end{definition}

\begin{definition}[\textbf{Flag complex}]
A simplicial complex $L$ is said to be a \emph{flag complex} if every collection of pairwise adjacent vertices of $L$ spans a simplex of $L$.
\end{definition}

\begin{definition}[\textbf{Gromov's condition}]
A cubical complex $X$ satisfies the Gromov's condition if $Lk(\A, X)$ is a flag complex for each $\A \in X^{(0)}$.
\end{definition}

Following Gromov, we say that a cubical complex $X$ is \emph{non positively curved} (simply NPC) if it satisfies the Gromov's condition. 

From the way we constructed $X_n$, we see that any vertex of $\dlk{\A, X_n}$ corresponds to a vertex $\B$ such that $\A = t_i \B$ for some generator $t_i \in \TT_n$. By Remark \ref{rmk_maxima_1-1_corres}, there exists a bijection
$$
 \{\B \mid \A = t_i \B\;\text{for some}\; t_i\} \leftrightarrow \{i \mid 1 \leq i \leq n\} \times S(\A)
$$
Identify the later set with  $\{(i,k) \in  \mathds{N} \times \mathds{N} \mid 1\leq i \leq n,\; 1\leq k \leq h(\A) \} $. Under this (composition of) identification, points $\{(i,k)\mid 1\leq k \leq h(\A)\mid \} $ correspond to vertices in $\dlk{\A, X_n}$ which are joined to $\A$ by edges labelled by $t_i$. Consider the simplicial complex $L_{n,h}$ for $h\in \mathds{N}$ defined by
\begin{itemize}
    \item [($1$)] vertices: $L_{n,h}^{(0)} =\{(x,y) \in  \mathds{N} \times \mathds{N} \mid 1\leq i \leq n,\; 1\leq k \leq h \} $
    \item [($2$)] simplexes: A collection of vertices $\{(x_0, y_0), \cdots, (x_k, y_k)\}$ form a $k$-simplex if $x_i$'s are all distinct and $y_i$'s are all distinct.
\end{itemize}

\begin{remark}\label{dlink_ln,h}
By Corollary \ref{coro_cube}, we see that $L_{n,h}$ and $Lk_\downarrow (\A, X_n)$ are identical if $h = h(\A)$. Note that $L_{n,h}$ is flag for all $n, h \in \mathds{N}$. If two vertices of $L_{n,h}$ are connected by an edge, their first coordinates are distinct and second coordinates are distinct. This means that any collection $C \subset L_{n,h}^{(0)}$ of pairwise adjacent vertices satisfy the condition ($2$) of definition of $L_{n,h}$ above. So a collection $C$ forms a simplex of $L_{n,h}$.
\end{remark}

\begin{lemma}\label{link}
For every vertex $\A \in X_n$, $Lk (\A, X_n)$ is flag.
\end{lemma}

\begin{proof}Suppose $C $ is a collection of pairwise adjacent vertices of $Lk(\A, X_n)$. We may assume $C =\{z_1, \cdots, z_k\} \cup \{z'_1, \cdots, z'_{k'}\}$ where $z_i \in Lk_\uparrow (\A, X_n)$ and $z'_j \in Lk_\downarrow (\A, X_n)$. For each vertex $\A \in X_n$, there exists unique $n$-cube having $\A$ as the bottom vertex. So $\alk{\A, X_n}$ is simply the standard $(n\!-\!1)$-simplex. By Remark \ref{dlink_ln,h} $Lk_\downarrow (\A, X_n)$ is also flag. So those subcollections of $C$ form a $(k\!-\!1)$-simplex $\Si$ and a $(k'\!-\!1)$-simplex $\Si'$ respectively in the ambient complexes.

Our claim is that there exist $(k+k')$-cube which contributes $(k+k'-1)$ simplex $\Si \ast \Si'$ to $Lk(\A, X_n)$. Consider the ordered pair notation introduced in Remark \ref{rmk_maxima_1-1_corres} for $\{z_1, \cdots, z_k\}$. Let $T_1 =\{t_{i_1}, \cdots, t_{i_k}\} \subset \{t_1, \cdots, t_n\}$ denote the set consisting of first coordinates of vertices $\{z_1, \cdots, z_k\}$. Likewise let $T_2 =\{t_{j_1}, \cdots, t_{j_{k'}}\}$ be the set corresponding to $\{z'_1, \cdots, z'_{k'}\}$. Observe that the $(k\!-\!1)$-simplex $\Si \subset Lk_\downarrow (\A, X_n)$ corresponds to $k$-cube $\rho$ with top vertex $\A$, which is generated by $T_1$. Similarly, $\Si' \subset Lk_\uparrow (\A, X_n)$ corresponds to $k'$-cube $\rho'$ which is generated by $T_2 $, whose bottom vertex is $\A$. Since $T_1$ and $T_2$ disjoint, the $k$-cube $\rho$ together with edges labelled by $T_2$ spans a $(k+k')$-cube. This cube is the desired cube containing $\rho$ and $\rho'$, and so $\A$ as well. So a collection $C$ forms a simplex $\Si \ast \Si'$ in $Lk (\A, X_n)$.
\end{proof}

By Lemma \ref{lm_X_n_1-conn} and Lemma \ref{link} we have the following Theorem.
\begin{theorem}
For all $n \in \mathds{N}$, $X_n$ is a cubing.
\end{theorem}

The following is known facts for a cubing, and proofs can be found in \cite{bible}.
\begin{corollary}\label{coro:contrac_unique_geod}
For all $n \in \mathds{N}$, $X_n$ is uniquely geodesic and is contractible.
\end{corollary}

\subsection{Cayley graph $\C_n$ is a median graph}\label{sec_median}

For an edge $e$ of $\C_n$, let $\partial_{\!-} e$ and $\partial_{\!+} e$ denote the initial and terminal vertices of $e$. Although generators of $\TT_n$ do not have inverses, we can consider the $reverse$ edge of $e$. Let $\overline{e}$ denote the reverse edge of an edge $e$, i.e.,  $\partial_{\!-} \overline{e} =\partial_{\!+} e$ and $\partial_{\!+} \overline{e}=\partial_{\!-} e$. Note that $\overline{\overline{e}} = e$.

By a \emph{path} we mean an edge path; concatenation of edges (including reverse edges) $e_1 , \cdots, e_k$ where $\partial_+(e_i) = \partial_-(e_{i+1})$ for $i=1, 2, \cdots, k-1$. If a path $p$ is a concatenation of edges $e_1, \cdots, e_k$ in order, we write this path as $p = e_1 \cdots e_k$ and the \emph{reverse} path $\overline{p}$ of $p$ is defined by $\overline{p} = \overline{e}_{k} \cdots \overline{e}_1$.

A path is called \emph{ascending} if it does not contain reverse edges. Similarly a path is called \emph{descending} it consists of only reverse edges. Obviously the reverse path of an ascending path is descending and vice versa. Suppose $p$ is an ascending path joining $\A$ to $\T \A$ for some $\T \in \TT_n$. One reads the edge labeling of $p$ in order to form a word $\T= \T_1 \cdots \T_k$, $\T_i \in \{t_1, \cdots, t_n\}$. The expression for $\T$ is not unique in ($\TT_n$) because of commutativity of $\TT_n$. Indeed permutations of generators in $\T$ produce a class of paths rel two end vertices $\A$ and $\T \A$. In the sequel, a choice of such path is less relevant. Instead we will be interested in two end vertices of an ascending path. When we say a path $p$ $given \,\;by$ $\T \in \TT_n$, we mean a choice from the class of paths rel two end vertices $\A$ and $\T\A$, and write $p\simeq\T$.

For a pair of vertices $\A$ and $ \B $, the \emph{distance in} $\C_n$ is the smallest length of paths joining them. We denote the distance by $d(\cdot,\cdot )$. A \emph{geodesic} $[\A, \B]$ joining vertices $\A$ and $\B$ is a path whose length is $d(\A, \B)$.

\begin{remark}\label{rmk_asc_des_height}
Note that $d(\A, \B) \geq |h (\A)- h( \B)|$. Every ascending/descending path is a geodesic. If $p$ is an ascending/descending path joining $\A$ to $\B$ then $d(\A,\B) = |h(\B)-h(\A)|$.
\end{remark}

\medskip
\noindent
{\bf Standard geodesics in $\C_n$.}

We say a path $p$ has a $turn$ if $p$ contains $\overline{e}_{t_i} {e}_{t_j}$ for some $i,j$. One can apply liftings (finitely many times) to transform $p$ to a path of the form
\begin{equation}
{e}_{t_{i_1}} \cdots {e}_{t_{i_k}} \overline{e}_{t_{j_1}}\cdots \overline{e}_{t_{j_\ell}} \;\;(k, \ell \geq 0).
\end{equation}
Such path is called \emph{standard}, i.e., a path is standard if it is a concatenation of one ascending path and one descending path where the ascending path occurs first.
Figure \ref{fig_liftings} illustrates liftings of paths. Note that the length of paths does not increase under liftings. As before, for a standard path $p$ given by $(\T )(\overline{\T}')$ joining $\A$ to $\B$, there are many expressions (determined by permutations in each parenthesis). However every such choice shares important information: initial vertex $\A$, terminal vertex $\B$ and the top vertex. We write $p\simeq\T \cdot \overline{\T}'$ and we mean $p$ is a concatenation of one ascending path and one descending path which are determined by choices in $\T$ and $\T'$ respectively.     

\begin{figure}[h]
\includegraphics[width=1.0\textwidth]{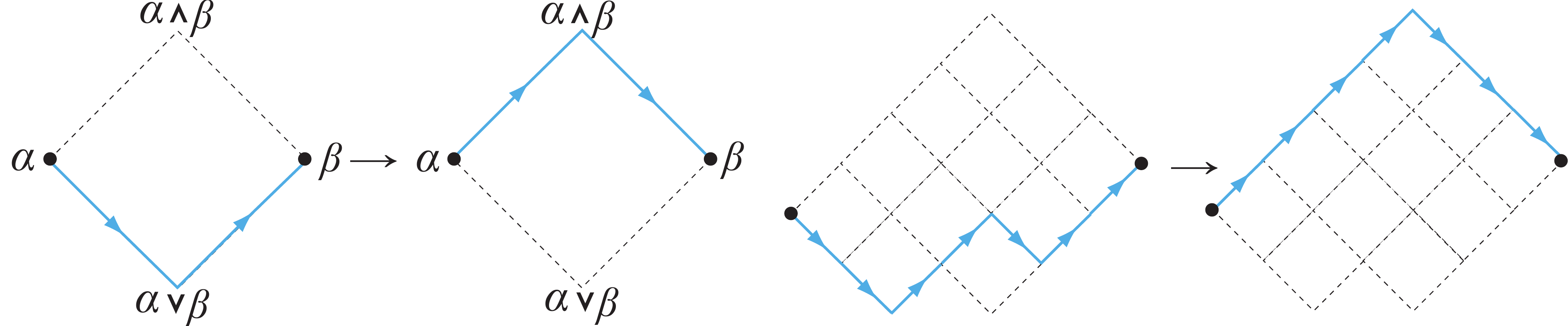}
\caption{ Every path can be transformed to a standard path via liftings} \label{fig_liftings}
\end{figure}

\begin{remark}\label{rmk:standard_geodesic}
Note that the top vertex of a standard path is a common upper bound of the two end vertices.  By Lemma \ref{lm_lub,glb}, a standard path $\pi \cdot \overline{\pi'}$ joining $\A$ to $\B$ is a geodesic if and only if $\pi \A  = \A \wedge \B = \pi'\B $.
\end{remark}
A standard path is called a \emph{standard geodesic} if it is a geodesic.

\begin{definition}[\textbf{Median graph}]
The \emph{geodesic interval} $\I[\A, \B]$ is the collection of vertices lying on geodesics $[\A, \B]$. A graph is called a \emph{median} if, for each triple of vertices $\A, \B, \gamma,$ the geodesic intervals $[\A, \B], [\B, \gamma]$ and $[\gamma,\A ]$ have a unique common point.
\end{definition}

\begin{proposition}\label{prop:inner_geodesic}
$\gamma \in \mathcal{I}[\A, \B]$ if and only if $d(\A, \B)=d(\A, \gamma)+d(\gamma, \B)$.
\end{proposition}

\begin{proof}Suppose $\gamma \in \mathcal{I}[\A, \B]$. There exists a geodesic joining $\A$ and $\B$ that contains $\gamma$. So $\gamma$ determines two subpaths which are geodesics joining $\A$ to $\gamma$, and $\gamma$ to $\B$ respectively. Thus $d(\A, \B)=d(\A, \gamma)+d(\gamma, \B)$. Conversely, if $d(\A, \B)=d(\A,\gamma)+d(\gamma,\B)$ then any concatenation of two geodesics $[\A,\gamma]$ and $[\gamma,\B]$ is again a geodesic containing $\gamma$ and hence $\gamma \in \mathcal{I}[\A, \B]$.
\end{proof}

\begin{lemma}\label{lm:interior_median}
If $\gamma \in \mathcal{I}[\A, \B]$ then
$$
\mathcal{I}[\A, \B]\cap\mathcal{I}[\B, \gamma]\cap\mathcal{I}[ \gamma,\A] = \{\gamma\}.
$$
\end{lemma}
\begin{proof}
Obviously $\gamma$ is in the intersection if $\CC \in \mathcal{I}[\A, \B]$. Suppose the intersection contains other point $\gamma'$. From Proposition \ref{prop:inner_geodesic}, we have
\begin{align*}
d(\A, \B) & =d(\A,\gamma)+d(\gamma,\B)\\
 & =d(\A,\gamma')+d(\gamma',\gamma)+d(\gamma,\gamma')+d(\gamma',\B) &\; (\CC' \in \mathcal{I}[\A, \CC], \CC' \in \mathcal{I}[\B, \CC] )\\
& = d(\A,\gamma')+d(\gamma',\B) +2d(\gamma',\gamma) &(\CC' \in \mathcal{I}[\A, \B])\\
&=d(\A,\B) +2d(\gamma',\gamma).
\end{align*}
So $d(\gamma',\gamma)=0$ or $\gamma'=\gamma$.
\end{proof}

We will need the following two lemmas.
\begin{lemma}\label{n-rect}
Suppose $\CC, \CC' \in \I[\A,\B]$. Then $\CC \wedge \CC' \in \I[\A,\B]$. Moreover if there exists a lower bound of $\CC$ and $\CC'$ then $\CC \vee\CC' \in \I[\A,\B]$
\end{lemma}

\begin{proof}Suppose $\CC, \CC' \in \I[\A,\B]$. We first construct a geodesic joining $\A$ to $\B$ which passes $\CC \wedge \CC'$. Consider elements $\pi, \pi', \T,\T',\sigma,\sigma', \rho,\rho' \in \TT_n$ to express $\A \wedge \CC, \,\A \wedge \CC', \, \CC \wedge \B,$ and $\CC'\wedge \B$ as follows

$\pi\A = \A \wedge \CC = \T\CC$, $\pi'\A = \A \wedge \CC' = \T'\CC'$,
$\sigma\CC = \CC \wedge \B= \rho\B$, and $\rho' \CC' = \CC'\wedge \B = \rho'\B$.\\
By Remark \ref{rmk:standard_geodesic}, the following standard paths are all geodesic:
$\pi \cdot \overline{\T}$ joining $\A$ and $\CC$; $\Si \cdot \overline{\rho}$ joining $\CC$ and $\B$; $\pi' \cdot \overline{\T'}$ joining $\A$ and $\CC'$; $\Si' \cdot \overline{\rho'}$ joining $\CC'$ and $\B$.
Figure \ref{interval} illustrates those standard paths.

\begin{figure}[ht]
\begin{center}
\includegraphics[width=1.0\textwidth]{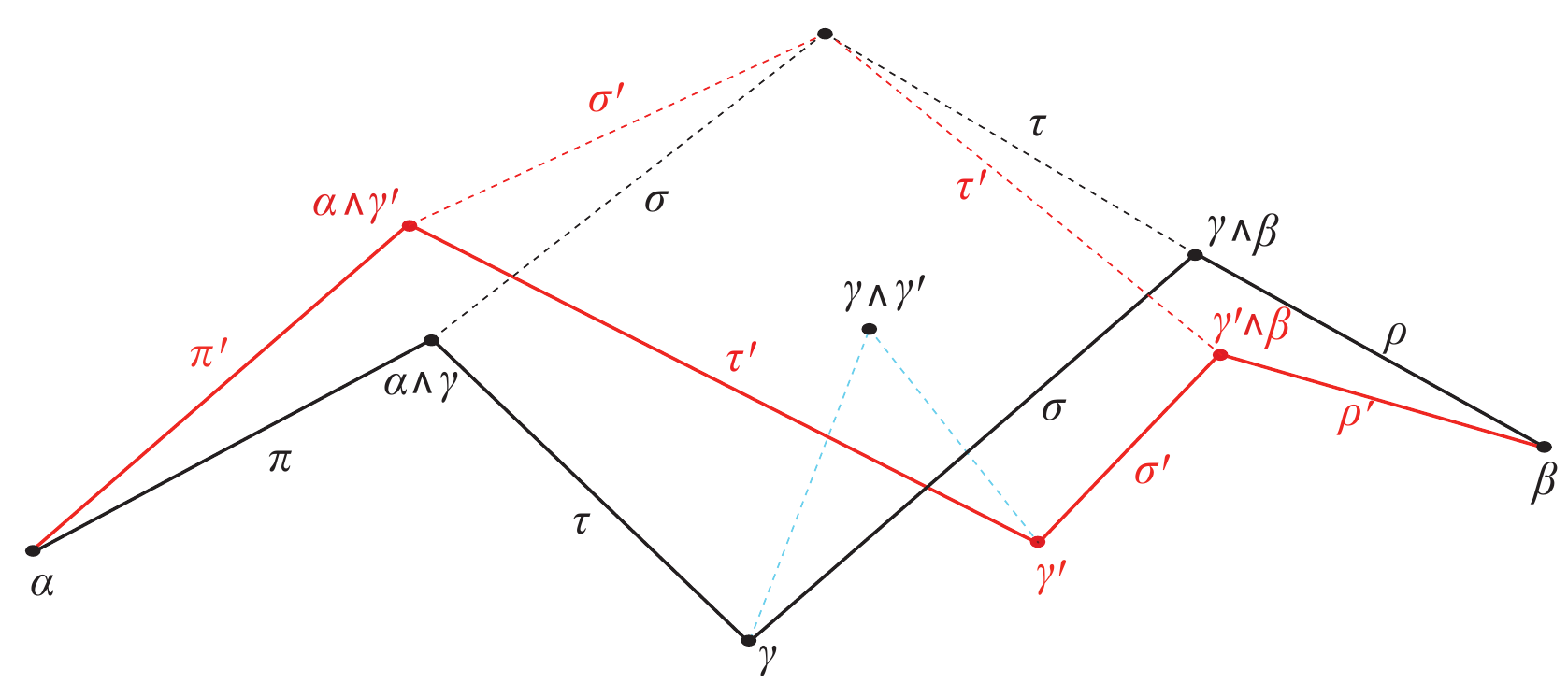}
\caption{Two geodesics $p_1 \simeq \pi \cdot\overline{\T}\cdot\Si \cdot\overline{\rho}$ and $p_2 \simeq  \pi' \cdot\overline{\T'}\cdot\Si' \cdot\overline{\rho'}$ of $\A$ and $\B$ passing $\CC$ and $\CC'$ respectively} \label{interval}
\end{center}
\end{figure}

\textbf{Claim 1}: $\CC \wedge \CC' \leq \A \wedge \B$. The standard path $\pi \overline{\T}$ connecting $\A$ to $\CC$ is a geodesic by Remark \ref{rmk:standard_geodesic}. By the same reason, the standard path $\sigma \overline{\T}$ connecting $\CC$ to $\B$ is a geodesic. Consider the path $p_1$ defined by the concatenation of $\pi \cdot \overline{\T}$ and $\Si \cdot \overline{\rho}$,
$$
p_1 \simeq \pi\cdot\overline{\T}\cdot \Si \cdot\overline{\rho}.
$$By Proposition \ref{prop:inner_geodesic}, this path $p_1$ is a geodesic since $\CC \in \I[\A,\B]$.
Similarly we have a geodesic $p_2$ of $\A$ and $\B$ which passes $\CC'$ given by
$$
    p_2\simeq \pi' \cdot\overline{\T'}\cdot \Si' \cdot\overline{\rho'}.
$$
By applying liftings to $p_1$ and $p_2$ one obtains geodesics $\tilde{p_1}$ and $\tilde{p_2}$ given by $$
    \tilde{p}_1 \simeq \pi \cdot\Si\cdot \overline{\T}\cdot \overline{\rho}, \;
    \tilde{p}_2\simeq \pi' \cdot\Si'\cdot \overline{\T'}\cdot \overline{\rho'}.
$$ From Remark \ref{rmk:standard_geodesic}, we see that the top vertex of these paths is $\A \wedge \B$. Thus $\A \wedge \B = \pi \Si\A = \T \Si \CC$ and $\A \wedge \B = \pi' \Si'\A = \T ' \Si'\CC'$. So $\A \wedge \B \geq \CC $ and $\A \wedge \B \geq \CC'$, and Claim $1$ is verified.

A chain $\CC \leq (\CC\wedge\CC') \leq (\A \wedge \B)$ allows one to decompose $\T$ and $\Si$ as $\T=\T_1 \T_2$ and $\Si = \Si_1 \Si_2$ so that
\begin{equation}\label{1}
\T_1 \Si_1\CC \!= \!\CC \wedge \CC', \;\,\T_2 \Si_2(\CC \wedge \CC')\! =\! \A \wedge \B.
\end{equation}Figure \ref{fig:interval_2} illustrates this decomposition. Consider the path connecting $\A$ to $\CC\wedge \CC'$ and then to $\B$ defined by
$$
    p_3\simeq \pi\cdot \Si_1 \cdot \overline{\T_2} \cdot\Si_2\cdot\overline{\T_1}\cdot\overline{\rho}.
$$

\begin{figure}[ht]
\begin{center}
\includegraphics[width=.8\textwidth]{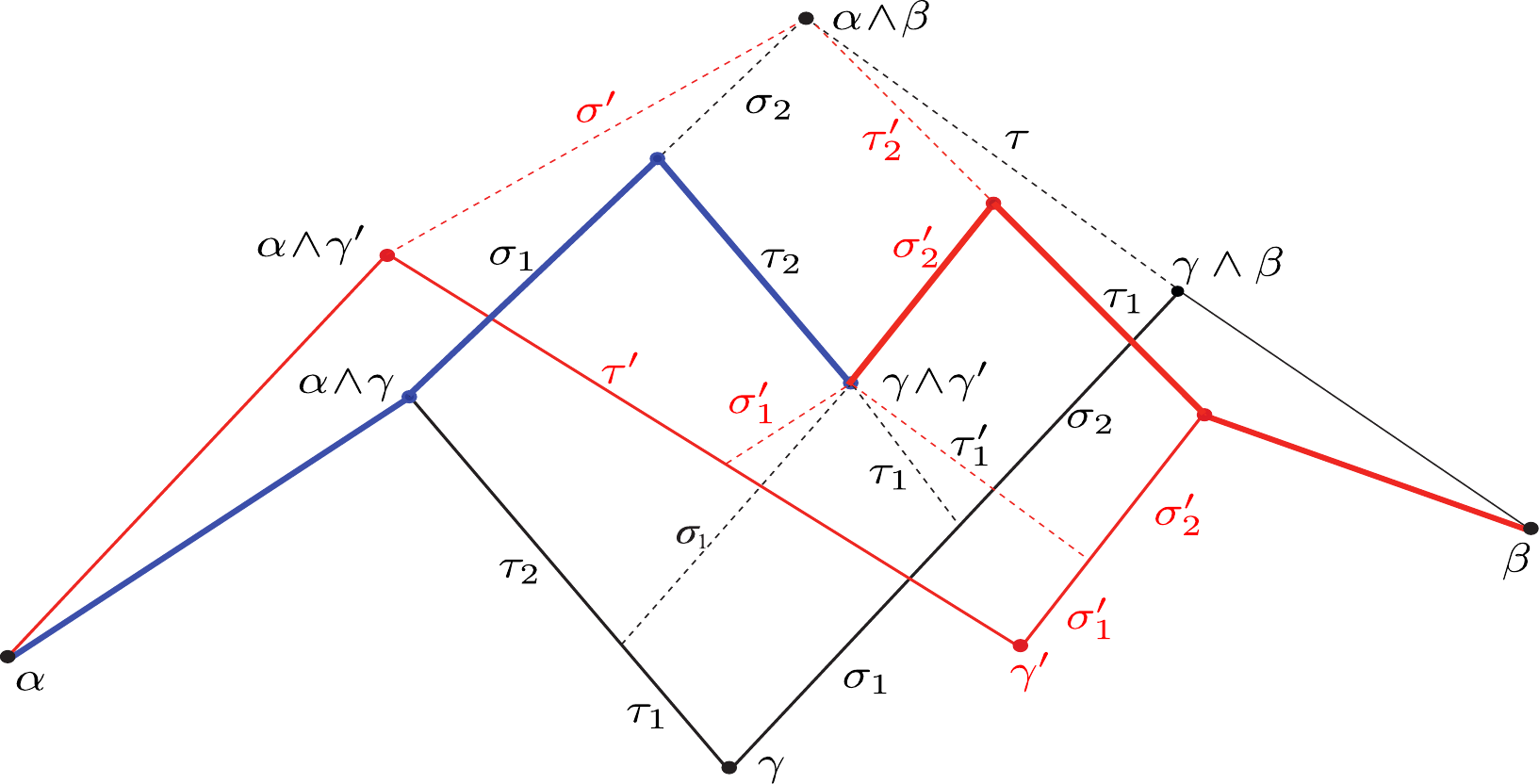}
\caption{A geodesic $p_3\simeq \pi\cdot \Si_1 \cdot \overline{\T_2} \cdot\Si_2\,\overline{\T_1}\cdot\overline{\rho}$ of $\A$ and $\B$ passing $\CC\wedge \CC'$.} \label{fig:interval_2}
\end{center}
\end{figure}
Apply appropriate liftings to the path $p_3$
to obtain the geodesic $\tilde{p_1}$. Since liftings do not increase lengths, two paths $\tilde{p_1}$ and $p_3$ have the same length. Thus $p_3$ is a geodesic joining $\A$ to $\B$, which passes $\CC \wedge \CC'$. Therefore $\CC \wedge \CC' \in \I[\A,\B]$.

For the second assertion of the Lemma, we assume the existence of $\CC \vee\CC'$. As before, use a chain $\CC' \leq (\CC\wedge\CC') \leq (\A \wedge \B)$ to get the decomposition of $\T'$ and $\Si'$; $\T'=\T'_1 \T'_2 , \; \Si'=\Si'_1 \Si'_2 $ such that
\begin{equation}\label{2}
\T'_1 \Si'_1\CC \!= \!\CC \wedge \CC', \;\,\T'_2 \Si'_2(\CC \wedge \CC')\! =\! \A \wedge \B \;\,
\end{equation}
From identities (\ref{1}) and (\ref{2}) we have $\T_2 \Si_2= \T'_2 \Si'_2$. Proposition \ref{prop_big_square} together with the existence of $\CC \vee\CC'$ implies that
$$
\T'_1 \Si'_1(\CC \vee \CC') = \CC, \; \text{and}\; \T_1 \Si_1(\CC \vee \CC') = \CC'.
$$

\textbf{Claim 2}: $\T_2 = \T'_2$ and $\Si_2 =\Si'_2$. Look at the loop $\ell_1$ formed by two ascending paths $\T \Si$ and $\Si \T$ emanating from $\CC$. We first show $\ell_1$ fits into the situation described in Proposition \ref{prop_big_square}, i.e., $(\A \wedge \CC) \wedge (\CC \wedge \B ) = \A \wedge \B$ and $(\A \wedge \CC) \vee (\CC \wedge \B ) = \CC$. This loop has $\CC$ and $\A \wedge \B$ as bottom and top vertices respectively. Recall that the path $p_1\simeq \pi\cdot\overline{\T}\cdot \Si \cdot\overline{\rho}$ is a geodesic of $\A$ and $\B$. The restriction $p_4$ of $p_1$ defined by
$$
    p_4\simeq \overline{\T}\cdot \Si
$$
is also a geodesic joining $\A \wedge \CC$ and $\CC \wedge \B $. Consider the standard geodesic $p_5$, joining $\A$ to $\B$, defined by
$$
p_5 \simeq  \pi\cdot\Si\cdot\overline{\T} \cdot\overline{\rho}.
$$
This geodesic $p_5$ also passes $\A \wedge \CC$ and $\CC \wedge \B $. So the restriction $p_6$ of $p_5$ defined by
$$
p_6\simeq \Si\cdot\overline{\T}
$$
is a standard geodesic connecting $\A \wedge \CC$ to $\CC \wedge \B $. By Remark \ref{rmk:standard_geodesic}, $(\A \wedge \CC) \wedge (\CC \wedge \B ) = \A \wedge \B$. Moreover, by Proposition \ref{prop_big_square}, $(\A \wedge \CC) \vee (\CC \wedge \B ) = \CC$.

There is another loop $\ell_2$ formed by two ascending paths $\T' \Si'$ and $\Si' \T'$ emanating from $\CC'$. This loop $\ell_2$ has $\A \wedge \B$ as top vertex and $\CC'$ as bottom vertex. By an analogous argument that we applied to the loop $\ell_1$, one can show
$$
(\A \wedge \CC') \wedge (\CC' \wedge \B ) = \A \wedge \B,  \; (\A \wedge \CC') \vee (\CC' \wedge \B ) = \CC'.
$$
Again, one can deduce the following fact from $\ell_3$, formed by two ascending paths $\T'_1 \Si'_1 \T \Si$ and $\T_1 \Si_1\Si'\T'$ emanating from $\CC\vee\CC'$.
$$
(\A \wedge \CC) \wedge (\CC' \wedge \B ) = \A \wedge \B,  \; (\A \wedge \CC) \vee (\CC' \wedge \B ) = \CC \vee \CC'.
$$
Now Proposition \ref{prop_big_square} implies that $\Si$ and $\T$ do not share a letter $t_i$. Similarly there is no common letter between $\Si'$ and $\T'$, and $\Si$ and $\T'$. So we have $\T_2 = \T'_2$ and $\Si_2 =\Si'_2$ from the identity $\T_2 \Si_2= \T'_2 \Si'_2$.

\textbf{Claim 3} $\T_1 = \Si'_1 = 1$. Apply Proposition \ref{prop_big_square} to the loop $\ell_3$ again to see that two diagonal edges of $\ell_3$ are given by the same element of $\TT_n$. In particular we have

\begin{equation}\label{3}
\Si= \T_1 \Si_1 \Si'.
\end{equation}
This means
$$
\Si = \Si_1 \Si_2 = \T_1 \Si_1 \Si'_1 \Si'_2
$$
By Claim $2$ we have $\T_1 \Si'_1 =1$. This completes the proof since the path  $p  \simeq\pi \overline{\T} \,\overline{\T'_1}\, \Si_1 \,\Si' \,\overline{\rho'}$ passes $\CC \vee \CC'$ and lifts to the geodesic $p_3$;
\begin{align*}
 p & \simeq\pi \overline{\T_2} \,\overline{\T'_1}\, \Si_1 \,\Si'_2 \,\overline{\rho'} \\
& \simeq\pi \overline{\T_2} \,\Si_1\, \overline{\T'_1}\,\Si'_2 \,\overline{\rho'}\\
& \simeq \pi\,\Si_1\,\overline{\T_2} \,\Si_2 \overline{\T'_1}\, \overline{\rho'}.
\end{align*}
\end{proof}

\begin{corollary}\label{convex}
For any vertices $\A, \B$, $\mathcal{I}[\A, \B]$ is convex in $\mathcal{C}_n^{(0)}$, i.e., for any $\CC ,\CC' \in \mathcal{I}[\A, \B]$, $\mathcal{I}[\CC, \CC']\subset \mathcal{I}[\A, \B]$.
\end{corollary}

\begin{proof}In case of $\CC \leq \CC'$, one can show that $\mathcal{I}[\CC, \CC']\subset \mathcal{I}[\A, \B]$ by using Proposition \ref{prop_big_square} and Lemma \ref{n-rect}. For the general case, use induction on the distance between $\CC$ and $\CC'$ together with Lemma \ref{n-rect} and the previous observation to the pais of vertices
$$\CC \leq \CC \wedge \CC',\; \CC'\leq \CC \wedge \CC',
\;\CC \vee \CC' \leq \CC,\; \text{and}\; \CC \vee \CC' \leq \CC'
$$.

\end{proof}

The following notion of orthant of a vertex is useful for us to examine the structure of $\mathcal{C}_n^{(0)}$ carefully. For a vertex $\A \in\mathcal{C}_n^{(0)}$, orthant $\mathcal{O}(\A)$ of $\A$ is defined by
\begin{equation}\label{orthant}
\mathcal{O}(\A)= \{\B \mid \B \geq \A\}.
\end{equation}

Note that $\mathcal{O}(\A)$ is also convex in $\mathcal{C}_n^{(0)}$ for any vertex $\A$.

\begin{lemma}\label{geodes}
Suppose $\mathcal{O}(\CC) \cap \mathcal{I}[\A, \B] \neq \emptyset$ for some $\CC \notin \mathcal{I}[\A, \B]$. There exists unique $\delta_0 \in \mathcal{O}(\CC) \cap \mathcal{I}[\A, \B]$ with smallest distance from $\CC$. Moreover $d(\A, \CC) = d(\A, \delta_0) + d(\delta_0, \CC)$.
\end{lemma}

\begin{proof}Suppose $\CC \notin \mathcal{I}[\A, \B]$ and $\mathcal{O}(\CC) \cap \mathcal{I}[\A, \B] = \{\delta_1, \cdots \delta_k\}$. Since $\CC$ is a common lower bound of those elements, there exists $glb$ of $\{\delta_1, \cdots \delta_k\}$. Let $\delta_0$ denote the $glb$. The Claim $1$ in proof of Theorem \ref{medianC_n} shows that $\delta_0$ is the unique vertex in $\mathcal{I}[\A, \B]$ with the smallest distance from $\CC$. If $d(\A, \CC) >d(\A, \delta_0) + d(\delta_0, \CC)$ then it is not difficult to show there exists another vertex $\delta' \in \mathcal{O}(\CC) \cap \mathcal{I}[\A, \B]$ with smaller distance. So $\delta_0$ satisfies the equality.
\end{proof}

\begin{theorem}\label{medianC_n}
$\C_n$ is a median graph for any $n \in \mathds{N}$.
\end{theorem}

\begin{proof}From Lemma \ref{lm:interior_median}, it suffices to show $\mathcal{I}[\A, \B]\cap\mathcal{I}[\B, \gamma]\cap\mathcal{I}[ \gamma,\A]$ is a singleton set for any $\A, \B$ and $\CC \notin \mathcal{I}[\A, \B]$.

\textbf{Case 1:} $\mathcal{O}(\CC) \cap \mathcal{I}[\A, \B] \neq \emptyset$. From Lemma \ref{geodes}, we see that there exists unique $\delta_0 \in \mathcal{O}(\CC)$ with smallest distance from $\CC$. We want to show $\mathcal{I}[\A, \B]\cap\mathcal{I}[\B, \gamma]\cap\mathcal{I}[ \gamma,\A] = \{\delta_0\}$. From Lemma \ref{geodes} (the second assertion), a concatenation of two geodesics joining $\A$ to $\delta_0$ and $\delta_0$ to $\CC$ is again a geodesic. Thus the intersection of three intervals contains $\delta_0$. Suppose the intersection contains another vertex $w$.

\begin{figure}[ht]
\begin{center}
\includegraphics[width=1.0\textwidth]{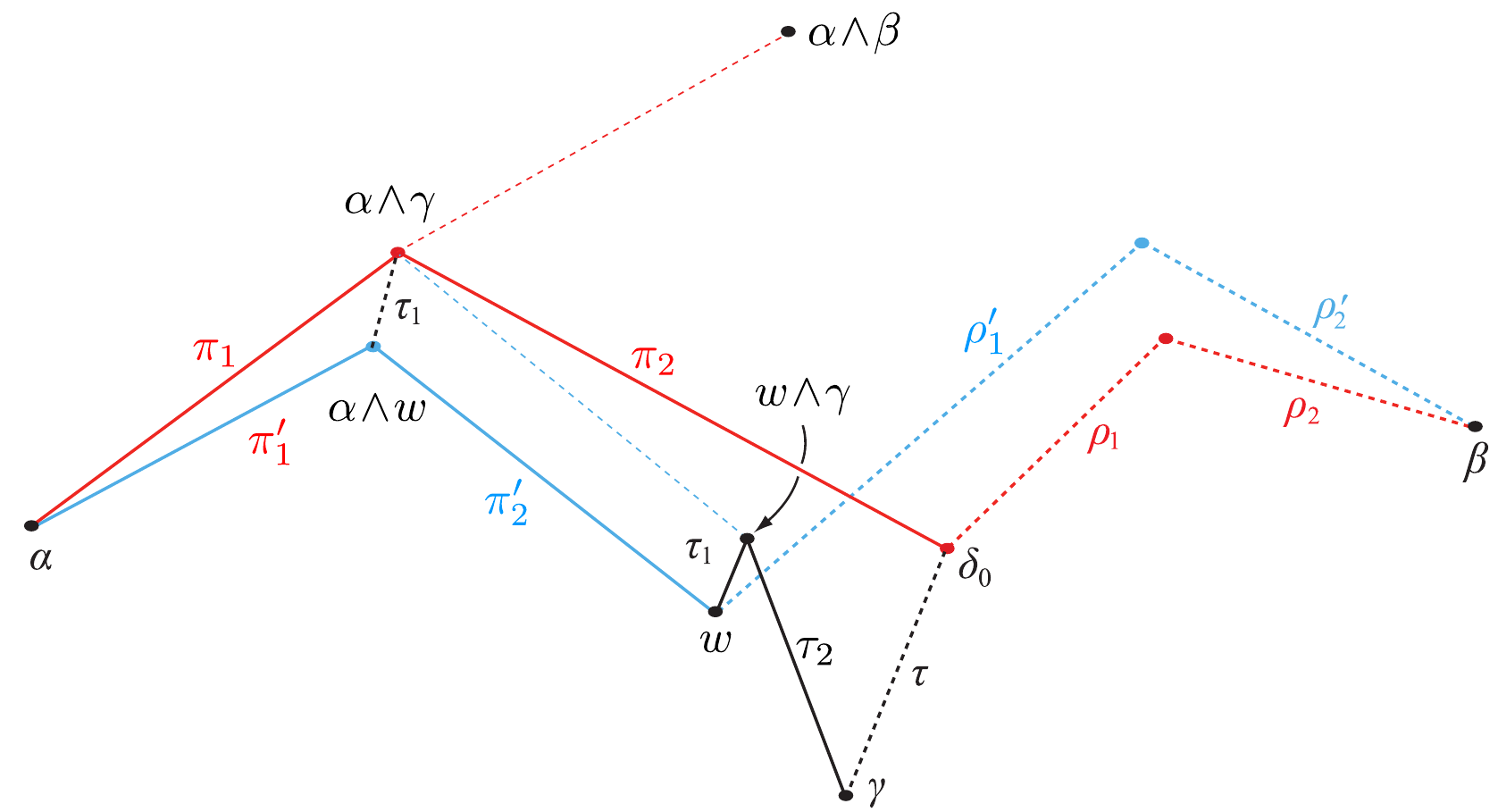}
\caption{$\delta_0 \in \mathcal{O}(\CC)$ with $d( w,\CC)>d(\delta_0,\CC)$ } \label{interval_a}
\end{center}
\end{figure}

\textbf{Claim 1:} $d( w,\CC)>d(\delta_0,\CC)$. Say $p_1, p_2, p_3$ are standard geodesics joining $\A$ to $\delta_0$, $\A$ to $w$ and $w$ to $\CC$ respectively, given by
$$
p_1= \pi_1 \overline{\pi}_2, \; p_2 = \pi'_1 \overline{\pi}'_2, \; p_3 =\T_1 \overline{\T}_2,  $$ and $\delta_0 = \T \CC $ for some $\pi_1, \pi_2,\pi'_1, \pi'_2 , \T_1, \T_2, \T \in \TT_n$. Figure \ref{interval_a} illustrates this situation for those paths $p_1$(solid red), $p_2$(solid blue) and $p_3$(solid black). Since the concatenation $p_2 p_3$ is a geodesic joining $\A$ to $\CC$, after applying liftings to $p_2 p_3$, one obtains a standard geodesic passing $\A \wedge \CC$. Consider the ascending path joining $w$ to $\A \wedge \CC$ given by $\T_1 \pi'_2$. This passes $w\wedge\CC$ an so $w\wedge\CC \in \mathcal{I}[\A, \B]$ by Corollary \ref{convex}. Since $w\wedge\CC \geq \CC$ we see that $w\wedge\CC \geq \delta_0$ and $w\wedge\CC \neq \delta_0$ by the definition of $\delta_0$. Thus $d( w,\CC)= d( w,w\wedge\CC)+d(w\wedge\CC, \CC) > |\T_2| > |\T | = d(\delta_0,\CC)$.

\textbf{Claim 2:} $w \notin \in \mathcal{I}[\A, \B]\cap\mathcal{I}[\B, \gamma]\cap\mathcal{I}[ \gamma,\A]$. Say $p_4$ and $ p_5$ are standard geodesics joining $\delta_0$ to $\B$, $w$ to $\B$ respectively given by
$$
p_4= \rho_1 \overline{\rho}_2, \; p_5 = \rho'_1 \overline{\rho}'_2.
$$

Since $\delta_0$ belongs to the intersection the concatenation $p_1 p_4$ is a geodesic with length $d(\A, \B)$. If $w$ belongs to the intersection then the concatenation $p_2 p_5$ is a geodesic with the same length. So we must have

$$
|\pi_1|+|\pi_2|+|\T| = d(\A, \CC) =|\pi'_1|+|\pi'_2|+|\T'_1|+|\T'_2|
$$and
$$
|\pi_1|+|\pi_2|+|\rho_1|+|\rho_2| = |\pi'_1|+|\pi'_2|+|\rho'_1|+|\rho'_2|.
$$

However claim $1$ implies that $|\T'_1|+|\T'_2| >|\T|$ and hence $|\pi_1|+|\pi_2| > |\pi'_1|+|\pi'_2|$. So we have $|\rho'_1|+|\rho'_2|>|\rho_1|+|\rho_2|$. This means that any standard path of the concatenation $\overline{p}_3\, p_5$ does not path $\CC \wedge \B$. So the path $\overline{p}_3\, p_5$ is not a geodesic. Since paths $p_3$ and $p_5$ are arbitrary, $w \notin \mathcal{I}[\CC, \B] $.

\textbf{Case 2:} $\mathcal{O}(\CC) \cap \mathcal{I}[\A, \B] = \emptyset$. Obviously $d(\mathcal{O}(\CC), \mathcal{I}[\A, \B]) = min \{d(u,v) \mid u \in \mathcal{O}(\CC),\; v \in \mathcal{I}[\A, \B]\}>0.$ Say the distance is $d>0$. Let $D \subset \mathcal{O}(\CC)$ denote the set of all vertices of $\mathcal{O}(\CC)$ realizing $d$. Note that the cardinality of $D$ is finite (there are finitely many elements in $\mathcal{O}(\CC)$ whose height is less than $h(\A \wedge \B)$ and, by the argument showing $\delta_0$ is unique (Lemma \ref{geodes}), there are finitely many vertices realizing $d$ whose height is greater than $h(\A \wedge \B)$). Take \emph{glb} over $D$ and let $\E$ denote this unique element. Say $\E'\in \mathcal{I}[\A, \B]$ is unique vertex with $d(\E', \E)=d$  (uniqueness follows from the argument in the proof of Lemma \ref{geodes}). See Figure \ref{interval_b}.

\textbf{Claim 3:} $d(\A, \E') + d + d(\E, \CC)= d(\A, \CC)$ and $d(\B, \E') + d + d(\E, \CC)= d(\B, \CC)$. Pick any standard geodesics $q_1 \overline{q_2}$ connecting $\A$ to $\E'$ and ascending paths  $q_3$ and $q_4$ joining $\E'$ to $\E$, and $\CC$ to $\E'$ respectively. Then the concatenation $q_1 \overline{q}_2 q_3 \overline{q}_4$ is a geodesic. Consider the standard path $q_1 q_2 \overline{q}_3\overline{q}_4$. Observe that the ascending path $q_1 q_2$ connecting $\A$ to $\A \wedge \CC$ realizes the distance $d(\A,\mathcal{O}(\CC))$. Thus any standard path given by $q_1 q_2 \overline{q}_3\overline{q}_4$ is a geodesic. Similar argument shows that $d(\B, \E') + d + d(\E, \CC)= d(\B, \CC)$. So $\E'$ belongs to the intersection $\mathcal{I}[\A, \B]\cap\mathcal{I}[\B, \gamma]\cap\mathcal{I}[ \gamma,\A]$.

\begin{figure}[ht]
\begin{center}
\includegraphics[width=1.0\textwidth]{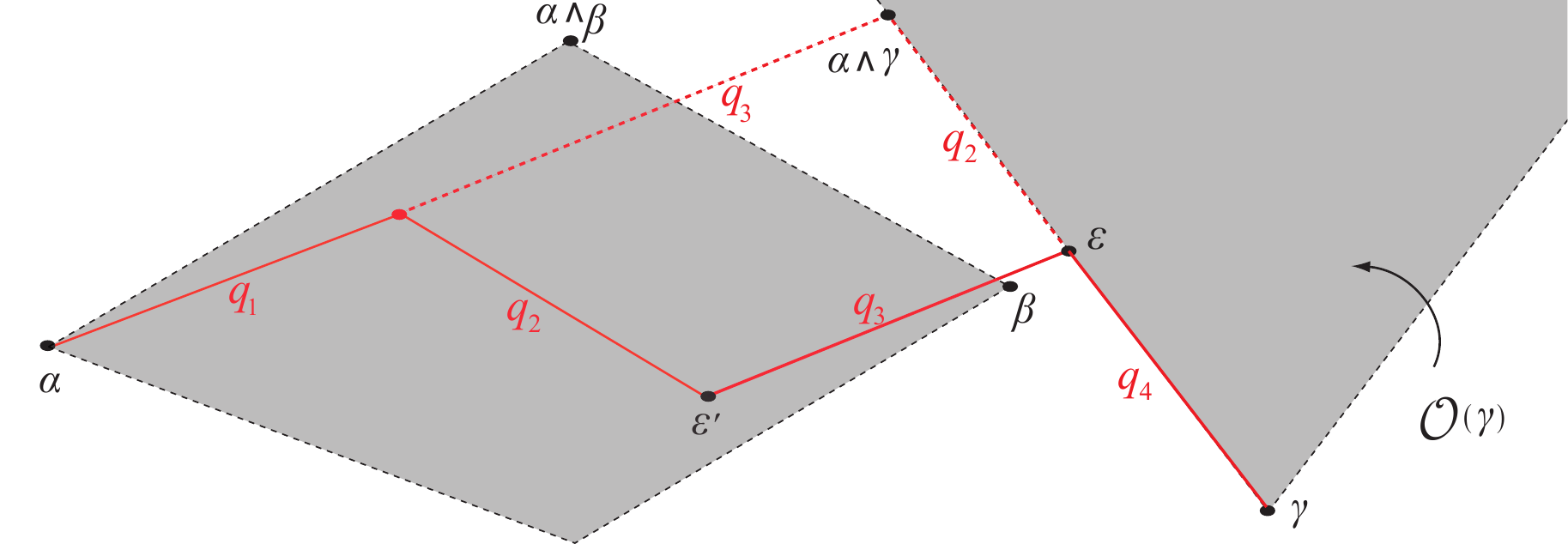}
\caption{Vertices $\E$ and $\E'$ with $d(\E', \E)= d(\mathcal{O}(\CC), \mathcal{I}[\A, \B])$} \label{interval_b}
\end{center}
\end{figure}

\textbf{Claim 4:} $\mathcal{I}[\A, \B]\cap\mathcal{I}[\B, \gamma]\cap\mathcal{I}[ \gamma,\A]= \{\E'\}$. Suppose the intersection contains another vertex $w \in \mathcal{I}[\A, \B]$. We apply an analogous trick as in \textbf{Case 2:} with small difference with $d(w, \CC) > d + d(\E ,\CC)$. One can show this inequality together with
$$
d(\A, \CC)=d(\A, w)+d(w, \CC)
$$ yields
$$
d(\B, \CC)>d(\B, w)+d(w, \CC).
$$
\end{proof}



\subsection{Properties of the action of $\ho_n$ on $X_n$}\label{sec_action_of_H_n_on_X_n}
In this section we examine the action of $\ho_n$ on $X_n$. First we examine that the stabilizer of every cell is a finite symmetric group.

\begin{lemma}\label{lm:stab_finite_symm}
Suppose $\A$ is a vertex of $X_n$ with $h(\A)=h$. Then the stabilizer of $\A$ is the finite symmetric group $\Sigma_h$ on $h$ points.
\end{lemma}

\begin{proof}Let $\Sigma_h \leq \ho_n$ denote the symmetric group on a finite set $S(\A)= Y_n - (Y_n)\A$. We show the stabilizer of $\A$ is simply $\Sigma_h \leq \ho_n$. If $g \in \Sigma$ then $\A g = \A$. Conversely, if $g \in \ho_n$ and $\A g = \A$ then $g$ restricted to the set $(Y_n)\A$ must be the identity. So $supp(g) \subset (\A)$ and hence $g \in \Sigma_h$.
\end{proof}

Recall the following definition of a Morse function defined on a (affine) CW-complex complex $X$ where
$\varphi_j : \Box^k \to \Si_j^k \subset  X^{k} $ denote the attaching map of $k$-cell $\sigma_j^k$.
\begin{definition}[\textbf{Morse function}]\label{dfn:Morse_function}
A map $f :X \to \mathds{R}$ is a $Morse\; function$
if
\begin{itemize}
    \item[-] for every cell $\varphi_j(\Box^k)$ of $X$ $f \varphi_j :\Box^k \to \mathds{R}$ extends to an affine map $\mathds{R}^m \to \mathds{ R}$	 and $f \varphi_j $ is constant only when $ k=0$ and

    \item[-] the image of the $0$-skeleton is discrete in $\mathds{R}$.

\end{itemize}
\end{definition}


Note that the map $h$ defined on a cubing  $X_n$ has the property that $h \varphi_j :\Box^k \to \mathds{R}$ extends to the standard height function $\R^k \to \R$ up to the translation by $h \varphi (0)$ (see Figure \ref{fig_diagram_morse}). Note also that $h \varphi_j:\Box^k \to \mathds{R}$ is trivial only when $k=0$ for all $j$. Moreover, the image $h(X_n^{(0)})$ is just $h(\C_n) = h(\M_n) = \mathds{Z}_{\geq0}$. Therefore the map $h$ is a Morse function.

\begin{remark}\label{rmk:height-preserving_action}
The map $h :X_n \to \R_{\geq0}$ satisfies the following
$$
h(\A g ) = h(\A) \;\,\text{and}\; \,h(t \A) = h(t) + h(\A)
$$for all $g \in  \ho_n$ and $t \in \TT_n$. In that sense, the action of $\ho_n$ is `horizontal' and the action of $\TT_n$ is `vertical'.
\end{remark}

Let $X_{n,r}$ denote the subcomplex of $X_n$ consisting of cubes up to height $r$, i.e.,
$$
X_{n,r} := \{\sigma \in X_n \mid h(\sigma) \subset  [0, r]\}.
$$

\begin{lemma}\label{lm:ccpt_action}
Suppose a $k$-cube $\sigma \subset X_{n}$ is generated by $T=\{\T_i, \cdots, \T_k\} \subset \{t_1, \cdots, t_n\}$ with bottom vertex $\A$. There exists $g \in \ho_n$ such that $\sigma \cdot g$ is the $k$-cube generated by $T$ with bottom vertex $t_1^{h(\A)}$.
\end{lemma}

\begin{proof}From the proof of Lemma \ref{genM_n}, we see that there exists an element $g \in \ho_n$ such that $\A g = t_1^{h(\A)}$. Since the action of $\ho_n$ on $X_n$ is cellular, a cube $\sigma g$ is the desired cube bottom vertex $t_1^{h(\A)}$.
\end{proof}

\begin{corollary}\label{coro:ccpt_action_level}
For $r \in \Z_{\geq0}$, $\ho_n$ acts on $X_{n,r}$ cocompactly.
\end{corollary}

\begin{proof}
Fix $r \in \Z_{\geq0}$. There are finitely many $k$-cubes with bottom vertex $t_1^{r'}$ for $1 \leq r' \leq r$ and $ 0 \leq k \leq n$. Lemma \ref{lm:ccpt_action} implies that for any cube $\sigma \subset X_{n,r}$ there exists $g\in\ho_n$ such that $\sigma \cdot g$ is a $k$-cube with bottom vertex $t_1^{r'}$ for some $r' \leq r$. So the quotient $X_{n,r} / \ho_n$ is finite.
\end{proof}

\begin{definition}[\textbf{Semi-simple group action}]
Let $X$ be a metric space and let $g$ be an isometry of $X$. The displacement
function of $g$ is the function $d_g: X \to \mathds{R}_{\geq 0}$ defined by $d_g (x) = d(x, x. g)$. The translation length of $d_g$ is the number $|d_g | := inf\{d_g (x) | x \in X\}$. The set of points where $d_g$ attains this infimum will be denoted Min($d_g$). More generally, if $G$ is a group acting by isometries on $X$, then Min($G$)$ := \cap_{g \in G}$ Min($d_g$). An isometry $d_g$ is called semi-simple if Min($d_g$) is non-empty. An action of a group by isometries of $X$ is called semi-simple if all of its elements are semi-simple.
\end{definition}

The following theorem summarizes the action of $\ho_n$ on $X_n$, see Figure \ref{setup} for the underlying theme.

\medskip
{\bf Theorem \ref{thm_action_on_X_n}.}\;
{\sl For each integer $n\!\geq\! 1$, there exists a $n$-dimensional cubing $X_n$ and a Morse function $h:X_n\rightarrow \mathds{R}_{\geq0}$ such that $\mathcal{H}_n$ acts on $X_n$ properly (but not cocompactly) by height-preserving semi-simple isometries. Furthermore, for each $r \in \mathds{R}_{\geq0}$ the action of $\mathcal{H}_n$ restricted to the level set $h^{\!-1}\!(r)$ is cocompact.}
\medskip

\begin{figure}[h]
\begin{center}
\includegraphics[width=0.5\textwidth]{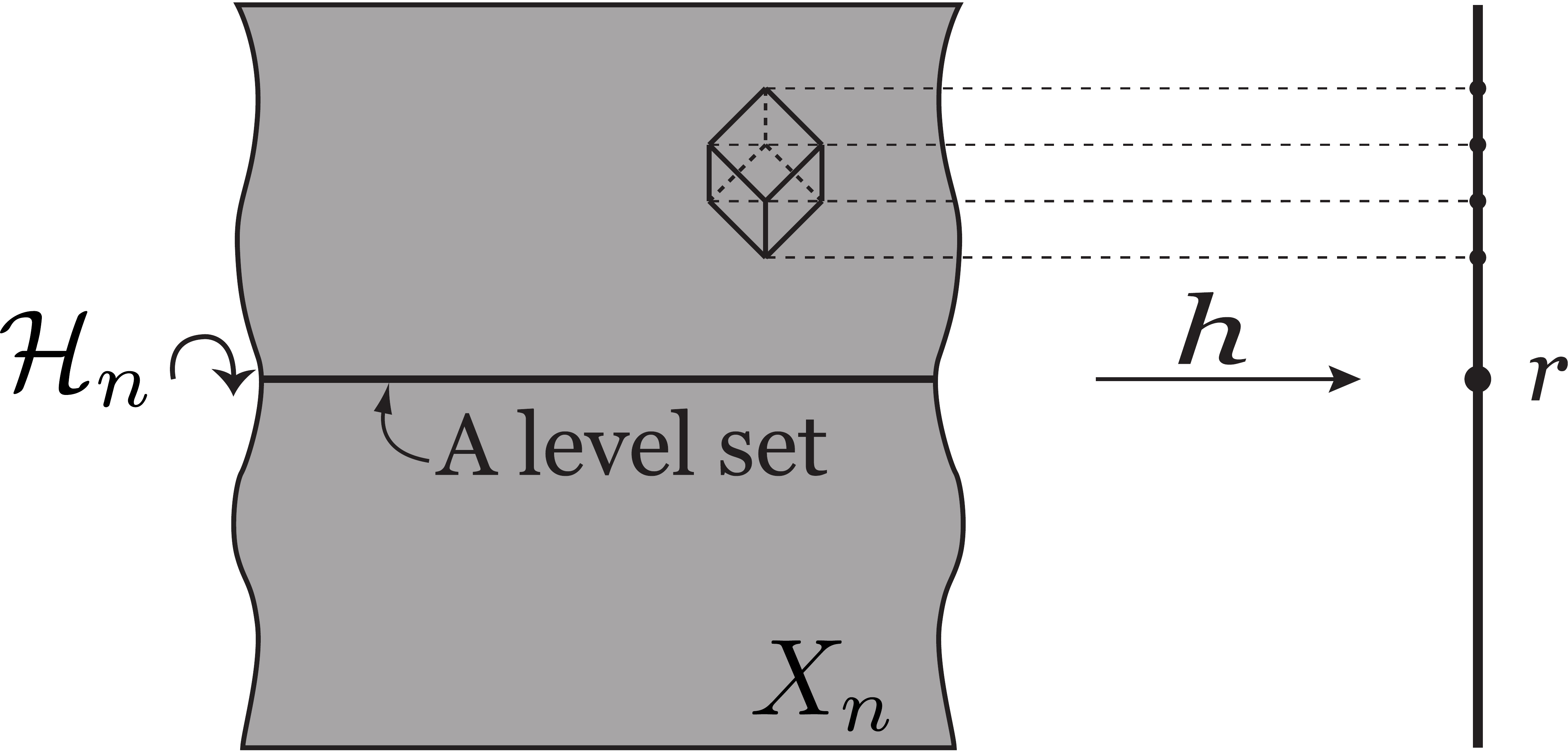}
\caption{A Morse function $h$ on a CAT(0) cubical complex $X$} \label{setup}
\end{center}
\end{figure}

\begin{proof}\textbf{Proper action.} We want show that, for $\A , \B \in X_n^{(0)}$ the set $S(\A, \B):=\{g \in \ho_n | \A g = \B\}$ is finite. An element $\A$ has the right inverse (left inverse in composition of functions) $\A^{-1} : (Y_n)\A \to Y_n$. So if $g \in S(\A, \B)$ then $g_{|(Y_n)\A}$ is completely determined by
$$
g = \A^{-1} \B.
$$This means that one can decompose $g$ by
\begin{equation*}
g =
    \begin{cases}
      \A^{-1} \B& \text{on }  (Y_n)\A\\
       f & \text{on } S(\A)
    \end{cases}
\end{equation*}for some $f\in \ho_n$ with $supp(f) \subset  S(\A)$. Since there are finitely many $f$ with $supp(f) \subset  S(\A)$, $S(\A, \B)$ is finite. Note that if $\A = \B$ then $S(\A, \B)$ is simply the finite symmetric group on $h(\A)=h(\B)$ points that we discussed in Lemma \ref{lm:stab_finite_symm}. Now suppose $\Si$ is a cube with vertices $\{\A_1, \cdots, \A_k \}$. If $\Si g $ intersects $\Si$ non trivially then $g \in S(\A_i, \A_j)$ for some $1 \leq i ,j \leq k$. Therefore the action is proper.

\textbf{Height-preserving isometries.} The action preserves height by Remark \ref{rmk:height-preserving_action}. To show the action in question is by isometries, we want to use the fact that two actions of $\TT_n$ and $\ho_n$ on $X_n$ commute. Suppose $\A, \B \in X_n^{(0)}$. Consider the smallest convex set $H$ containing $\A$ and $\B$, i.e., the intersection of all convex sets containing $\A$ and $\B$. Observe that the cell structure of $X_n$ is completely determined by the action of $\TT_n$ on $X_n^{(0)}$. Since the two actions of $\TT_n$ and $\ho_n$ on $X_n$ commute, $H g$ is a convex set with the same cell structure of $H$. By Corollary \ref{coro:contrac_unique_geod}, $\A$ and $\B$ are joined by the unique geodesic $\ell$, which lies in $H$. The image $\ell g$ is the local geodesic (in $H g$) joining $\A g$ and $\B g$. Since $H g$ is convex $\ell g$ is the global geodesic. So we have $d(\A, \B)= d(\A g, \B g)$ for $\A, \B \in X_n^{(0)}$ and $g \in \ho_n$.

\textbf{Semi-simple action.} For any element $f \in \ho_n$ with finite order, it is not difficult to check that $d_f = 0$ and Min($d_g$) is non empty. For example one can take $\A \in \M_n$ such that $supp (f) \subset S(\A)$ to see $d(\A f, \A)=0$ and $\A \in $Min($d_f$). Suppose $g \in \ho_n$ with $\varphi (g) = (m_1, \cdots, m_n) \neq \overrightarrow{0}$ ($m_1= -(m_2 +\cdots +m_n)$).

\textbf{Claim:} $d_g =\sqrt{\sum m_i^2}$. We want to find explicit element $\A \in X_n^{(0)}$ realizing $d_g$. The idea is to choose a vertex $\A$ with big enough height so that `translation' by $g$ is realized clearly. Consider the element $g' = g_1^{m_2} g_2^{m_3} \cdots g_{n-1}^{m_n}$. Since $\varphi (g)=\varphi (g)'$, there exists $f\in \ho_n$ such that $g = f g'$ and that $supp(f) \subset B_{n,r}$ where $B_{n,r} \subset Y_n$ is the ball centered at the origin of radius $r$. Set $r' = max\{r,\sum |m_i|\}$ and consider the element $\A = t_1^{r'} t_2^{r' }\cdots t_n^{r'}$. The maximality of $r'$ ensures that $\A, \A g' \in \TT_n$. By Corollary \ref{coro_glb_T_n}, there exists $\A \vee \A g'$. It is possible to find explicit expression for $\A \wedge \A g'$ as follows. Define $k_i, k'_i \in Z$, $i=1, \cdots, n$ by
$$
    k_1 = max\{0, -\sum_2^n m_i\},\; k'_1 = max\{0, \sum_2^n m_i\}  \;\text{and}
$$
$$
    k_i=max\{ 0, m_i\},\;k_i=max\{ 0, -m_i\}\;\text{for}\; 2\leq i \leq n.
$$
Since $\varphi(\A ) =(r', \cdots, r')$ and $\varphi(\A g')= (r' +m_1 ,r' +m_2, \cdots, r'+{m_n})$, the top vertex is given by
\begin{equation}\label{eq_big_lub}
   \T \A = \A  \wedge (\A g')= \T'(\A g')
\end{equation}
where $\T = t_1^{k_1} \cdots t_n^{k_n}$ and $\T' =  t_1^{k_1} \cdots t_n^{k_n}$.
By Proposition \ref{prop_big_square}, $\A \vee \A g'$ satisfies
\begin{equation}\label{eq_big_glb}
    \A = \T'(\A \vee \A g')\;\text{and}\;\A g' = \T(\A \vee \A g').
\end{equation}

\begin{figure}[h]
\begin{center}
\includegraphics[width=0.35\textwidth]{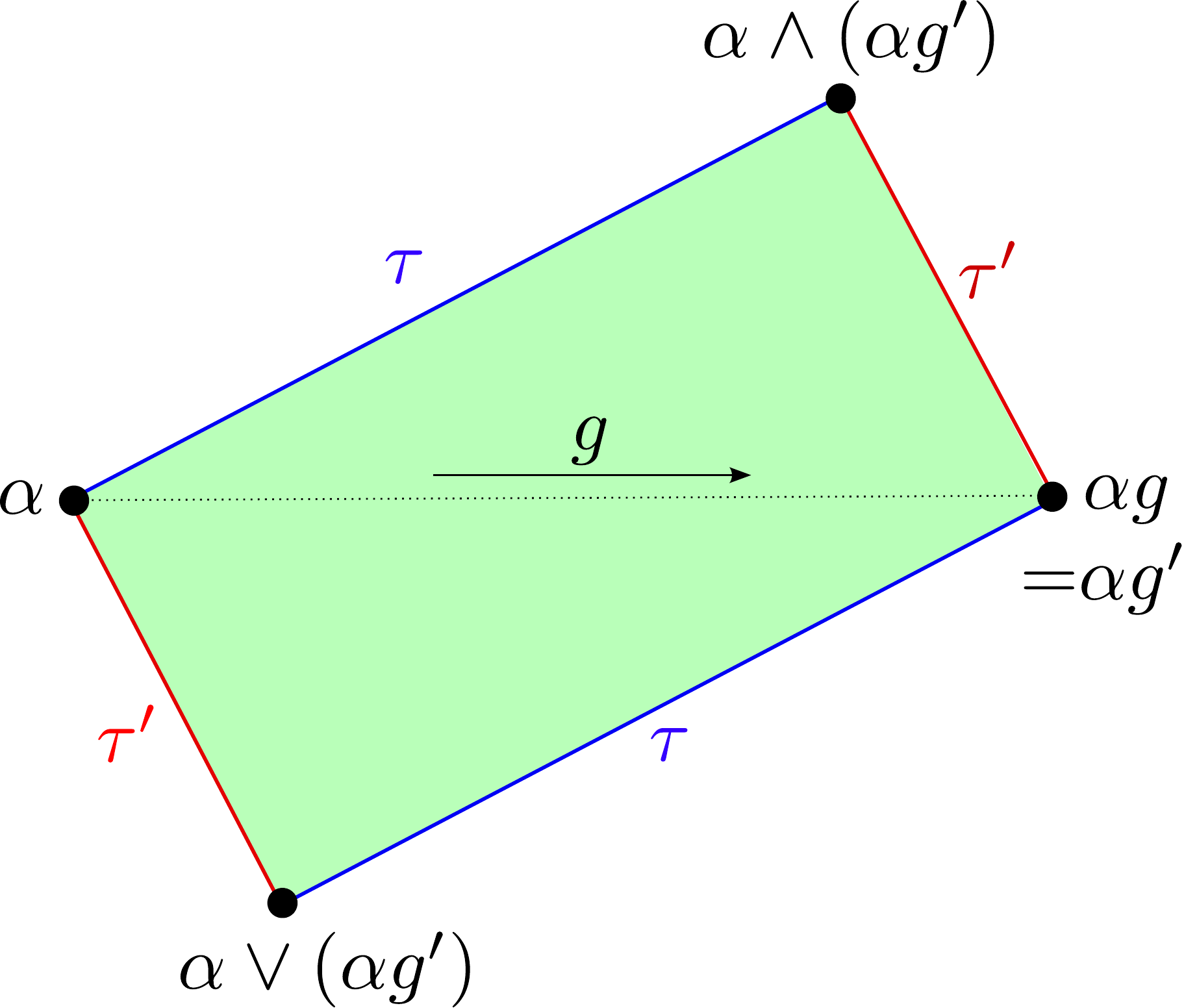}
\caption{If $h(\A)$ is big enough, the translation by $g$ is realized in a `big rectangle' $R$} \label{fig:big_rectangle_semi-simple}
\end{center}
\end{figure}
Equalities (\ref{eq_big_lub}) and (\ref{eq_big_lub}) imply that there exists a `big rectangle' $R$ which contains $\A \wedge \A g'$ and $\A \vee \A g'$ as its top and bottom vertices respectively (as described in Proposition \ref{prop_big_square}). See Figure \ref{fig:big_rectangle_semi-simple}. Observe that $R$ is the convex hull of two vertices $\A$ and $\A g'$. The diagonal of $R$ joining $\A$ and $\A g'$ has length $\sqrt{\sum m_i^2}$. The diagonal is the geodesic joining $\A$ to $\A g'$ because no smaller cube contains those two vertices. So we have
$$
d(\A , \A g) = d(\A, \A f g')=d(\A, \A g') =\sqrt{\sum m_i^2}.
$$
Note that the size of $R$ is determined by $\varphi (g) = (m_1, \cdots, m_n)$ and that $R$ has the smallest size among convex hulls containing $\B$ and $ \B g$ for $\B \in X_n^{(0)}$. This means $d_g \geq d(\A, \A g)$. Therefore $d_g$ attains minimum at $\A$ and $\A \in $Min($d_f$).

\textbf{Cocompact action.}
Note that $h^{\!-1}\!(r) \subset X_{n,r'}$ for $r\in \R$ and $r\leq r' \in \Z$. By Remark \ref{coro:ccpt_action_level}, the action on the level set is cocompact for each $r \in Z_{\geq 0}$.
\end{proof}



\subsection{Finiteness Properties of $\ho_n$} \label{sec_finite_prop}
In this section we discuss finite properties of $\ho_n$. We first recall properties $F_n$ and $F\!P_n$ (see \cite{BrownBook},\cite{Geoghegan}) which are generalized concepts of being finitely generated and finitely presented of groups.

We say a group $G= \langle\mathcal{A}\,|\, \mathcal{R}\rangle$ is \emph{finitely generated} if $\mathcal{A}$ is finite and \emph{finitely presented} if both $\mathcal{A}$ and $\mathcal{R}$ are finite. These finite conditions can be interpreted topologically via \emph{presentaion $2$-complex (Cayley complex)} $K=K(\mathcal{A};\mathcal{R})$. $K$ has one vertex and it has one edge $e_a^1$ (oriented and labelled by $a$) for each generator $a\in \mathcal{A}$. The $2$-cells $e_r^2$ of $K$ are indexed by the relators $r \in \mathcal{R}$; if $r=a_1 \cdots a_k$ then $\sigma_r$ is attached along the loop labelled by $a_1 \cdots a_k$. From the construction of $K$ we see that it has finite $1$-skeleton if $G$ is finitely generated and it has finite $2$-skeleton as well if $G$ is finite presented. By the Seifert-Van Kampen theorem we have $\pi_1 (K) \cong G$. Note that the existence of a complex $K$ with finite $2$-skeleton with $\pi_1 (K) \cong G$ guarantees finite presentedness of a group $G$.

More general finiteness properties are based on $K(G,1)$ spaces (\emph{Eilenberg-Mac Lane complex}). For a group $G$, a complex $K$ is called a $K(G,1)$ complex if $\pi_1 (K) \cong G$ and the universal cover $\widetilde{K}$ is contractible. It is known that for a group $G$ there exists unique $K(G,1)$ complex having one vertex up to homotopy.

\begin{definition}[\textbf{Property $F_n$}]
We say a group $G$ has type $F_n$ if there exists a $K(G,1)$ complex having finite $n$-skeleton.
\end{definition}



Consider the augmented chain complex $C_\ast (\widetilde {K};\mathds{Z})$ of the universal cover of a $K(G,1)$ complex $K$.
\begin{equation*}
\cdots \xrightarrow{\;\;\;\partial_3\;\;\;} C_2 (\widetilde {K};\mathds{Z})\xrightarrow{\;\;\;\partial_2\;\;\;} C_1 (\widetilde {K};\mathds{Z})\xrightarrow{\;\;\;\partial_1\;\;\;} C_0 (\widetilde {K};\mathds{Z})\xrightarrow{\;\;\;\E\;\;\;} \mathds{Z} \to 0.
\end{equation*}This is unique free $\mathds{Z}G$-resolution of $\mathds{Z}$ up to chain homotopy. Note that each $C_i(\widetilde {K};\mathds{Z})$ is finitely generated free $\mathds{Z}G$-module if $K$ has finite $i$-skeleton. A module is \emph{projective} if it is a direct summand of a free module.

\begin{definition}[\textbf{Property $F\!P_n$}]
We say a group $G$ has type $F\!P_n$ if there exists a projective $\mathds{Z}G$-resolution of $\mathds{Z}$ which is finitely generated in dimensions $\leq n$.
\end{definition}

\begin{remark}An immediate consequence is that if a group $G$ is has type $F_n$ then it has type $F\!P_n$ for integers $n \in \mathds{N}$.
\end{remark}


A $G$-CW complex is a complex $K$ together with a homomorphism $G \to Aut K$. A \emph{$G$-filtration} of a CW complex $K$ is a countable collection of $G$-subcomplexes $K_0 \subset K_1 \subset \cdots $ such that $K = \bigcup K_i$.

Let $K$ be a contractible $G$-CW complex which admits a $G$-filtration $\{K_i\}$ satisfying
\begin{itemize}
    \item the stabilizer of every cell is finitely presented and has type $F\!P_n$ for all $n$
    \item each $K_i$ is finite mod $G$
    \item for all sufficiently large $j$, $K_{j+1}$ is obtained from $K_j$ by the adjunctions of $n$-cells, up to homotopy.
\end{itemize}

\begin{theorem}[\textbf{Brown's Criterion}]\label{KBT}
With the above assumption, $G$ has type $F\!P_{n-1}$ but not $F\!P_n$. If $n\geq 3$ then $G$ is finitely presented.
\end{theorem}

Brown show the following finiteness properties of $\ho_n$ by constructing a CW complexes satisfying the above criterion (\cite{Brown}). By Theorem \ref{thm_action_on_X_n}, a cubing $X_n$, $n \geq 1$, satisfies the criterion above, Remark \ref{coro:ccpt_action_level}.

\medskip
{\bf Corollary \ref{coroll_fin_prop_H_n}.}\;
{\sl For $n\!\geq\!2$, $\mathcal{H}_n$ is of type $F\!P_{\!{n-1}}$ but not $F\!P_{\!n}$, it is finitely presented for $n \geq 3$.}
\medskip

\begin{proof}We show our $X_n$ satisfies the above Brown's criterion. By Corollary \ref{coro:contrac_unique_geod}, $X_n$ is contractible. Lemma \ref{lm:stab_finite_symm} implies the first condition is satisfied since a finite symmetric group satisfies required finiteness properties. The second condition follows from Remark \ref{coro:ccpt_action_level}. With respect to a Morse function $h:X_n \to \R$, each descending link of a vertex $v \in X_n$ is homotopic to bouquet of spheres $\mathds{S}^{n-1}$ if the height of $h(v)  \geq 2n-1$ ( Lemma \ref{sphere}). This means that the passage from $X_{n,h}$ to $X_{n,h+1}$ consists of the adjunction of $n$-cubes, up to homotopy. The theorem therefore follows from the criterion Theorem \ref{KBT}.
\end{proof}

\begin{lemma}\label{sphere}
If $h \geq 2n-1$ then $L_{n,h}$ homotopic to a bouquet of spheres $\mathds{S}^{n-1}$.
\end{lemma}

\begin{proof}We argue by induction on $n$. If $n=1$, then $L_{n,h}$ is a bouquet of spheres $\mathds{S}^0$, provided $h \geq 2 \cdot 1 -1$. Assume $L_{n-1,k}$ is homotopy equivalent to a bouquet of copies of $\mathds{S}^{n-2}$ if $k \geq 2h -1$ and $n \geq 2$.

We use Betvina-Brady Morse theory to understand the topology. Equip $L=L_{n,h}$ with a Morse function $f : L \to [0, 2] \subset \mathds{R}$ as follows. First consider the partition of the vertex set of $L_{n,h}$: $L_0=\{(x,y) \mid x \geq 2,\;y \geq 2\}\cup \{(1,1)\}$, $L_1=\{(1,y) \mid y \geq 2\}$ and $L_2=\{(x,1) \mid x \geq 2\}$. Arrange the vertices so that $f(x,y) = h$ if $(x,y)\in L_h$, $h= 0,1,2$. See Figure \ref{morse}.

\begin{figure}[ht]
\begin{center}
\includegraphics[width=0.5\textwidth]{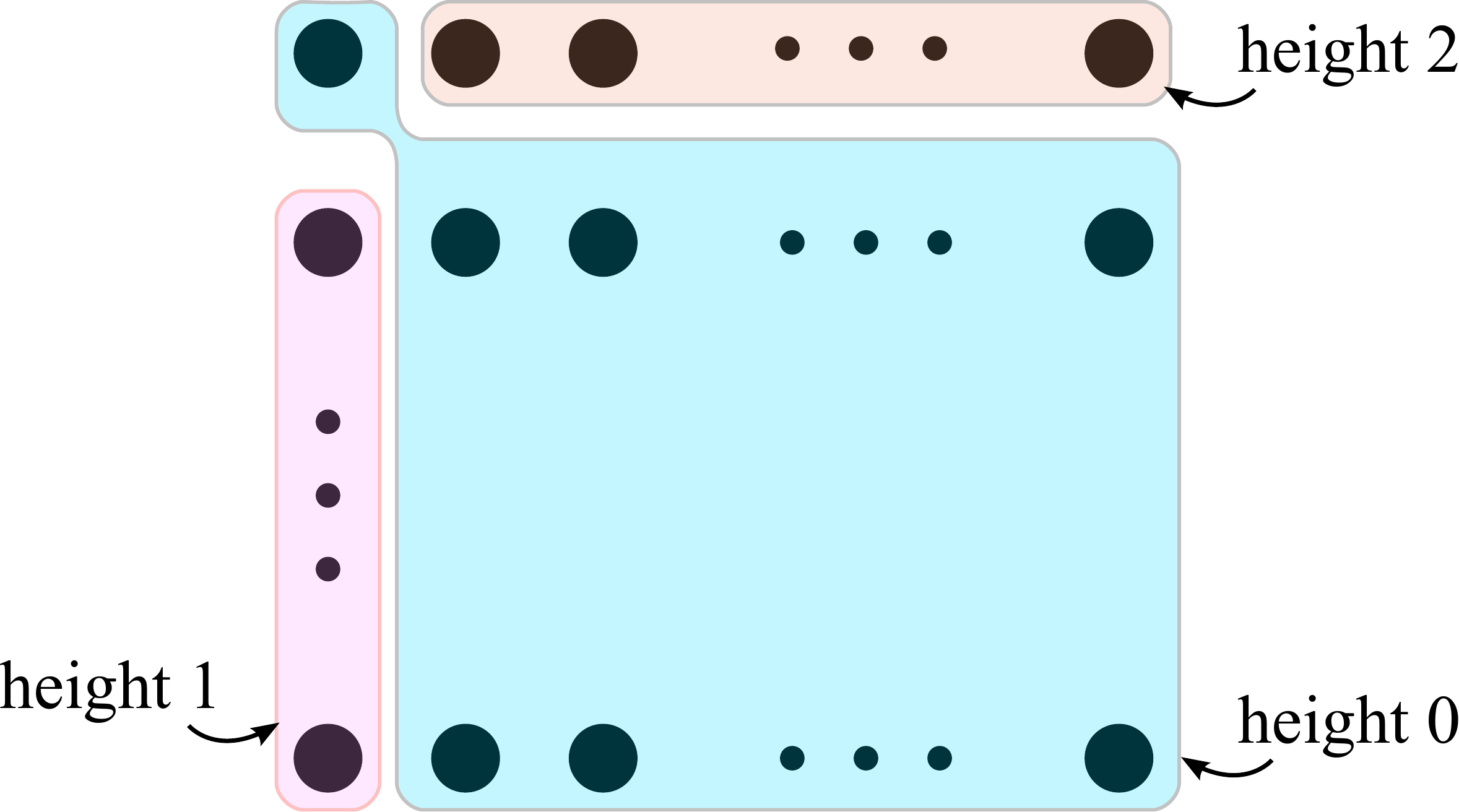}
\caption{The complex $L_{n,h}$ equipped with a Morse function $f : L_{n,h} \to [0, 2]$} \label{morse}
\end{center}
\end{figure}

It is clear that $f^{-1}(0)$ is a cone on $L_{n-1,h-1}$ and so it is contractible. Note that even if there are horizontal cells at height $0$ it is still true that $f^{-1}[0,1]$ is homotopy equivalent to $f^{-1}(0)$ with the copies of $Lk_\downarrow((1,y),L)$ conned off ($(1, y)\in L_1$). Observe that the descending link of a vertex $(1, y) \in L_1$ (in $L$) is spanned by vertices $(x,y')\in L_0$ with $ y' \neq y$. Now $Lk_\downarrow((1,y),L)\simeq L_{n-1,h-2}$, $f^{-1}[0,1]$ is obtained from $f^{-1}(0)$ by adjoining, for each $(1, y)$, a cone over $L_{n-1,h-2}$. In view of inductive hypothesis, $f^{-1}[0,1]$ is homotopy equivalent to a bouquet of spheres $\mathds{S}^{n-1}$.

Observe that $Lk_\downarrow((x,1),L)\simeq L_{n-1,h-1}$ for each vertex $(x,1)\in L_2$. So, by the inductive hypothesis again, the complex $L = f^{-1}[0,2]$ is homotopy equivalent to $f^{-1}[0,1]$ with the copies of spheres $\mathds{S}^{n-2}$ conned off. (Similar proof can be found in \cite{Brown}.)
\end{proof}

%


\section{Isoperimetric inequalities for $\ho_n$}
\subsection{Dehn functions of finitely presented groups}\label{section_Dehn_function}

We recall the (algebraic) definition of Dehn function of a group from \cite{Bridson}. Let $P = \langle \mathcal{A} \mid \mathcal{R} \rangle$ be a finite presentation for a group $G$ with the identity $1_G$. A word $w$ is an element of the free monoid with the generating set $\mathcal{A} \cup \mathcal{A}^{-1} $. Denote the length of $w$ by $| w |_G$. We say
$w$ is null-homotopic when $w = 1_G$, i.e., $w$ lies in the normal closure of $\mathcal{R}$ in the free group $F(\mathcal{A})$. We define the \emph{area} of a null-homotopic word $w$ to be
$$
Area(w):= min\{N \mid w=^{free} \prod _{i=1}^{N} u_i^{-1} r_i u_i \; \text{with}\; u_i \in F(\mathcal{A}), \, r_i \in \mathcal{R}^{\pm1}\}.
$$
The \emph{Dehn function} of $P$ is the function $\delta_P : \mathds{N} \to \mathds{N}$ defined by
$$
\delta_P (x): = max\{Area(w) \mid w= 1_G, |w|_G \leq x\}.
$$
Although the Dehn function $\delta_P (x)$ depends on the presentation $P$, asymptotic growth type of $\delta_P (x)$ (as $x$ tends to infinity) only depends on $G$ up to the equivalence relation $\simeq$ defined as follows (see \cite{Bridson}).
Two functions $f,g: \mathds{N} \to [0, \infty)$ are said to be $\simeq$ equivalent if $f \preceq g$ and $g \preceq f$, where $f \preceq g$ means that there exists a constant $C>0$ such that $f(x) \leq C g(Cx +C) +Cx +C$ for all $x \in \mathds{N}$.
Up to this equivalence relation, any finite presentation $P$ of a group $G$ determines the same asymptotic growth type, which is called the $Dehn \;function$ of $G$. The Dehn function of $G$ is denoted by $\delta_G (x)$. The Dehn function is an important invariant of group theory. If a group $G$ is CAT(0), the $\delta_G(x)$ is bounded above by $x^2$, see \cite{[Bridson-Haefliger]}. A group $G$ is hyperbolic if and only if $\delta_G(x) \simeq x$, see \cite{Gromov}. In particular, every finite group has a linear Dehn function.
\begin{remark}
For any $m \in \mathds{N}$ the symmetric group $\Sigma_m$ on $\{1,2, \cdots, m\}$ satisfies
$\delta_{\Sigma_k} (x) \leq Cx +C$ for some $C >0$. Note that the constant $C$ depends on the group $\Sigma_m$ and hence on $m$. We need an upper bound for $\delta_{\Sigma_m} (x)$ which only depends on $x$ (see Lemma \ref{lm:cubic}).
\end{remark}

\subsection{Exponential isoperimetric inequalities for $\ho_n$}\label{section_Exp_isoperi_ineq_ho}

In this section we aim to establish exponential upper bounds for $\delta_{\ho_n}$ for $n \geq 3$.

\medskip
{\bf Theorem \ref{thm:exp_bd_n}.}\;
{\sl For $n\geq 3$, the Dehn function $\delta_{\ho_n}(x)$ satisfies
$$
 \delta_{\ho_n}(x)\preceq e^x.
$$}
\medskip

We start with the case $n=3$. Let $B_{3,x} \subset Y_3$ centered at the origin with radius $x \in N$. Let $\Sigma_{3,x}\leq \ho_3$ denote the finite symmetric group on $B_{3,x}$ given by (\ref{eq_presen_fin_symm_new}). The following is an outline for the proof.

\begin{enumerate}
    \item For a given word $w=1 \in \ho_3$ with $|w|\leq x$, rewrite $w$ by $w' \in \Sigma_{3x}$ by using a canonical way such that
        \item [-] $|w'|_{ \Sigma_{3,x}} \leq x^5$,
        \item [-] the gap between $w$ and $w'$ is filled with area$\leq x^2$.
    \item Establish a cubic upper bound for  $\delta_{\Sigma_{3,x}}(x)$.
    \item Bound $Area_{\ho_3}(r)$ for all relators $r$ of $\Sigma_{3,x}$ by $e^x$,
\end{enumerate}

\begin{figure}[h]
\includegraphics[width=1.0\textwidth]{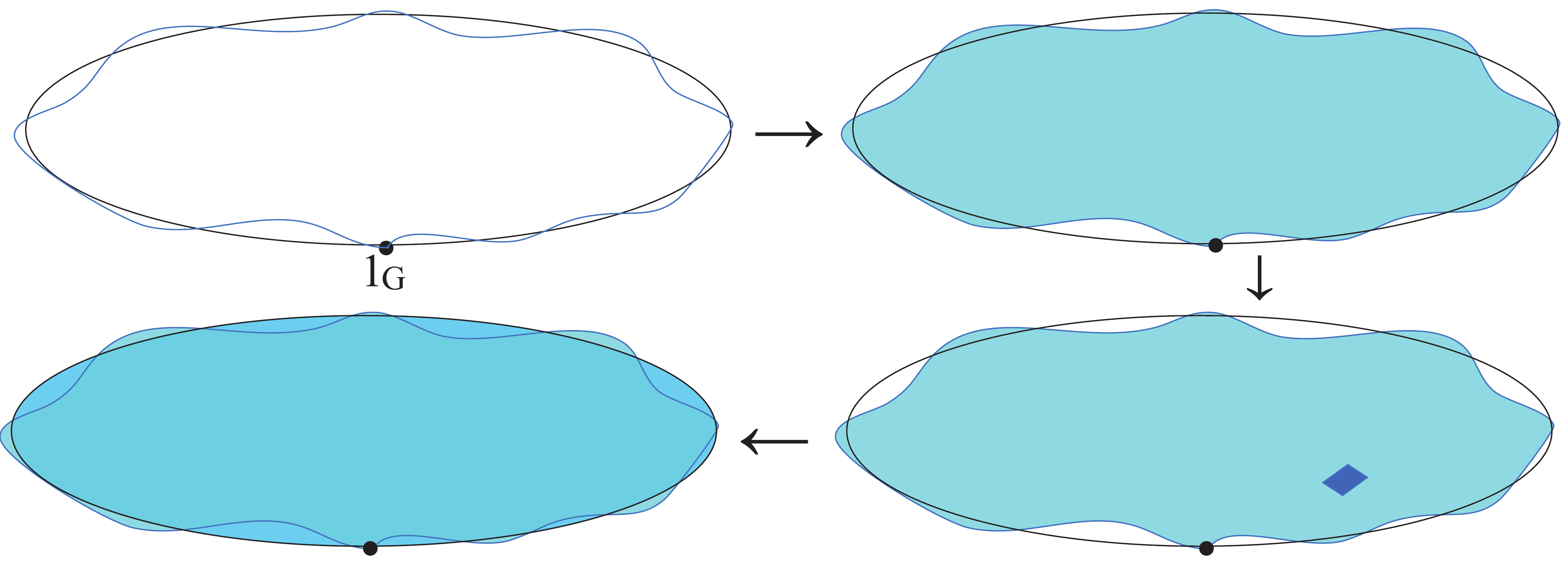}
\caption{Sketch of the proof for $\delta_{\mathcal{H}_3}(x) \preccurlyeq e^x$} \label{sketch}
\end{figure}

Figure \ref{sketch} illustrates the strategy. Now we have a desired upper bound for the area of $w$ since

\begin{equation}\label{eq_calculation_exp_3}
Area_{\mathcal{H}_3}(w) \leq (x^5)^{3} \cdot e^{Ax+A} +x^2\leq e^{(A+18)x+A}+x^2 \preccurlyeq e^x.
\end{equation}
We remark that, in each step, any bound not exceeding exponential function is good enough for us. We have established exponential upper bound for $\delta_{\ho_3}(x)$ up to those claims.

First we establish a cubic upper bound for $\delta_{\Sigma_{3,x}}(x)$ by using the following fact. Let $S_m$ denote the finite symmetric group on $m$ points with the Coxeter presentation given in (\ref{coxeter_sys}), where $\{\Si_1, \cdots, \Si_{m-1}\}$ is the generating set.
Let $s_i = \Si_{\phi(i)}$ for some $\phi :\mathds{N} \to \{1,2,\cdots, m-1\}$.

\begin{theorem}[Deletion Theorem, \cite{Humphreys}]\label{Humphreys}
Let $\Sigma_N $ denote the symmetric group on $N$ letters with generatros $\sigma_1, \cdots , \sigma_{N-1}$. Suppose $w=s_1 s_2 \cdots s_k$, $s_i \in \{\sigma_1, \cdots , \sigma_{N-1}\}$. If $|w|<k$ then there exists $i$ and $j$ ($1 \leq i < j \leq k$) such that
\begin{equation}\label{conj1}
s_{i+1} s_{i+2} \cdots s_{j} = s_{i} s_{i+1} \cdots s_{j-1},
\end{equation} and so
$$
w= s_1 s_2 \cdots \widehat{s}_i \cdots \widehat{s}_j\cdots s_{k}.
$$
\end{theorem}

\begin{lemma}\label{lm:cubic}
For any $m \in \mathds{N}$, a null-homotopic word $w \in \S_{m}$ with $|w|\leq x$ satisfies $Area(w) \leq x^3$.
\end{lemma}
\begin{proof} We first show any word $w \in \S_{m}$ given by the form
$$
s_j = s_i^{s_{i+1} s_{i+2}\cdots s_{j-1}}
$$
as in (\ref{conj1}) has at most quadratic area in $|j-i|$. ($g^{h} :=h^{-1} g h$.) Suppose $w= {s_{k+1}}({\overline{s}_0}^{s_1 \cdots s_k})$ represents the identity of $\Sigma_m$. We want to show $Area(w) \leq k^2$ by induction on $k$. 

The base case is obvious since $ {s_{2}}({\overline{s}_0}^{s_1})=1$ is a single commutation relation. We may assume no two consecutive letters in $w$ are the same. Now suppose $w= {s_{k+2}}({\overline{s}_0}^{s_1 \cdots s_{k+1}})$ is null-homotopic. We consider two cases. \\
\textbf{Case 1}. There exists $i_0$ ($k+2 \geq i_0 \geq 1$) such that either
\begin{equation}\label{comm1}
[s_{i_0}, s_i]=1 \;\,\text{for all} \;\, i_0 > i \geq 0
\end{equation} or
\begin{equation}\label{comm2}
[s_{i_0}, s_i]=1\;\, \text{for all} \,\; k+2 \geq i > i_0 .
\end{equation}

If the condition (\ref{comm1}) holds then one can apply commutation relations consecutively to decompose the diagram of $w$ into a $s_{i_0}-$corridor with length $2(i_0 -1) +1$ and a diagram of $w'=  {s_{k+1}}({\overline{s}_0}^{s_1 \cdots \widehat{s}_{i_0} \cdots  s_k})$. (See Figure \ref{fig:reduced_diagram}) Thus, by induction assumption, we have
$$
Area(w) \leq Area(w')+ 2(i_0 -1)+1 < (k+1)^2.
$$ By using an analogous argument together with condition (\ref{comm2}), one can draw the same upper bound.

\begin{figure}[h]
\begin{center}
\includegraphics[width=.55\textwidth]{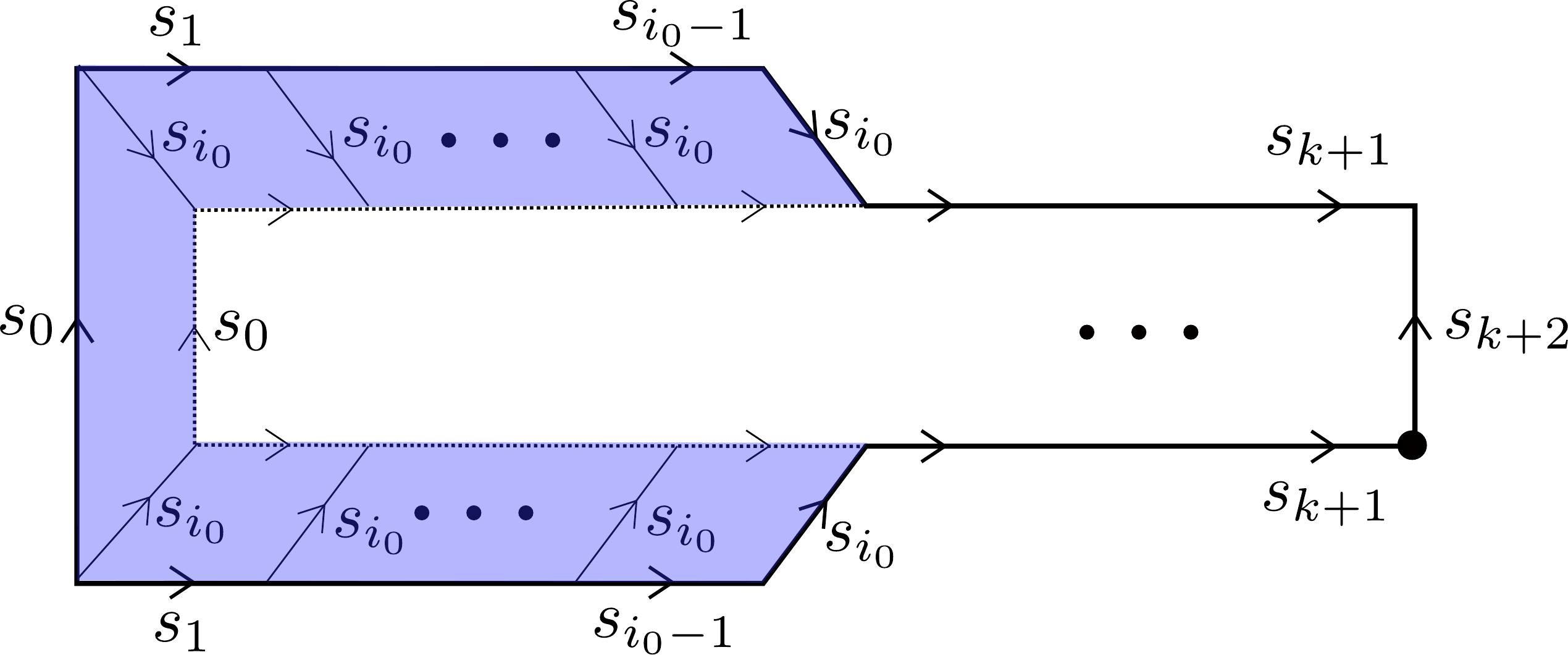}
\caption{The diagram for $w$ can be reduced by commutation relations in Case 1.} \label{fig:reduced_diagram}
\end{center}
\end{figure}
\textbf{Case 2}. For each $k+1 \geq i \geq 1$ there exists $ i'' >i > i'$ such that
\begin{equation}\label{comm3}
[s_{i''}, s_{i}] \neq 1 \;\,\text{and} \;\,[s_{i}, s_{i'}] \neq 1.
\end{equation} The above condition simply says that $\sigma(i'') = \sigma(i)\pm1$ and $\sigma(i') = \sigma(i)\pm1$. We need to examine subwords of $w$. Let $w_j$ denote the subword ${\overline{s}_0}^{s_1 \cdots  s_j}$ for $j=1, \cdots, k+1$. For $j=0$, $w_0$ is simply defined to be $s_0$. Note that each $w_j$ is a transposition since it is a conjugation of a transposition $s_0$. For a transposition $s \in \Sigma_m$ let $d_{+}(s)$ and $d_{-}(s)$ denote the two points of $supp(s)$ with $d_{+}(s) > d_{-}(s)$. The function $d :\{0, 1, \cdots, k+1\} \to \N$ measures the difference $d_{+} - d_{-}$, i.e., $d(j)=d_{+}(w_j) - d_{-}(w_j)$ for $j=0,1,\cdots k+1$. Set $D(j) = \{d_- (w_{j}), d_- (w_{j}) +1, \cdots , d_+ (w_{j})\}$. Note that $d(0)=d(k+1)=1$. So there exists $i$ such that $d(i+1) \leq d(i)$, say $i_0$ is the smallest such number. Observe that $|d(j) - d(j+1)| \leq 1$ for all $j$ since conjugating by $s_j$ introduces at most one point to $D(j-1)$. 
The following observation is crucial to establish desired bound;\\
$$
d(w_{i_0})= i_0 +1.
$$ This identity implies
there is $1$-to-$1$ correspondence between
\begin{equation}\label{1to1}
D(i_0)- D(0)\; \text{ and}\; \{s_1, s_2, \cdots, s_{i_0}\}.
\end{equation}
This bijection and the condition (\ref{comm3}) implies
\begin{equation}\label{order}
d_+(s_i) = d_+ (s_j) +1 \Rightarrow i>j
\end{equation}
for all $i$ with $i_0 \leq i \leq 1$. 

From the $1$-to-$1$ correspondence (\ref{1to1}) and the fact that $D(s_{i_0+1}) \subset D(w_{i_0})$, there exists unique $i_0   > i'  \geq 0$ such that $s_{i_0 +1} = s_{i'}$. Say $s_{i_0 +1} = s_{i'} = (p \,\,p\!+\!1)$, transposition exchanging $p$ and $p\!+\!1$. On the other hand, if $d_-(s_{i_0 +1}) \geq d_+(s_0)$ then, from (\ref{order}), we see that there exists unique $i''$ with $i_0 \geq i'' > i'$ such that $s_{i''} = (p\!+\!1 \,\, p\!+\!2)$. Observe that $s_{i_0 +1}$ commutes with  $s_i$ for all $i$ with $i_0 \geq i >i''$ and that $s_{i'}$ commute with $ s_j$ for all $j$ with $i'' > j \geq i'$. Apply those commutation relations to rearrange letters in the expression of $w = {s_{k+2}}({\overline{s}_0}^{s_1 \cdots s_{k+1}})$ so that $s_{i_0 +1}$, $s_{i''}$ and $s_{i'}$ show up in a row. Then apply the relation
$$
s_{i_0 +1} s_{i''}s_{i'} = s_{i'} s_{i''} s_{i'}= s_{i''} s_{i'} s_{i''},
$$ where the second identity comes from the braid relation $((p \, p\!+\!1)(p\!+\!1 \,\, p\!+\!2))^3 =1$. Now one applies the argument of \textbf{Case 1} since $s_{i''}$ commute with all $s_i$ for $i \leq i''$. In all, from

\begin{align*}
s_1 \cdots s_{i_0} s_{i_0 +1} &= s_1 \cdots s_{i'} \cdots  s_{i''} \cdots s_{i_0} s_{i_0 +1} \\
 &= s_1 \cdots \widehat{s}_{i'} \cdots  s_{i''-1} (s_{i'} s_{i''} s_{i_0 +1}) s_{i''+1} \cdots s_{i_0}
 \\
& =s_1 \cdots \widehat{s}_{i'} \cdots  s_{i''-1} (s_{i''} s_{i'} s_{i''}) s_{i''+1} \cdots s_{i_0}
\\
& = s_{i''} s_1 \cdots \widehat{s}_{i'} \cdots  s_{i''-1} ( s_{i'} s_{i''}) s_{i''+1} \cdots s_{i_0}
\end{align*}
we have
$$
w= {s_{k+2}}({\overline{s}_0}^{s_1 \cdots s_{i''} \cdots s_{k+1}})= {s_{k+2}}({\overline{s}_0}^{s_{i''} s_1 \cdots \widehat{s}_{i''} \cdots s_{k+1}})={s_{k+2}}({\overline{s}_0}^{ s_1 \cdots \widehat{s}_{i''} \cdots s_{k+1}})
$$
The number of required relators in the above process is at most
$$
2\{(i_0 -i'' -1) + (i'' - i' -1) +1+ i''\} + k^2 \leq k^2 + 2k < (k+1)^2
$$
since $i_0 \leq {k \over 2}$. So we have shown any word in the form of (\ref{conj1}) has area at most $|j-i|^2 \leq k^2$.

This means whenever one applies the identity (\ref{conj1}) to a null-homotopic word $s_1 \cdots s_x \in \Sigma_m$, the number of relators is bounded by $x^2$. At the same time one can reduce the number of generators in the expression by two. So we have $\delta_{\Sigma_m} (x) \leq x^3$.
\end{proof}


As discussed in Remark \ref{lm_Si=ker}, one of important feature of $\ho_n$ is $\Sigma_{n,\infty} \hookrightarrow \ho_n$ when $n\geq 3$. If a word $w \in \ho_3$ is null-homotopic then $w \in \Sigma_{3,\infty}  = \cup_r \Sigma_{3,r}$ (Lemma \ref{lm_A_n_iso_Sigma}) It is natural to ask the minimum $r$ so that $w \in \Sigma_{3,r}$. We have a reasonable bound for $r$ as well as the length of $w$ in $\Sigma_{3,r}$.

\begin{lemma}\label{lm:canonical_bound}
There is a canonical way so that any null-homotopic word $w \in \ho_3 $ with $|w|\leq x$ can be written as a word $w' \in \Sigma_{3,x}$ with length at most $O(x^5)$.
\end{lemma}
\begin{proof} Suppose $w =1_{\mathcal{H}_3}$ with $|w|_{\mathcal{H}_3} =x$. Then $w$ can be written as $w=g_1^{m_1} g_2^{n_1} \alpha^{\epsilon_1} \cdots g_1^{m_k} g_2^{n_k} \alpha^{\epsilon_k}$ for some $m_i, n_i \in \mathds{Z}$, $\epsilon_i \in \{-1, 0, 1\}$.

Use the identity
\begin{equation}\label{con}
g \alpha \equiv \alpha^{g^{-1}} g
\end{equation}
to rewrite $w$ as
\begin{equation}\label{w}
w\equiv f \, g_1^{m_1} g_2^{n_1} \cdots g_1^{m_k} g_2^{n_k},
\end{equation}
where $f$ is a product of at most $k(<x)$ many transpositions whose supports belong to the ball $ B_{3,x}$ of radius $x$. The symbol $\equiv$ indicates that no relators are required in the above rewriting. The identity in free group $[g_1, g_2]=g_1 g_2 g_1^{-1} g_2^{-1}= \alpha$ allows one to exchange $g_1$ and $g_2$;
\begin{equation}\label{eq_comm_identity_free_1}
g_2 g_1 = \alpha^{-1} g_1 g_2, \; g_2 g_1^{-1} = \alpha^{g_1} g_1^{-1} g_2,
\end{equation}
\begin{equation}\label{eq_comm_identity_free_2}
g_2^{-1} g_1 = \alpha^{g_1 g_2 g_1^{-1}} g_1 g_2^{-1},\;\text{and}\; g_2^{-1} g_1^{-1} = (\alpha^{-1})^{g_2 g_1} g_1^{-1} g_2^{-1}.
\end{equation}

Use identities (\ref{eq_comm_identity_free_1}) and (\ref{eq_comm_identity_free_2}) together with (\ref{con}) to rewrite (\ref{w}) again as
\begin{equation}\label{eq_comm_identity_free_3}
w=f \,f_1 f_2 \cdots f_{\ell}\,g_1^{0} g_2^{0}
\end{equation}
where each $f_i$ is a transposition of $B_{3,x}$. By applying identities (\ref{eq_comm_identity_free_1}) and (\ref{eq_comm_identity_free_2}) consecutively one can show that $g_2^{n_i} g_1^{m_j}$ can be written as $f'  g_1^{m_j}g_2^{n_i}$ and that $f'$ is a product of at most $|n_i m_j|$ transpositions of $B_{3,x}$. This means that the number $\ell$ in expression (\ref{eq_comm_identity_free_3}) is bounded by
$$
(\sum  |m_j|)(\sum  |n_i|) \leq x^2.
$$
We need to bound the length of transpositions $f_i$'s in (\ref{eq_comm_identity_free_3}), $i=1, \cdots, \ell$. The isomorphism $\chi^\ast:\Sigma_{3,x} \cong S_{3x}$ (Theorem \ref{thm_iso_presen_symm}) transforms each $f_i \in \Sigma_{3,x}$ into a transposition $\chi^\ast(f_i)$ of the set $\{1, \cdots, 3x\}$. By Lemma \ref{lm:Humph_fact}, $\chi^\ast(f_i)$ has length at most $O(x^2)$ in $S_{3x}$. Since each generator of $S_{3x}$ becomes a word of $S_{3x}$ of length $\leq 2x$ under $\chi^\ast$, $\chi^\ast(f_i)$ correspond to a word of $\Sigma_{3,x}$ with length at most $O(x^3)$. In all, the element $f \,f_1 f_2 \cdots f_{\ell}$ in the expression (\ref{eq_comm_identity_free_3}) has length at most $O(x^5)$.

Finally we need to calculate the area between two loops, namely one given by $w$ and the other one given by $w'$. Let $B>0$ denote the maximum number of relations that we used in (\ref{eq_comm_identity_free_1}) and (\ref{eq_comm_identity_free_2}). Note that whenever one exchanges $g_1$ and $g_2$ by applying an identity of (\ref{eq_comm_identity_free_1}) and (\ref{eq_comm_identity_free_2}), a transposition is produced. We already calculated the number of those exchanges of two generators, which is bounded by $x^2$.
\end{proof}

The argument if the proof of the above lemma extends to general cases $n\geq 4$.
\begin{lemma}\label{lm:canonical_bound_2}
There is a canonical way so that any null-homotopic word $w \in \ho_n $ with $|w|\leq x$ can be written as a word in $\Sigma_{n,x}$ with length at most $O(x^5)$.
\end{lemma}

\begin{proof} Identities (\ref{con})(\ref{eq_comm_identity_free_1})(\ref{eq_comm_identity_free_2}) holds for $g\in \ho_n$ and for all pair of generators $g_i, g_j$ of $\ho_n$. Moreover the isomorphism between  $\Sigma_{3,x} $ and $S_{3x}$ (Theorem \ref{thm_iso_presen_symm}) allows one to bound the length of each transposition of $B_{n,r}$ by $O(x^3)$. As before we establish an upper bounds: $O(x^5)$ for the length of rewritten word $w'$, and $x^2$ for the area between two loops $w$ and $w'$.
\end{proof}

The following fact can be found in \cite{Humphreys}.
\begin{lemma}\label{lm:Humph_fact}
Let $S_m$ be the finite symmetric group with the Coxeter system. For $\Si \in S_m$, set $r_i(\Si) = |\{j: i<j\; \text{but}\; \Si(i) > \Si(j)\}|$. The length of an element $\Si$ is given by $\sum_1^m r_i(\Si)$. In particular, there exist a unique element with the largest length $\sum_1^{m-1} i$.
\end{lemma}

\begin{lemma}\label{exp_u_k_v_k_3}
Any relator $R$ of $\Sigma_{3,r}$ requires at most $O(e^x)$ relators of $ \ho_n $.
\end{lemma}
\begin{proof}Each generator of $\Sigma_{3,r}$ is an involution. Observe that this can be expressed as appropriate conjugation of $\A^2 =1$ in $\ho_3$. Similarly a braid relator $(\Si_i \Si_{i+1})^3 =1$ can be expressed by conjugation of the relator $(\alpha \alpha^{g_1}\!)^3$. Thus each relator of two types of $\Sigma_{3,r}$ requires area only $1$ in $\ho_n$. For commutation relator of $S_{n,r}$ we need to examine two sequences of words $u_k \!=\!\alpha^{\overline{g}_1 ^k} \alpha^{\overline{g}_2 ^{k}}$ and $v_k\! =\![\alpha ,\alpha^{\overline{g}_1^{k+1}}\!]$. Note that we already check $u_k$ and $u_k$ represent the identity in Lemma \ref{lm_key_identity}. One can use simultaneous induction on $k$ together with the argument in the proof of Lemma \ref{lm_key_identity} to show areas of $u_k$ and $v_k$ are bounded above by $O(e^x)$. The fillings depicted in Figure \ref{diagram} require $O(e^x)$ relators. Since one can produce all commutation relation of $\Sigma_{3,r}$ by taking appropriate conjugation of $v_k$, all the relators of $\Sigma_{3,r}$ require at most $\mathcal{O}(e^x)$ relators of $ \ho_n $.
\end{proof}

\begin{figure}[h]
\includegraphics[width=1.0\textwidth]{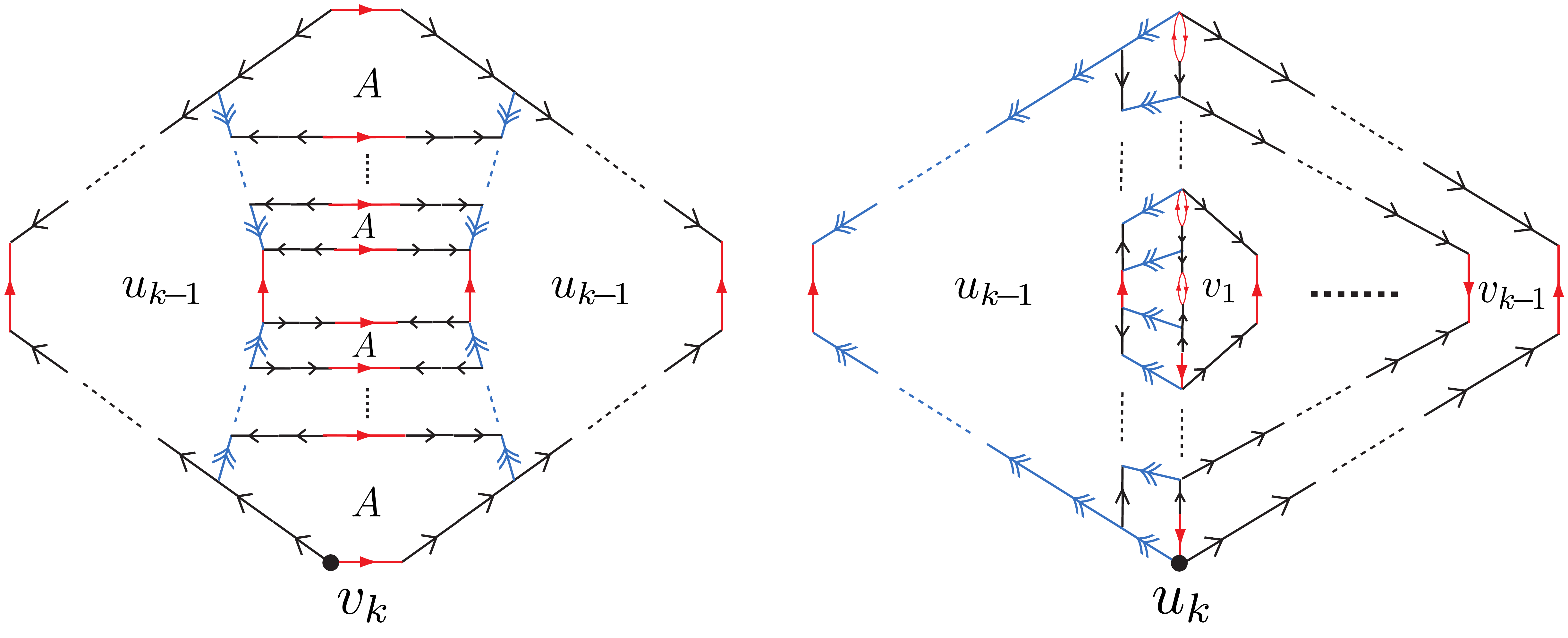}
\caption{Diagrams of $u_k \!=\!\alpha^{\overline{g}_1 ^k} \alpha^{\overline{g}_2 ^{k}}$ and $v_k\! =\![\alpha ,\alpha^{\overline{g}_1^{k+1}}\!]$ $(A:\alpha^{g_1^{-2}g_2} =\alpha^{g_1^{-2}})$}  \label{diagram}
\end{figure}

\textit{Proof of Theorem \ref{thm:exp_bd_n}.} Fix $n \geq 4$. We follow the $3$-step plan that we applied to $\ho_3$. For Step $1$, Lemma \ref{lm:canonical_bound_2} provides the same upper bounds: $O(x^5)$ for the length of rewritten word, and $O(x^2)$ for the gap between two words. The upper bound established in Lemma \ref{lm:cubic} does not depend on the subscript. This means we still have the same cubic upper bound for $\delta_{\Sigma_{n,x}}(x)$. One can show first two types of relations of $\Sigma_{n,x}$ require area $1$ by taking appropriate conjugations as before. Note that there are variations of words $u_k$ and $v_k$ in $\ho_n$. Instead of using $g_1$ and $g_2$ one can use $g_i$ and $g_i$ ($1\leq i < j \leq n-1$) to generate commutation relations of $\Sigma_{n,x}$. Again, exponential upper bounds for those words can be established by simultaneous induction on the length of two sequences of words. In all, we establish exponential upper bound for $\ho_n$.



\begin{remark}
We remark that if words $u_k$ and $v_k$ (and their variations for general cases) have exponential areas then exponential upper bounds for $\delta_{\ho_3}$ ($\delta_{\ho_3}$)is sharp.
\end{remark}



\bibliographystyle{siam}
\bibliography{thesisbib}

\end{document}